\documentclass[11pt]{amsart}

\usepackage[dvipsnames]{xcolor}
\usepackage{tikz}
\usetikzlibrary{positioning}
\usepackage{amsmath}
\usepackage{amssymb}
\usepackage[T1]{fontenc}
\usepackage[utf8]{inputenc}
\usepackage{mathtools}
\usepackage{amsfonts}
\usepackage{fancyhdr}
\usepackage{mathrsfs}
\usepackage{verbatim}
\usetikzlibrary{arrows}
\usepackage{stmaryrd}

\begin{document}

\title{On analytic analogues of quantum groups}
\author{Craig Smith}

\maketitle

\begin{abstract}
In this paper we present a new construction of analytic analogues of quantum groups over non-Archimedean fields and construct braided monoidal categories of their representations. We do this by constructing analytic Nichols algebras and use Majid's double-bosonisation construction to glue them together. We then go on to study the rigidity of these analytic quantum groups as algebra deformations of completed enveloping algebras through bounded cohomology. This provides the first steps towards a $p$-adic Drinfel'd-Kohno Theorem, which should relate this work to Furusho's $p$-adic Drinfel'd associators. Finally, we adapt these constructions to working over Archimedean fields.
\end{abstract}

\tableofcontents

\newtheorem{theorem}{Theorem}[section]
\newtheorem{corollary}[theorem]{Corollary}
\newtheorem{example}[theorem]{Example}
\newtheorem{lem}[theorem]{Lemma}
\newtheorem{obs}[theorem]{Observation}
\newtheorem{ass}[theorem]{Assumptions}
\newtheorem{prop}[theorem]{Proposition}
\theoremstyle{definition}
\newtheorem{defn}[theorem]{Definition}
\newtheorem*{theorem*}{Theorem}

\newtheorem*{Prop3.4}{Proposition 3.4}
\newtheorem*{Cor3.39}{Corollary 3.39}

\newtheorem{rem}[theorem]{Remark}
\numberwithin{equation}{section}

\newenvironment{definition}[1][Definition]{\begin{trivlist}
\item[\hskip \labelsep {\bfseries #1}]}{\end{trivlist}}

\newenvironment{remark}[1][Remark]{\begin{trivlist}
\item[\hskip \labelsep {\bfseries #1}]}{\end{trivlist}}

\newenvironment{notation}[1][Notation]{\begin{trivlist}
\item[\hskip \labelsep {\bfseries #1}]}{\end{trivlist}}

\setcounter{section}{-1}

\section{Introduction}

In 2007, Soibelman gave a rough introduction to $p$-adic analogues of quantum groups in \cite{QpSaQpG} as examples of non-commutative spaces over non-Archimedean fields. Inspired by this, Lyubinin explicitly constructs a $p$-adic quantum hyperenveloping algebra in \cite{pQHAfSL2} in the case of $\mathfrak{sl}_{2}$. His construction involves using Skew-Tate algebras to construct completions of the positive and negative parts of the quantum enveloping algebra. The disadvantage of this construction is that it requires some work to generalise this to arbitrary Kac-Moody Lie algebras. In this paper we present an alternative construction of analytic analogues of quantum groups over non-Archimedean fields that works for any Kac-Moody Lie algebra and construct braided monoidal categories of their representations. With this we hope to exhibit interesting new analytic representations of braid groups. We then go on to use bounded cohomology to study the rigidity of these analytic quantum groups as algebra deformations of completed enveloping algebras. We hope that this will provide the first steps towards a $p$-adic Drinfel'd-Kohno Theorem, relating this work to Furusho's $p$-adic Drinfel'd associators in \cite{pMPatpKZE}.\\

In \cite{ItQG}, Lusztig constructs the positive and negative parts of quantum enveloping algebras as quotients of tensor algebras by the radical of a duality pairing. This is an example of more a general construction, a \emph{Nichols algebra}, discussed in detail in \cite{PHA}. Section 1 of this paper is devoted to presenting the definitions and results required to define analytic analogues of Nichols algebras. All of this is done in the categories of IndBanach spaces over both Archimedean and non-Archimedean fields.\\

Majid's construction in \cite{DBoBG} brings together dually paired braided Hopf algebras $B$ and $C$ with compatible respective right and left actions of a Hopf algebra $H$ to form a new Hopf algebra $U(B,H,C)$, the \emph{double-bosonisation}. The motivation behind this construction is that one can recover the quantum enveloping algebra $U_{q}(\mathfrak{g})$ from $U(B,H,C)$ if $B=U_{q}^{+}(\mathfrak{g})$ and $C=U_{q}^{-}(\mathfrak{g})$ are the respective positive and negative parts of a quantum enveloping algebra and $H=U_{q}^{0}(\mathfrak{g})$ is the Cartan part. Section 2 of this paper recalls and rephrases Majid's double-bosonisation construction in the context of IndBanach spaces, which will allow us to construct analytic analogues quantum enveloping algebras from analytic Nichols algebras in the subsequent sections.\\

In Section 3 we restrict ourselves to working over non-Archimedean fields. We begin Subsection 3.1 by proving the existence of the analytic Nichols algebras defined in Section 1 through two different constructions. The first, given in the proof of Proposition \ref{Radius1NicholsAlgebrasExists}, exhibits the quotient of a completed tensor algebra by a certain universal Hopf ideal as an analytic Nichols algebra. The second, given in Proposition \ref{BilinearFormGivesNicholsAlgebra}, constructs an analytic Nichols algebra as the quotient of a completed tensor algebra by the radical of a duality pairing. In particular this second construction gives the duality pairing between Nichols algebras that allows us to use Majid's double-bosonisation construction. We show in Proposition \ref{NicholsAlgebrasEquivalentDefinition} that these two constructions are equivalent. In Subsection 3.2 we apply these constructions to obtain completions of the positive and negative parts of quantum enveloping algebras and use Majid's double-bosonisation to glue them together into completions of quantum enveloping algebras. We call the resulting IndBanach Hopf algebras \emph{analytic quantum groups}.\\

Unfortunately, we see in Subsection 3.3 that the R-matrix of $U_{q}(\mathfrak{g})$ still does not converge in any of our analytic quantum groups. Nonetheless, in Subsection 3.4 we use an alternate description of our analytic quantum groups as quotients of a Drinfel'd doubles to obtain a braided monoidal category of representations analogous to the BGG category $\mathcal{O}$. We then present an example in the case of $\mathfrak{g}=\mathfrak{sl}_{2}$ of such a braided representation with no highest weight vectors. The further study of these braided representations should produce interesting new examples of braid group representation on Banach spaces and may give a new context in which some special analytic functions, such as $p$-adic multiple polylogarithms, naturally arise.\\

The classical rigidity results of Chevalley, Eilenberg and Cartan from the 1940s assert that there are no non-trivial formal deformations (as an algebra) of the universal enveloping algebra of a semisimple Lie algebra $\mathfrak{g}$. The proof relies on the vanishing of certain Lie algebra cohomology groups. In Theorems \ref{Rigidity1} and \ref{Rigidity2} of Subsection 3.5 we prove an analogous result that, provided an unproven bounded cohomology vanishing result holds, any algebra deformation of a completed enveloping algebra is isomorphic to the trival one. In particular this implies Corollary \ref{RigidityApplied} that asserts that, modulo a bounded cohomology vanishing result, our analytic quantum groups are isomorphic to the trivial algebra deformation of a completed enveloping algebra. Furthermore, both of these isomorphisms are unique up to conjugation. In Subsection 3.6 we highlight the benefits of working over formal powerseries $k \llbracket \hslash \rrbracket$ as opposed to convergent powerseries. In particular this allows us to prove Theorems \ref{Rigidity1b} and \ref{Rigidity2b}, rigidity results analogous to Theorems \ref{Rigidity1} and \ref{Rigidity2} that require weaker assumptions on bounded cohomology.\\

In \cite{QG}, Kassel uses algebraic analogues of the rigidity theorems of Subsections 3.5 and 3.6 to present a proof of the Drinfel'd-Khono Theorem over $\mathbb{C}$. This theorem states that the category of representations of the quantum enveloping algebra is equivalent, as a braided monoidal category, to the category of $U(\mathfrak{g})$-modules with associativity constraint given by the Drinfel'd associator and braiding given by the associated R-matrix. As a result of this, the associated braid group representations are equivalent. This can be interpreted as a statement about the monodromy of the Knizhnik-Zamolodchikov (KZ) equations that govern the Drinfel'd associator. In \cite{pMPatpKZE}, Furusho uses $p$-adic multiple polylogarithms to construct solutions to the $p$-adic KZ equations and a $p$-adic Drinfel'd associator. In the future the author hopes to expand upon the work in Subsections 3.5 and 3.6 to prove a $p$-adic analogue of the Drinfel'd-Khono theorem and to investigate links to Furusho's work.\\

Finally, in Section 4 we adapt these constructions to working over Archimedean fields. We begin by proving the existence of some analytic Nichols algebras and then use Majid's double-bosonisation to form Archimedean analytic quantum groups. We finish by constructing a braided monoidal category of representations as in Subsection 3.4. Again, we hope that the further study of these representations will produce interesting new braid group representations in which we might see 
some special analytic functions arising, such as the quantum dilogarithms that appear in \cite{PRftQDaQM}.

\subsection*{Funding}
This work was supported by the Engineering and Physical Sciences Research Council [EP/M50659X/1].

\subsection*{Acknowledgements}
\thanks{The author would like to thank Kobi Kremnitzer for his expert supervision and continued support throughout this research, without which writing this paper would not have been possible. He would also like to thank Dan Ciubotaru and Shahn Majid for their advice and encouragement.}

\section{Braided IndBanach Hopf algebras, analytic gradings and Nichols algebras}

\subsection{Preliminary definitions}

Let $k$ be a complete valued field throughout. 

\begin{defn}
\label{BanachCategory}
Let $\text{Ban}_{k}$ denote the category of $k$-Banach spaces, each equipped with a specific norm, and bounded linear transformations between them, and let $\text{Ban}_{k}^{\leq 1}$ denote the wide subcategory whose morphisms are bounded linear transformations of norm at most 1, the \emph{contracting category of Banach spaces}. By \emph{wide} we mean that $\text{Ban}_{k}^{\leq 1}$ contains all objects of $\text{Ban}_{k}$. If our field is non-Archimedean then Banach spaces may be defined in two ways, depending on whether we require norms to satisfy the usual triangle inequality or the strong triangle inequality. We shall therefore distinguish between the following two cases:
\begin{itemize}
\item[{\bf (NA)}] $k$ is non-Archimedean and we require all norms in $\text{Ban}_{k}$ to satisfy the strong triangle inequality; and
\item[{\bf (A)}] $k$ is not necessarily non-Archimedean and we only require norms in $\text{Ban}_{k}$ to satisfy the weak triangle inequality.
\end{itemize}
\end{defn}

\begin{defn}
\label{Contracting(Co)Products}
Let $(V_{i})_{i \in I}$ be a family of Banach spaces. Let us define the \emph{contracting product} of this family as the Banach space
$$\prod\nolimits_{i \in I}^{\leq 1}V_{i}=\{(v_{i})_{i \in I} \in \times_{i \in I} V_{i} \mid \text{Sup}_{i \in I} \|v_{i}\| \leq \infty\}$$
with norm $\|(v_{i})\|=\text{Sup}_{i \in I} \|v_{i}\|$ in both the {\bf (NA)} and {\bf (A)} cases, and the \emph{contracting coproduct} as the Banach space
$$\coprod\nolimits_{i \in I}^{\leq 1}V_{i} = \{(v_{i})_{i \in I} \in \times_{i \in I} V_{i} \mid \sum_{i \in I} \|v_{i}\| \leq \infty\}$$
with norm $\|(v_{i})\|=\sum_{i \in I} \|v_{i}\|$ in the case of {\bf (A)} and
$$\coprod\nolimits_{i \in I}^{\leq 1}V_{i} = \{(v_{i})_{i \in I} \in \times_{i \in I} V_{i} \mid \text{lim}_{i \in I} \|v_{i}\| =0\}$$
with norm $\|(v_{i})\|=\text{Sup}_{i \in I} \|v_{i}\|$ in the case of {\bf (NA)}.
\end{defn}

\begin{prop}
The category $\text{Ban}_{k}^{\leq 1}$ has small limits and colimits.
\end{prop}

\begin{proof}
Indeed, it has kernels and cokernels inhereted from $\text{Ban}_{k}$, and it is easy to verify that Definition \ref{Contracting(Co)Products} describes products and coproducts in this category.
\end{proof}

\begin{defn}
For a set $I$, let $\text{Ban}_{k}^{I,\text{bd}}$ be the category whose objects are collections $(V_{i})_{i \in I}$ of Banach spaces $V_{i}$ indexed by $i \in I$ and whose morphisms are uniformly bounded,
$$\text{Hom}((V_{i})_{i \in I},(V'_{i})_{i \in I}):=\prod\nolimits_{i \in I}^{\leq 1}\underline{\text{Hom}}(V_{i},V'_{i}).$$
\end{defn}

\begin{lem}
\label{ContractingUniversalProperty}
$\prod_{i \in I}^{\leq 1}$ and $\coprod_{i \in I}^{\leq 1}$ define functors from $\text{Ban}_{k}^{I,\text{bd}}$ to $\text{Ban}_{k}$. Furthermore, contracting products are right adjoints to the diagonal functors
$$\Delta^{I}:\text{Ban}_{k} \rightarrow \text{Ban}_{k}^{I,\text{bd}}, \quad V \mapsto (V)_{i \in I},$$
and likewise contracting coproducts are left adjoints to $\Delta^{I}$.
\end{lem}

\begin{proof}
This follows from Lemma 2.5 of \cite{TRTfIBS}.
\end{proof}

\begin{defn}
Let us denote by $\text{IndBan}_{k}$ the Ind completion of $\text{Ban}_{k}$. That is, objects of $\text{IndBan}_{k}$ are filtered diagrams of Banach spaces, $X:I \rightarrow \text{Ban}_{k}$, and morphisms are given by
$$\text{Hom}(X,Y)=\text{colim}_{j \in J} \text{lim}_{i \in I} \text{Hom}(X(i),Y(j)).$$
We think of these objects as formal colimits, and use the notation $\text{"colim"}_{i \in I} X(i)$ for the diagram $X$. For a Banach space $V$ we denote by $\text{"}V\text{"}$ the object in $\text{IndBan}_{k}$ represented by the constant diagram at $V$, and often just as $V$ when there is no ambiguity.
\end{defn}

\begin{prop}
The category $\text{IndBan}_{k}$ is a complete and cocomplete quasi-abelian category, and can be given a closed monoidal structure extending that of $\text{Ban}_{k}$ by defining
$$(\text{"colim"}_{i \in I} X_{i})\hat{\otimes}(\text{"colim"}_{j \in J} Y_{j}):= \text{"colim"}_{\substack{i \in I \\ j \in J}} X_{i} \hat{\otimes} Y_{j},$$
$$\underline{\text{Hom}}(\text{"colim"}_{i \in I} X_{i},\text{"colim"}_{j \in J} Y_{j}):= \text{colim}_{j \in J} \text{lim}_{i \in I} \underline{Hom}(X_{i},Y_{j}).$$
\end{prop}

\begin{proof}
Since $\text{Ban}_{k}$ has cokernels and finite direct sums, $\text{IndBan}_{k}$ is cocomplete. An explicit construction of limits in $\text{IndBan}_{k}$ can be found in Section 1.4.1 of \cite{LaACH}. Proposition 2.1.17 of \cite{QACaS} asserts that $\text{IndBan}_{k}$ is quasi-ableian.
\end{proof}

\begin{remark}
For an account of Ind completions see \cite{CaS}, and more on $\text{IndBan}_{k}$ can be found in \cite{TRTfDM}, \cite{SDiBAG}, \cite{NAAGaRAG} and \cite{LaACH} and numerous other excellent sources.
\end{remark}

\begin{defn}
We extend the definition of contracting (co)products to $\text{IndBan}_{k}$ as follows. The contracting product and coproduct functors
$$\prod\nolimits^{\leq 1}_{I},\coprod\nolimits^{\leq 1}_{I}:\text{Ban}_{k}^{I,\text{bd}} \rightarrow \text{Ban}_{k}$$
induce functors from the Ind completion of $\text{Ban}_{k}^{I,\text{bd}}$,
$$\text{IndBan}_{k}^{I,\text{bd}}:=\text{Ind}(\text{Ban}_{k}^{I,\text{bd}}),$$
to $\text{IndBan}_{k}$, which we will continue to denote as $\prod^{\leq 1}_{I}$ and $\coprod^{\leq 1}_{I}$ respectively. There is a faithful diagonal embedding functor $\Delta^{I}:\text{IndBan}_{k} \rightarrow \text{IndBan}_{k}^{I,\text{bd}}$ induced by $\Delta^{I}:\text{Ban}_{k} \rightarrow \text{Ban}_{k}^{I,\text{bd}}$.
\end{defn}

\begin{prop}
With the above definitions, there are adjunctions
$$\text{Hom}(\coprod\nolimits_{I} ^{\leq 1}X_{I},Y) \cong \text{Hom}(X_{I},\Delta^{I}Y)$$
and
$$\text{Hom}(Y,\prod\nolimits_{I} ^{\leq 1}X_{I}) \cong \text{Hom}(\Delta^{I}Y,X_{I})$$
for $X_{I} \in \text{IndBan}_{k}^{I,\text{bd}}$, $Y \in \text{IndBan}_{k}$.
\end{prop}
\begin{proof}
This follows from the adjunction given in Lemma \ref{ContractingUniversalProperty} by taking filtered colimits.
\end{proof}

\subsection{Braided IndBanach Hopf algebras}

\begin{defn}
\label{Braiding}
Let $V$ be an IndBanach space. We say that a morphism $c:V \hat{\otimes} V \rightarrow V \hat{\otimes} V$ is a \emph{pre-braiding} on $V$ if it satisfied the \emph{hexagon axiom}, \emph{i.e.} the diagram
\begin{center}
\begin{tikzpicture}[node distance=6cm, auto]
  \node (A) {$V \hat{\otimes} V \hat{\otimes} V$};
  \node (B) [right=1.3cm of A] {$V \hat{\otimes} V \hat{\otimes} V$};
  \node (C') [below=1cm of A] {};
  \node (C) [left=0.5cm of C'] {$V \hat{\otimes} V \hat{\otimes} V$};
  \node (D') [below=1cm of B] {};
  \node (D) [right=0.5cm of D'] {$V \hat{\otimes} V \hat{\otimes} V$};
  \node (E) [below=1cm of C'] {$V \hat{\otimes} V \hat{\otimes} V$};
  \node (F) [below=1cm of D'] {$V \hat{\otimes} V \hat{\otimes} V$};
  \draw[->] (A) to node {$c \otimes \text{Id}_{V}$} (B);
  \draw[->] (B) to node {$\text{Id}_{V} \otimes c$} (D);
  \draw[->] (D) to node {$c \otimes \text{Id}_{V}$} (F);
  \draw[->] (A) to node [swap]{$\text{Id}_{V} \otimes c$} (C);
  \draw[->] (C) to node [swap]{$c \otimes \text{Id}_{V}$} (E);
  \draw[->] (E) to node [swap]{$\text{Id}_{V} \otimes c$} (F);
\end{tikzpicture}
\end{center}
commutes. We say that the pair $(V,c)$ is a \emph{pre-braided IndBanach space}. If $c$ is an isomorphism then $c$ is a \emph{braiding} and $(V,c)$ is a \emph{braided IndBanach space}.
\end{defn}

\begin{defn}
Let $V$ be a Banach space. We define the \emph{Banach tensor algebra}, $T(V)$, to be the contracting coproduct
$$T(V):=\coprod\nolimits_{n \in \mathbb{Z}_{\geq 0}}^{\leq 1}V^{\hat{\otimes} n}$$
where $V^{\hat{\otimes}0}:=k$. For $r>0$ we will use the notation $T_{r}(V)$ for the Banach space $T(V_{r})$, where $V_{r}$ is the Banach space $V$ with its norm rescaled by $r$. For $r \leq r'$ there is a natural map $T_{r'}(V) \rightarrow T_{r}(V)$, and for all $\rho \geq 0$ we will denote by $T_{\rho}(V)^{\dagger}$ the colimit $\text{"colim"}_{r>\rho}T_{r}(V)$ of this system. We will call this the \emph{dagger tensor algebra} or \emph{overconvergent tensor algebra} of radius $\rho$.
\end{defn}

\begin{lem}
Given a pre-braiding $c$ of a Banach space $V$, there is an induced map $c_{n,m}:V^{\hat{\otimes} n} \hat{\otimes} V^{\hat{\otimes} m} \rightarrow V^{\hat{\otimes} m} \hat{\otimes} V^{\hat{\otimes} n}$ with $\|c_{n,m}\| \leq \|c\|^{mn}$ for each $n,m \geq 0$ satisfying the commutative diagram below:
\begin{center}
\begin{tikzpicture}[node distance=6cm, auto]
  \node (A) {$V^{\hat{\otimes} l} \hat{\otimes} V^{\hat{\otimes} m} \hat{\otimes} V^{\hat{\otimes} n}$};
  \node (B) [right=1.2cm of A] {$V^{\hat{\otimes} m} \hat{\otimes} V^{\hat{\otimes} l} \hat{\otimes} V^{\hat{\otimes} n}$};
  \node (C') [below=1cm of A] {};
  \node (C) [left=0.2cm of C'] {$V^{\hat{\otimes} l} \hat{\otimes} V^{\hat{\otimes} n} \hat{\otimes} V^{\hat{\otimes} m}$};
  \node (D') [below=1cm of B] {};
  \node (D) [right=0.2cm of D'] {$V^{\hat{\otimes} m} \hat{\otimes} V^{\hat{\otimes} n} \hat{\otimes} V^{\hat{\otimes} l}$};
  \node (E) [below=1cm of C'] {$V^{\hat{\otimes} n} \hat{\otimes} V^{\hat{\otimes} l} \hat{\otimes} V^{\hat{\otimes} m}$};
  \node (F) [below=1cm of D'] {$V^{\hat{\otimes} n} \hat{\otimes} V^{\hat{\otimes} m} \hat{\otimes} V^{\hat{\otimes} l}$};
  \draw[->] (A) to node {$c_{l,m} \otimes \text{Id}_{V}$} (B);
  \draw[->] (B) to node {$\text{Id}_{V} \otimes c_{l,n}$} (D);
  \draw[->] (D) to node {$c_{m,n} \otimes \text{Id}_{V}$} (F);
  \draw[->] (A) to node [swap]{$\text{Id}_{V} \otimes c_{m,n}$} (C);
  \draw[->] (C) to node [swap]{$c_{l,n} \otimes \text{Id}_{V}$} (E);
  \draw[->] (E) to node [swap]{$\text{Id}_{V} \otimes c_{l,m}$} (F);
\end{tikzpicture}
\end{center}
Hence if $\|c\| \leq 1$ then there is an induced pre-braiding $\tilde{c}$ on $T(V)$ with $\|\tilde{c}\| = \|c\|$. Furthermore, $\tilde{c}$ is a braiding if and only if $c$ is an isometry.
\end{lem}

\begin{proof}
Applying successively $\text{Id}_{V}^{\otimes n-i} \otimes c \otimes \text{Id}_{V}^{\otimes i-1}$ for $i=1, \ldots, n$ we obtain a map $c_{n}:V^{\hat{\otimes} n} \hat{\otimes} V \rightarrow V \hat{\otimes} V^{\hat{\otimes} n}$ with $\|c_{n}\| \leq \|c\|^{n}$. Then successive applications of $\text{Id}_{V}^{\otimes i-1} \otimes c_{n} \otimes \text{Id}_{V}^{\otimes m-i}$ for $i=1, \ldots, m$ gives a map $c_{n,m}:V^{\hat{\otimes} n} \hat{\otimes} V^{\hat{\otimes} m} \rightarrow V^{\hat{\otimes} m} \hat{\otimes} V^{\hat{\otimes} n}$ with $\|c_{n,m}\| \leq \|c_{n}\|^{m}\leq \|c\|^{nm}$. The commutativity of the given diagram follows from repeated applications of the hexagon axiom from Definition \ref{Braiding}.
\end{proof}

\begin{defn}
\label{BraidedBialgebra}
A \emph{pre-braided IndBanach bialgebra} is an IndBanach space $A$ with both the structure of an algebra, $(A, \mu, \eta)$, and a coalgebra, $(A, \Delta, \varepsilon)$, and equipped with a pre-braiding $c$ on $A$ such that
$$c \circ (\eta \otimes \text{Id}) = \text{Id} \otimes \eta, \quad c \circ (\text{Id} \otimes \eta) = \eta \otimes \text{Id},$$
$$(\text{Id} \otimes \mu)(c \otimes \text{Id})(\text{Id} \otimes c)=c \circ (\mu \otimes \text{Id}), \quad (\mu \otimes \text{Id})(\text{Id} \otimes c)(c \otimes \text{Id})=c \circ (\text{Id} \otimes \mu),$$
$$(\varepsilon \otimes \text{Id}) \circ c = \text{Id} \otimes \varepsilon, \quad (\text{Id} \otimes \varepsilon) \circ c = \varepsilon \otimes \text{Id},$$
$$(c \otimes \text{Id})(\text{Id} \otimes c)(\Delta \otimes \text{Id})=(\text{Id} \otimes \Delta) \circ c, \quad (\text{Id} \otimes c)(c \otimes \text{Id})(\text{Id} \otimes \Delta)=(\Delta \otimes \text{Id}) \circ c,$$
the diagram
\begin{center}
\begin{tikzpicture}[node distance=6cm, auto]
  \node (A) {$A \hat{\otimes} A$};
  \node (B) [below=2.045cm of A] {$A$};
  \node (C) [right=2cm of A] {$A \hat{\otimes} A \hat{\otimes} A \hat{\otimes} A$};
  \node (D) [below=2cm of C] {$A \hat{\otimes} A$};
  \node (E) [below=0.75cm of C] {$A \hat{\otimes} A \hat{\otimes} A \hat{\otimes} A$};
  \draw[->] (A) to node [swap]{$\mu$} (B);
  \draw[->] (C) to node {$\text{Id} \otimes c \otimes \text{Id}$} (E);
  \draw[->] (A) to node {$\Delta \otimes \Delta$} (C);
  \draw[->] (E) to node {$\mu \otimes \mu$} (D);
  \draw[->] (B) to node {$\Delta$} (D);
\end{tikzpicture}
\end{center}
commutes and
$$\Delta \circ \eta = \eta \otimes \eta, \quad \varepsilon \circ \eta = \text{Id}_{k}, \quad \varepsilon \circ \mu = \varepsilon \circ \varepsilon.$$
If $c$ is an isomorphism then $A$ is a \emph{braided IndBanach} \emph{bialgebra}.
\end{defn}

\begin{remark}
For any IndBanach algebra $A$ with a (pre-) braiding $c$, we may define morphisms
$$A \hat{\otimes} A \hat{\otimes} A \hat{\otimes} A \overset{\text{Id} \otimes c \otimes \text{Id}}{\xrightarrow{\hspace*{1cm}}} A \hat{\otimes} A \hat{\otimes} A \hat{\otimes} A \overset{\mu \otimes \mu}{\longrightarrow} A \hat{\otimes} A,$$
$$A \hat{\otimes} A \overset{\Delta \otimes \Delta}{\longrightarrow} A \hat{\otimes} A \hat{\otimes} A \hat{\otimes} A \overset{\text{Id} \otimes c \otimes \text{Id}}{\xrightarrow{\hspace*{1cm}}} A \hat{\otimes} A \hat{\otimes} A \hat{\otimes} A.$$
The first four relations in Definition \ref{BraidedBialgebra} are equivalent to the first of these maps making $A \hat{\otimes} A$ an associative algebra with unit $\eta \otimes \eta$. The next four relations are equivalent to the second of these maps making $A \hat{\otimes} A$ an associative coalgebra with counit $\varepsilon \otimes \varepsilon$. Then the diagram and final three relations are equivalent to $\Delta$ and $\varepsilon$ being algebra homomorphisms, or equivalently $\mu$ and $\eta$ being coalgebra homomorphisms.
\end{remark}

\begin{prop}
\label{TensorBialgebra}
Assume {\bf (NA)}. For $V$ a Banach space, $T_{r}(V)$ is a Banach algebra with multiplication $\mu$ induced by the maps $V_{r}^{\hat{\otimes}n} \hat{\otimes} V_{r}^{\hat{\otimes}m} \overset{\sim}{\longrightarrow} V_{r}^{\hat{\otimes}n+m}$. Given a pre-braiding $c$ on $V$ with $\|c\| \leq 1$, $T_{r}(V)$ is a pre-braided Banach bialgebra with the comultiplication $\Delta$ uniquely determined by $\Delta(v)=1 \otimes v + v \otimes 1$ for all $v \in V_{r}=V_{r}^{\hat{\otimes} 1} \subset T_{r}(V)$ and the pre-braiding $\tilde{c}$ induced by $c$. Furthermore, $\Delta$ has norm at most 1.
\end{prop}

\begin{proof}
The fact that $T_{r}(V)$ forms a Banach algebra is clear from construction, with unit $k \overset{\sim}{\longrightarrow} V^{\hat{\otimes} 0}$. We have a map $\Delta_{1}:V_{r}^{\hat{\otimes} 1} \rightarrow T_{r}(V) \hat{\otimes} T_{r}(V)$, given by $v \mapsto 1 \otimes v + v \otimes 1$, with $\|\Delta_{1}\| \leq 1$. Suppose we have a map $\Delta_{n}:V_{r}^{\hat{\otimes} n} \rightarrow T_{r}(V) \hat{\otimes} T_{r}(V)$ with $\|\Delta_{n}\| \leq \|c\|^{n-1}$. Then we define $\Delta_{n+1}$ as the composition
$$
\begin{array}{rcl}
V_{r}^{\hat{\otimes}n+1}= V_{r}^{\hat{\otimes}n} \hat{\otimes} V_{r} &\overset{\Delta_{n} \otimes \Delta_{1}}{\xrightarrow{\hspace*{1cm}}}& T_{r}(V) \hat{\otimes} T_{r}(V) \hat{\otimes} T_{r}(V) \hat{\otimes} T_{r}(V)\\
&\overset{\text{Id} \otimes \tilde{c} \otimes \text{Id}}{\xrightarrow{\hspace*{1cm}}}& T_{r}(V) \hat{\otimes} T_{r}(V) \hat{\otimes} T_{r}(V) \hat{\otimes} T_{r}(V)\\
&\overset{\mu \otimes \mu}{\xrightarrow{\hspace*{1cm}}}& T_{r}(V) \hat{\otimes} T_{r}(V).
\end{array}
$$
Then since the multiplication on $T_{r}(V)$ is of norm 1, we have $\|\Delta_{n+1}\| \leq \|\Delta_{n}\| \cdot \|\Delta_{1}\| \cdot \|\tilde{c}\| \leq \|c\|^{n}$.
\end{proof}

\begin{prop}
\label{ArchimedeanTensorBialgebra}
Assume {\bf (A)}. For $V$ a Banach space, $T_{r}(V)$ is a Banach algebra with multiplication $\mu$ induced by the maps $V_{r}^{\hat{\otimes}n} \hat{\otimes} V_{r}^{\hat{\otimes}m} \overset{\sim}{\longrightarrow} V_{r}^{\hat{\otimes}n+m}$. Given a pre-braiding $c$ on $V$  with $\|c\| \leq 1$, $T_{0}(V)^{\dagger}$ is a pre-braided IndBanach bialgebra with comultiplication $\Delta$ whose restriction to $V$ is
$$\text{"colim"}_{r>0}V_{r} \cong V \rightarrow T(V) \hat{\otimes} T(V) \rightarrow T_{0}(V)^{\dagger} \hat{\otimes} T_{0}(V)^{\dagger}, \quad v \mapsto 1 \otimes v + v \otimes 1,$$
for the pre-braiding $\tilde{c}$ induced by $c$.
\end{prop}

\begin{proof}
The given map $V \rightarrow T(V) \hat{\otimes} T(V)$ induces maps $V_{r} \rightarrow T_{\frac{r}{2}}(V) \hat{\otimes} T_{\frac{r}{2}}(V)$ of norm at most $1$. By the same construction as in the proof of the previous proposition, we obtain maps $T_{r}(V) \rightarrow T_{\frac{r}{2}}(V) \hat{\otimes} T_{\frac{r}{2}}(V)$ which induce the desired comultiplication on $T_{0}(V)^{\dagger}$.
\end{proof}

\begin{remark}
Note that $V \mapsto T(V)$ is only functorial on the contracting category $\text{Ban}_{k}^{\leq 1}$. So, in the {\bf (NA)} case, the diagonal embedding $V \rightarrow V \oplus V$ induces the map $T(V) \rightarrow T(V \oplus V) \cong T(V) \hat{\otimes} T(V)$. However, in the {\bf (A)} case, the diagonal embedding is not contracting, thus $T(V)$ does not form a coalgebra.
\end{remark}

\begin{defn}
Let $A$ be a (pre-) braided IndBanach bialgebra. We say that $A$ is a (pre-) braided IndBanach Hopf algebra if the identity on $A$ is convolution invertible. That is, $\text{Id}_{A}$ is invertible with respect to the convolution product $\ast$ on $\text{Hom}(A,A)$,
$$f \ast g:A \overset{\Delta}{\longrightarrow} A \hat{\otimes} A \overset{f \otimes g}{\longrightarrow} A \hat{\otimes} A \overset{\mu}{\longrightarrow} A$$
for $f,g \in \text{Hom}(A,A)$, whose unit is $\eta \circ \varepsilon$. We will call the convolution inverse of $\text{Id}_{A}$ the antipode, and often denote it by $S$ or $S_{A}$.
\end{defn}

\begin{lem}
\label{GradedPiecesOfTensorAlgebraCoideals}
Assume {\bf (NA)}. For a Banach space $V$ with a pre-braiding $c$ on $V$ with $\|c\| \leq 1$, we have that, in $T(V)$, $\Delta(V^{\hat{\otimes} n}) \subset \sum_{i=0}^{n} V^{\hat{\otimes} i} \hat{\otimes} V^{\hat{\otimes} n-i}$, and so in particular
$$\Delta(\coprod\nolimits_{i \leq n}^{\leq 1}V^{\hat{\otimes} i}) \subset (\coprod\nolimits_{i \leq n-1}^{\leq 1}V^{\hat{\otimes} i}) \hat{\otimes} T(V) + T(V) \hat{\otimes} (\coprod\nolimits_{i \leq n-1}^{\leq 1}V^{\hat{\otimes} i})$$
for all $n \geq 0$.
\end{lem}

\begin{proof}
This follows by induction using the proof of Proposition \ref{TensorBialgebra}.
\end{proof}

\begin{prop}
\label{TensorHopfAlgebra}
Assume {\bf (NA)}. Given a Banach space $V$ with a pre-braiding $c$ on $V$ with $\|c\| \leq 1$, $T(V)$ is a pre-braided Banach Hopf algebra.
\end{prop}

\begin{proof}
We proceed as in Takeuchi's proof of Lemma 5.2.10 presented in \cite{HAatAoR} to find a convolution inverse to $\text{Id}_{T(V)}$. Let $\gamma = \eta \circ \varepsilon - \text{Id}_{T(V)}$. Then $\gamma|_{V^{\hat{\otimes} 0}} =0$. Now suppose that $\gamma^{\ast n}|_{\coprod_{i \leq n-1}^{\leq 1}V^{\hat{\otimes} i}} =0$ for some $n \geq 1$ and let $x \in \coprod_{i \leq n}^{\leq 1}V^{\hat{\otimes} i}$. Here we use the notation $\gamma^{\ast n}$ for the $n$-fold product of $\gamma$ under the convolution product. Since
$$\Delta(x) \in (\coprod\nolimits_{i \leq n-1}^{\leq 1}V^{\hat{\otimes} i}) \hat{\otimes} T(V) + T(V) \hat{\otimes} (\coprod\nolimits_{i \leq n-1}^{\leq 1}V^{\hat{\otimes} i}),$$
by Lemma \ref{GradedPiecesOfTensorAlgebraCoideals}, and since $\gamma^{\ast (n+1)}=\gamma^{\ast n} \ast \gamma = \gamma \ast \gamma^{\ast n}$ we see that $\gamma^{\ast n+1}(x)=0$. Hence, inductively, $\gamma^{\ast n+1}|_{\coprod_{i \leq n}^{\leq 1}V^{\hat{\otimes} i}} =0$ for all $n \geq 0$. It follows that $\sum_{n=0}^{\infty} \gamma^{\ast n}$ is well defined on the direct sum of vector spaces $\bigoplus_{n \in \mathbb{Z}_{\geq 0}}V^{\hat{\otimes} n}$. Furthermore, since $\|c\| \leq 1$ and so $\|\Delta\| \leq 1$ we see that $\| \gamma \| \leq 1$, $\| \gamma^{\ast n} \| \leq 1$ and so $\|\sum_{n=0}^{N} \gamma^{\ast n}\| \leq 1$ for all $N \geq 0$. It then follows that $\sum_{n=0}^{\infty} \gamma^{\ast n}$ converges to a well defined function on $T(V)$, which is convolution inverse to $\text{Id}_{T(V)}$ since
$$\begin{array}{rcl}
\text{Id}_{T(V)}\ast \sum_{n=0}^{\infty} \gamma^{\ast n} &=& (\eta \circ \varepsilon) \ast \sum_{n=0}^{\infty} \gamma^{\ast n} - \gamma \ast \sum_{n=0}^{\infty} \gamma^{\ast n}\\
&=& \sum_{n=0}^{\infty} \gamma^{\ast n} - \sum_{n=1}^{\infty} \gamma^{\ast n}\\
&=& \eta \circ \varepsilon.
\end{array}$$
\end{proof}

\begin{defn}
Assume {\bf (NA)}. Given a (pre-) braided Banach space $(V,c)$ with $\|c\| \leq 1$ and $r>0$ we denote by $T_{r}^{c}(V)$ the (pre-) braided Banach Hopf algebra described in Proposition \ref{TensorBialgebra} and Proposition \ref{TensorHopfAlgebra}, or just $T_{r}(V)$ if the (pre-) braiding is implicit.
\end{defn}

\begin{prop}
\label{TensorHopfAlgebraA}
Assume {\bf (A)}. Given a Banach space $V$ with a braiding $c$ on $V$ with $\|c\| \leq 1$, $T_{0}(V)^{\dagger}$ is a pre-braided Banach Hopf algebra.
\end{prop}

\begin{proof}
We take a slightly different approach to the proof of Proposition \ref{TensorHopfAlgebra}. The linear maps $S_{r}:V_{r} \rightarrow T_{r}(V)$ defined by $v \mapsto -v$ determine a unique algebra homomorphisms $S:T_{r}(V) \rightarrow T_{r}(V)^{\text{op}}$, where $T_{r}(V)^{\text{op}}$ is the opposite algebra whose multiplication is $\mu \circ \tilde{c}$. The compositions
$$T_{r}(V) \overset{\Delta}{\longrightarrow} T_{\frac{r}{2}}(V) \hat{\otimes} T_{\frac{r}{2}}(V) \overset{S_{\frac{r}{2}} \otimes \text{Id}}{\longrightarrow}  T_{\frac{r}{2}}(V) \hat{\otimes} T_{\frac{r}{2}}(V) \rightarrow T_{\frac{r}{2}}(V)$$
agree with $\eta \circ \varepsilon$ when restricted to $V$. It is then easy to check that they agree on the subalgebra generated by $V$, which is dense in $T_{r}(V)$. So they agree. Likewise this is true for $\text{Id} \otimes S_{\frac{r}{2}}$ in place of $S_{\frac{r}{2}} \otimes \text{Id}$. Taking colimits we obtain the antipode $S:T_{0}(V)^{\dagger} \rightarrow T_{0}(V)^{\dagger}$.
\end{proof}

\begin{defn}
Assume {\bf (A)}. Given a (pre-) braided Banach space $(V,c)$ with $\|c\| \leq 1$ we denote by $T_{0}^{c}(V)^{\dagger}$ the (pre-) braided Banach Hopf algebra described in Proposition \ref{ArchimedeanTensorBialgebra} and Proposition \ref{TensorHopfAlgebraA}, or just $T_{0}(V)^{\dagger}$ if the (pre-) braiding is implicit.
\end{defn}

\begin{remark}
The main distinction between the cases {\bf (NA)} and {\bf (A)} is that the tensor Hopf algebra can be defined on any radius in the non-Archimedean setting but in the Archimedean setting can only be defined at radius 0. This is not entirely unexpected. Analogously we see that, for a non-Archimedean field $k$, the balls of each radius in $k$ form additive subgroups, however the same cannot be said for Archimedean fields.
\end{remark}

\subsection{Analytic gradings}

\begin{defn}
Let $C$ be an IndBanach bialgebra. We will say that an IndBanach space $V$ is \emph{graded by} $C$ if it is a $C$-comodule. An IndBanach space $V$ with a (pre-) braiding $c$ is a \emph{graded (pre-) braided IndBanach space} (graded by $C$) if $c$ is a morphism of $C\hat{\otimes} C$-comodules. A \emph{graded (pre-) braided IndBanach bialgebra} is a (pre-) braided IndBanach bialgebra $A$ such that $A$ is a $C$-comodule, $\eta$, $\mu$, $\varepsilon$ and $\Delta$ are $C$-comodule homomorphisms, and $c$ is a $C \hat{\otimes} C$-comodule homomorphism. If, in addition, the identity on $A$ has a convolution inverse $S$ that is a $C$-comodule homomorphism then $A$ is a \emph{graded (pre-) braided IndBanach Hopf algebra}.
\end{defn}

The results in Sections 4.2 and 4.3 of \cite{TRTfIBS} justify our definition of grading above.

\begin{defn}
\label{FunctionsOnOpenDisk}
Let $k\{\underline{t}\}=k\{t_{1},\ldots,t_{N}\}$ be the bialgebra $\coprod_{\underline{n} \in \mathbb{N}^{N}}^{\leq 1} k \cdot \underline{t}^{\underline{n}}$ where the comultiplication maps $\underline{t}^{\underline{n}} \mapsto \underline{t}^{\underline{n}} \otimes \underline{t}^{\underline{n}}$ and the counit is $\underline{t}^{\underline{n}} \mapsto 1$, and the multiplication maps $\underline{t}^{\underline{m}} \otimes \underline{t}^{\underline{n}} \mapsto  \underline{t}^{\underline{m}+ \underline{n}}$ with unit $\underline{t}^{\underline{0}}$.
\end{defn}

\begin{lem}
IndBanach spaces graded by $k\{\underline{t}\}$ are of the form $\coprod_{\mathbb{N}^{N}}^{\leq 1}M$ for $M$ in $\text{IndBan}_{k}^{\mathbb{N}^{N},\text{bd}}$. The space of morphisms that respect the grading is
$$\underline{\text{Hom}}_{k\{\underline{t}\}}(\coprod\nolimits_{\mathbb{N}^{N}}^{\leq 1}M,\coprod\nolimits_{\mathbb{N}^{N}}^{\leq 1}M')= \text{colim}_{j \in J} \text{lim}_{i \in I} \prod\nolimits_{\underline{n} \in \mathbb{N}^{N}}^{\leq 1}\underline{\text{Hom}}(M_{\underline{n}}(i),M'_{\underline{n}}(j))$$
where $M=\text{"colim"}_{i \in I}(M_{\underline{n}}(i))_{\underline{n} \in \mathbb{N}^{N}}$ and $M'=\text{"colim"}_{j \in J}(M'_{\underline{n}}(j))_{\underline{n} \in \mathbb{N}^{N}}$ in $\text{IndBan}_{k}^{\mathbb{N}^{N},\text{bd}}$.
\end{lem}

\begin{proof}
This is Proposition 4.3 of \cite{TRTfIBS}.
\end{proof}

\begin{defn}
Let $k\{\underline{t}\}^{\dagger}:=\text{"colim"}_{\underline{r} > 1} k\{\underline{t}/ \underline{r}\}$, where the colimit is taken over all polyradii $\underline{r}=(r_{1},..,r_{N})$ with $1 < r_{i}$ and 
$$k\{\underline{t}/ \underline{r}\}=\coprod\nolimits_{\underline{n} \in \mathbb{N}^{N}}^{\leq 1} k_{\underline{r}^{\underline{n}}} \cdot t^{\underline{n}}=\left\lbrace \ \sum a_{\underline{n}} \underline{t}^{\underline{n}} \ \middle| \ \|a_{\underline{n}}\| \underline{r}^{\underline{n}} \rightarrow 0 \ \right\rbrace.$$
The algebra structure given as in Definition \ref{FunctionsOnOpenDisk} on each $k\{\underline{t}/ \underline{r}\}$ makes $k\{\underline{t}\}^{\dagger}$ an IndBanach algebra. Furthermore, the maps $k\{\underline{t}/ \underline{r}^{2}\} \rightarrow k\{\underline{t}/\underline{r}\} \hat{\otimes} k\{\underline{t}/ \underline{r}\}$, $\underline{t}^{\underline{n}} \mapsto \underline{t}^{\underline{n}} \otimes \underline{t}^{\underline{n}}$, and $k\{\underline{t}/\underline{r}\} \rightarrow k$, $\underline{t}^{\underline{n}} \mapsto 1$, induce an IndBanach bialgebra structure on $k\{\underline{t}\}^{\dagger}$. Likewise we define the IndBanach bialgebra $k\{\underline{t}/0\}^{\dagger}:=\text{"colim"}_{\underline{r} > 0} k\{\underline{t}/\underline{r}\}$.
\end{defn}

\begin{lem}
IndBanach spaces graded by $k\{\underline{t}\}^{\dagger}$ are of the form
$$M=\text{"colim"}_{\substack{\underline{r} > 1 \\ i \in I}}\coprod\nolimits_{\underline{n} \in \mathbb{N}^{N}}^{\leq 1} M(\underline{n},i)_{\underline{r}^{\underline{n}}}$$
for $\text{"colim"}_{i \in I} (M(\underline{n},i))_{\underline{n} \in \mathbb{N}^{N}}$ in $\text{IndBan}_{k}^{\mathbb{N}^{N},\text{bd}}$. The space of morphisms that respect the grading is
$$\underline{\text{Hom}}_{k\{\underline{t}\}^{\dagger}}(M,M')= \text{colim}_{j \in J} \text{lim}_{i \in I}\text{lim}_{\underline{r}<1}\prod\nolimits_{\underline{n} \in \mathbb{N}^{n}}^{\leq 1} \text{Hom}(M(\underline{n},i),M'(\underline{n},j))_{\underline{r}^{\underline{n}}},$$
for
$$M=\text{"colim"}_{\underline{r} > 1, i \in I}\coprod\nolimits_{\underline{n} \in \mathbb{N}^{N}}^{\leq 1} M(\underline{n},i)_{\underline{r}^{n}}$$
and
$$M'=\text{"colim"}_{\underline{r} > 1, j \in J}\coprod\nolimits_{\underline{n} \in \mathbb{N}^{N}}^{\leq 1} M'(\underline{n},j)_{\underline{r}^{n}}.$$
Similarly, IndBanach spaces graded by $k\{\underline{t}/0\}^{\dagger}$ are of the form
$$M=\text{"colim"}_{\substack{\underline{r} > 0 \\ i \in I}}\coprod\nolimits_{\underline{n} \in \mathbb{N}^{N}}^{\leq 1} M(\underline{n},i)_{\underline{r}^{\underline{n}}}$$
for $\text{"colim"}_{i \in I} (M(\underline{n},i))_{\underline{n} \in \mathbb{N}^{N}}$ in $\text{IndBan}_{k}^{\mathbb{N}^{N},\text{bd}}$. The space of morphisms that respect the grading is
$$\underline{\text{Hom}}_{k\{\underline{t}/0\}^{\dagger}}(M,M')= \text{colim}_{j \in J} \text{lim}_{i \in I}\text{lim}_{\underline{r}>0}\prod\nolimits_{\underline{n} \in \mathbb{N}^{n}}^{\leq 1} \text{Hom}(M(\underline{n},i),M'(\underline{n},j))_{\underline{r}^{\underline{n}}},$$
for
$$M=\text{"colim"}_{\underline{r} > 0, i \in I}\coprod\nolimits_{\underline{n} \in \mathbb{N}^{N}}^{\leq 1} M(\underline{n},i)_{\underline{r}^{n}}$$
and
$$M'=\text{"colim"}_{\underline{r} > 0, j \in J}\coprod\nolimits_{\underline{n} \in \mathbb{N}^{N}}^{\leq 1} M'(\underline{n},j)_{\underline{r}^{n}}.$$
\end{lem}

\begin{proof}
This is Proposition 4.8 and Proposition 4.10 of \cite{TRTfIBS}.
\end{proof}

\begin{defn}
We say that an IndBanach space is \emph{analytically $\mathbb{N}^{N}$-graded}, or just \emph{analytically graded} if $N$ is implicit, if it is graded over $k\{\underline{t}\}$. Likewise, we say that an IndBanach space is \emph{dagger-1 $\mathbb{N}^{N}$-graded}, or just \emph{dagger-1 graded}, (respectively \emph{dagger-0 $\mathbb{N}^{N}$-graded}, or \emph{dagger-0 graded}) if it is graded over $k\{\underline{t}\}^{\dagger}$ (respectively over $k\{\underline{t}/0\}^{\dagger}$). We will occasionally just use the term \emph{dagger graded} when there is no ambiguity.
\end{defn}

\begin{lem}
Assuming {\bf (NA)}, for each $0 < r$ and each pre-braided Banach space $(V,c)$ with $\|c\| \leq 1$ the Banach tensor algebra $T_{r}^{c}(V)$ is naturally an analytically graded pre-braided Hopf algebra. For any $0 < \rho$, $T_{\rho}^{c}(V)^{\dagger}:=\text{"colim"}_{r > \rho} T_{r}^{c}(V)$ is naturally a dagger-1 graded pre-braided Hopf algebra, and $T_{0}^{c}(V)^{\dagger}:=\text{"colim"}_{r > 0} T_{r}^{c}(V)$ is a dagger-0 graded pre-braided Hopf algebra. Likewise, assuming {\bf (A)}, for each pre-braided Banach space $(V,c)$ with $\|c\| \leq 1$ the dagger tensor algebra $T_{0}^{c}(V)^{\dagger}$ is naturally a dagger-0 graded pre-braided Hopf algebra.
\end{lem}

\begin{proof}
We define $T_{r}(V) \rightarrow k\{\frac{t}{s}\} \hat{\otimes} T_{r'}(V)$ by $x \mapsto t^{n} \otimes x$ for $x \in V^{\hat{\otimes} n}$ whenever $r \geq r's$. For $r=r'$, $s=1$, this gives $T_{r}(V)$ an analytic grading for which it is an analytically graded braided Hopf algebra. For fixed $\rho >0$ and each $r>\rho$ there exists $s>1$ such that $r>s\rho$, so we may define a map $T_{\rho}(V)^{\dagger} \rightarrow k\{t\}^{\dagger} \hat{\otimes} T_{\rho}(V)^{\dagger}$. Likewise we can define a map $T_{0}(V)^{\dagger} \rightarrow k\{t/0\}^{\dagger} \hat{\otimes} T_{0}(V)^{\dagger}$. These give $T_{\rho}(V)^{\dagger}$ and $T_{0}(V)^{\dagger}$ dagger-1 and dagger-0 gradings respectively. By construction all of the structure maps for the pre-braided Hopf algebras $T_{r}^{c}(V)$, $T_{\rho}^{c}(V)^{\dagger}$ and $T_{0}^{c}(V)^{\dagger}$ are graded.
\end{proof}

\begin{defn}
Let $C$ be an IndBanach coalgebra. A \emph{generalised element} of $C$ is a morphism $\lambda:k \rightarrow C$. We say that a generalised element is \emph{grouplike} if it is a coalgebra homomorphism. Let $G(C)$ denote the set of grouplike generalised elements of $C$.
\end{defn}

\begin{lem}
Let $C$ be an IndBanach coalgebra. Then $G(C)$ forms a monoid under the composition law where $\lambda \ast \lambda'$ is the composition $k \cong k \hat{\otimes} k \overset{\lambda \otimes \lambda'}{\longrightarrow} C \hat{\otimes} C \rightarrow C$. If $C$ is an IndBanach Hopf algebra then $G(C)$ is a group.
\end{lem}

\begin{proof}
Note that this is just the restriction of the convolution product on $\text{Hom}(k,C)$. As with Hopf algebras over vector spaces, it is clear that the inverse to $\lambda$ is the composition $k \overset{\lambda}{\longrightarrow} C \overset{S}{\longrightarrow} C$.
\end{proof}

\begin{remark}
Note that $G(k\{\underline{t}\}) \cong G(k\{\underline{t}\}^{\dagger}) \cong G(k\{\underline{t}/0\}^{\dagger}) \cong \mathbb{N}^{N}$ as monoids.
\end{remark}

\begin{defn}
Let $C$ be an IndBanach coalgebra and $V$ be an IndBanach space graded over $C$. Given $\lambda\in G(C)$ we define $V(\lambda)=\text{eq}(V \rightrightarrows C \hat{\otimes} V)$ as the equaliser of the given coaction of $C$ on $V$ and the map
$$V \cong k \hat{\otimes} V \overset{\lambda \otimes \text{Id}_{V}}{\longrightarrow} C \hat{\otimes} V.$$
We think of this as the degree $\lambda$ graded piece. If $\eta_{C}$ is the unit of $C$ then we use the notation $V(0):=V(\eta_{C})$.
\end{defn}

\begin{lem}
\label{TensorOfBraidedPieces}
Let $C$ be an IndBanach coalgebra and $(V,c)$ a $C$-graded, pre-braided IndBanach space and let $\lambda, \lambda' \in G(C)$. Assume further that $V$ and $V(\lambda')$ are flat as IndBanach spaces. Then $(V \hat{\otimes} V)(\lambda \otimes \lambda') = V(\lambda) \hat{\otimes} V(\lambda')$ where $\lambda \otimes \lambda'$ is the morphism $k \cong k \hat{\otimes} k \xrightarrow{\lambda \otimes \lambda'} C \otimes C$.
\end{lem}

\begin{proof}
The map $V(\lambda) \hat{\otimes} V(\lambda') \rightarrow V \hat{\otimes} V$ induces a map
$$V(\lambda) \hat{\otimes} V(\lambda') \rightarrow  (V \hat{\otimes} V)(\lambda \otimes \lambda').$$
Suppose we are given a morphism $f: W \rightarrow V \hat{\otimes} V$ such that the compositions
$$W \xrightarrow{f} V \hat{\otimes} V \cong k \hat{\otimes} k \hat{\otimes} V \hat{\otimes} V \xrightarrow{\lambda \otimes \lambda' \otimes \text{Id} \otimes \text{Id}} C \hat{\otimes} C \hat{\otimes} V \hat{\otimes} V$$
and
$$W \xrightarrow{f} V \hat{\otimes} V  \xrightarrow{\Delta_{V} \otimes \Delta_{V}} C \hat{\otimes} V \hat{\otimes} C \hat{\otimes} V \xrightarrow{\text{Id} \otimes \tau \otimes \text{Id}} C \hat{\otimes} C \hat{\otimes} V \hat{\otimes} V$$
agree. Postcomposing both with $\varepsilon \otimes \text{Id}\otimes \text{Id}\otimes \text{Id}$ we see that the compositions
$$W \xrightarrow{f} V \hat{\otimes} V \cong V \hat{\otimes} k \hat{\otimes} V \xrightarrow{\text{Id} \otimes \lambda' \otimes \text{Id}} V \hat{\otimes} C \hat{\otimes} V$$
and
$$W \xrightarrow{f} V \hat{\otimes} V  \xrightarrow{\text{Id} \otimes \Delta_{V}} V \hat{\otimes} C \hat{\otimes} V$$
agree, so we obtain a map from $W$ to the equaliser of $\text{Id} \otimes \Delta_{V}$ and $\text{Id} \otimes \lambda' \otimes \text{Id}$. Since $V$ is flat, this equaliser is $V \hat{\otimes} V(\lambda')$. Thus we have a unique map $f':W \rightarrow V \hat{\otimes}V(\lambda')$ such that the compositions
$$W \xrightarrow{f'} V \hat{\otimes} V(\lambda') \cong k \hat{\otimes} V \hat{\otimes} V(\lambda') \xrightarrow{\lambda \otimes \text{Id} \otimes \text{Id}} C \hat{\otimes} V \hat{\otimes} V(\lambda')$$
and
$$W \xrightarrow{f'} V \hat{\otimes} V(\lambda') \xrightarrow{\Delta_{V} \otimes \text{Id}} C \hat{\otimes} V \hat{\otimes} V(\lambda')$$
agree. Again, since $V(\lambda')$ is assumed to be flat, the equaliser of $\lambda \otimes \text{Id} \otimes \text{Id}$ and $\Delta_{V} \otimes \text{Id}$ is $V(\lambda) \hat{\otimes} V(\lambda')$ and we obtain a unique map $f'':W \rightarrow V(\lambda) \hat{\otimes} V(\lambda')$. This exhibits $V(\lambda) \hat{\otimes} V(\lambda')$ as the equaliser $(V \hat{\otimes} V)(\lambda \otimes \lambda')$ as required.
\end{proof}

\begin{lem}
\label{RestrictedBraiding}
Let $C$ be an IndBanach coalgebra and $(V,c)$ a $C$-graded, (pre-) braided IndBanach space and let $\lambda\in G(C)$. Assume further that $V$ and $V(\lambda)$ are flat as an IndBanach space. Then $c$ restricts to a (pre-) braiding of $C(\lambda)$.
\end{lem}

\begin{proof}
The braiding $c$ restricts to a morphism $(V \hat{\otimes} V)(\lambda \otimes \lambda) \rightarrow (V \hat{\otimes} V){(\lambda \otimes \lambda)}$. By Lemma \ref{TensorOfBraidedPieces}, this gives a braiding $V(\lambda) \hat{\otimes} V(\lambda) \rightarrow V(\lambda) \hat{\otimes} V(\lambda)$.
\end{proof}

\subsection{Analytic and Dagger Nichols algebras}
\

Theorem 4.3 of \cite{PHA} shows that the positive and negative parts of quantum enveloping algebras arise as Nichols algebras in the category of vector spaces. In this section we introduce analytic and dagger analogues of Nichols algebras that will allow us to construct the positive and negative parts of analytic quantum groups in Sections \ref{NAQGSection} and \ref{AQGSection}.

\begin{defn}
Given an IndBanach Hopf algebra $H$, we denote by $P(H)$ the equaliser of the two maps $H \rightarrow H \hat{\otimes} H$, given by the comultiplication $\Delta$ on $H$ and the sum of the maps
$$H \cong k \hat{\otimes} H \xrightarrow{\eta \otimes \text{Id}_{H}} H \hat{\otimes} H \text{ and } H \cong H \hat{\otimes} k \xrightarrow{\text{Id}_{H} \otimes \eta} H \hat{\otimes} H,$$
the \emph{primitive subspace} of $H$.
\end{defn}

\begin{defn}
\label{GeneratedBy}
Given an IndBanach algebra $A$ and an IndBanach space $V$ equipped with a strict monomorphism $V \hookrightarrow A$, we say that $A$ \emph{is generated by} $V$ if, for any diagram
\begin{center}
\begin{tikzpicture}[node distance=6cm, auto]
  \node (A) {$V$};
  \node (B) [right=1cm of A] {$A$};
  \node (C) [below=0.5cm of B] {$A'$};
  \draw[->] (A) to node {} (B);
  \draw[->] (A) to node {} (C);
  \draw[->] (C) to node [swap]{$f$} (B);
\end{tikzpicture}
\end{center}
where $A'$ is an IndBanach algebra and $f$ is a morphism of algebras, $f$ is an epimorphism.
\end{defn}

\begin{lem}
Given an IndBanach algebra $A$ and an IndBanach space $V$ equipped with a strict monomorphism $V \hookrightarrow A$, $A$ is generated by $V$ if and only if the induced map $\coprod_{n \geq 0}V^{\hat{\otimes}n} \rightarrow A$ is an epimorphism.
\end{lem}

\begin{proof}
If $V$ generates $A$ then the induced algebra map $\coprod_{n \geq 0}V^{\hat{\otimes}n} \rightarrow A$ is epic by assumption. Conversely, suppose that we have an epimorphism $\coprod_{n \geq 0}V^{\hat{\otimes}n} \rightarrow A$ and a diagram
\begin{center}
\begin{tikzpicture}[node distance=6cm, auto]
  \node (A) {$V$};
  \node (B) [right=1cm of A] {$A$};
  \node (C) [below=0.5cm of B] {$A'$};
  \draw[->] (A) to node {} (B);
  \draw[->] (A) to node {} (C);
  \draw[->] (C) to node [swap]{$f$} (B);
\end{tikzpicture}
\end{center}
as in Definition \ref{GeneratedBy}. Then for each $n \geq 0$ we obtain a commutative diagram
\begin{center}
\begin{tikzpicture}[node distance=6cm, auto]
  \node (A) {$V^{\hat{\otimes} n}$};
  \node (B) [right=1cm of A] {$A^{\hat{\otimes} n}$};
  \node (C) [below=0.5cm of B] {$(A')^{\hat{\otimes} n}$};
  \node (D) [right=0.7cm of B] {$A$};
  \node (E) [below=0.7cm of D] {$A'$};
  \draw[->] (A) to node {} (B);
  \draw[->] (A) to node {} (C);
  \draw[->] (C) to node [swap]{$f^{\hat{\otimes} n}$} (B);
  \draw[->] (E) to node [swap]{$f$} (D);
  \draw[->] (B) to node {} (D);
  \draw[->] (C) to node {} (E);
\end{tikzpicture}
\end{center}
which induces a diagram
\begin{center}
\begin{tikzpicture}[node distance=6cm, auto]
  \node (A) {$\coprod_{n \geq 0}V^{\hat{\otimes}n}$};
  \node (B) [right=1cm of A] {$A$};
  \node (C) [below=0.5cm of B] {$A'$};
  \draw[->] (A) to node {} (B);
  \draw[->] (A) to node {} (C);
  \draw[->] (C) to node [swap]{$f$} (B);
\end{tikzpicture}
\end{center}
which ensures that $f$ is an epimorphism since the map $\coprod_{n \geq 0}V^{\hat{\otimes}n} \rightarrow A$ is epic.
\end{proof}

\begin{defn}
Let $V$ be a flat IndBanach space with pre-braiding $c$. Fix an IndBanach bialgebra $C$ and a grouplike generalised element $\lambda:k \rightarrow C$. Then a flat braided graded IndBanach Hopf algebra $R$, graded over $C$, is called an \emph{IndBanach Nichols algebra} of $V$ if $R(0) \cong k$, $P(R) = R(\lambda) \cong V$ as a braided IndBanach space, and $R$ is generated by $R(\lambda)$. If $C=k\{t\}$, $\lambda:1 \mapsto t$, we say that $R$ is an \emph{analytic Nichols algebra}, and likewise if $C=k\{t\}^{\dagger}$ (respectively $C=k\{t/0\}^{\dagger}$), $\lambda:1 \mapsto t$, we say that  $R$ is a \emph{dagger-1 Nichols algebra} (respectively a \emph{dagger-0 Nichols algebra}).
\end{defn}

\begin{remark}
We require flatness in the above definition so that the braiding on $R$ automatically restricts to $R(\lambda)$ by Lemma \ref{RestrictedBraiding}. In the {\bf (NA)} case flatness is automatic by Lemma 3.49 of \cite{SDiBAG}.
\end{remark}

We will prove existence of these Nichols algebras in Sections \ref{NANicholsAlegbraSection} and \ref{ANicholsAlegbraSection}, following a discussion of how to form quantum groups using Majid's double-bosonisation construction.

\section{Double-bosonisation}
\label{Analytic Bosonisation}

In \cite{DBoBG}, Majid introduces a construction, \emph{double-bosonisation}, which he uses to reconstruct Lusztig's form of the quantum enveloping algebra $U_{q}(\mathfrak{g})$. We present here an adaptation of this construction to the setting of IndBanach spaces. For more on these ideas, see \cite{AaHAiBC}, a brief review of which is given in \cite{DBoBG}.

\begin{defn}
Let $H$ and $A$ be IndBanach Hopf algebras. A duality pairing is a bilinear form $\langle -,- \rangle:H \hat{\otimes} A \rightarrow k$ such that the following diagrams commute:
\begin{center}
\begin{tikzpicture}[node distance=6cm, auto]
  \node (A) {$H \hat{\otimes} H \hat{\otimes} A$};
  \node (A') [below=0.75cm of A] {$H \hat{\otimes} H \hat{\otimes} A$};
  \node (B) [below=2cm of A] {$H \hat{\otimes} A$};
  \node (C) [right=3cm of A] {$H \hat{\otimes} H \hat{\otimes} A \hat{\otimes} A$};
  \node (D) [below=2.05cm of C] {$k$};
  \node (E) [below=0.75cm of C] {$H \hat{\otimes} A$};
  \draw[->] (A') to node [swap]{$\mu_{H} \otimes \text{Id}$} (B);
  \draw[->] (A) to node [swap]{$\tau \otimes \text{Id}$} (A');
  \draw[->] (C) to node [swap]{$\text{Id} \otimes \langle - , - \rangle \otimes \text{Id}$} (E);
  \draw[->] (A) to node {$\text{Id} \otimes \text{Id} \otimes \Delta_{A}$} (C);
  \draw[->] (E) to node [swap]{$\langle - , - \rangle$} (D);
  \draw[->] (B) to node {$\langle - , - \rangle$} (D);
\end{tikzpicture} \quad
\begin{tikzpicture}[node distance=6cm, auto]
  \node (A) {$H$};
  \node (B) [right=1cm of A] {$H \hat{\otimes} A$};
  \node (C) [below=2cm of B] {$k$};
  \draw[->] (A) to node {$\eta_{A}$} (B);
  \draw[->] (A) to node [swap]{$\varepsilon_{H}$} (C);
  \draw[->] (B) to node {$\langle -,- \rangle$} (C);
\end{tikzpicture}

\begin{tikzpicture}[node distance=6cm, auto]
  \node (A) {$H \hat{\otimes} A \hat{\otimes} A$};
  \node (A') [below=0.75cm of A] {$H \hat{\otimes} A \hat{\otimes} A$};
  \node (B) [below=2cm of A] {$H \hat{\otimes} A$};
  \node (C) [right=3cm of A] {$H \hat{\otimes} H \hat{\otimes} A \hat{\otimes} A$};
  \node (D) [below=2.05cm of C] {$k$};
  \node (E) [below=0.75cm of C] {$H \hat{\otimes} A$};
  \draw[->] (A') to node [swap]{$\text{Id} \otimes \mu_{A}$} (B);
  \draw[->] (A) to node [swap]{$\text{Id} \otimes \tau$} (A');
  \draw[->] (C) to node [swap]{$\text{Id} \otimes \langle - , - \rangle \otimes \text{Id}$} (E);
  \draw[->] (A) to node {$\Delta_{H} \otimes \text{Id} \otimes \text{Id}$} (C);
  \draw[->] (E) to node [swap]{$\langle - , - \rangle$} (D);
  \draw[->] (B) to node {$\langle - , - \rangle$} (D);
\end{tikzpicture} \quad
\begin{tikzpicture}[node distance=6cm, auto]
  \node (A) {$A$};
  \node (B) [right=1cm of A] {$H \hat{\otimes} A$};
  \node (C) [below=2cm of B] {$k$};
  \draw[->] (A) to node {$\eta_{H}$} (B);
  \draw[->] (A) to node [swap]{$\varepsilon_{A}$} (C);
  \draw[->] (B) to node {$\langle -,- \rangle$} (C);
\end{tikzpicture}

\begin{tikzpicture}[node distance=6cm, auto]
  \node (A) {$H \hat{\otimes} A$};
  \node (B) [right=2cm of A] {$H \hat{\otimes} A$};
  \node (C) [below=0.6cm of A] {$H \hat{\otimes} A$};
  \node (D) [below=0.65cm of B] {$k$};
  \draw[->] (A) to node {$\text{Id} \otimes S_{A}$} (B);
  \draw[->] (B) to node {$\langle -,- \rangle$} (D);
  \draw[->] (A) to node [swap]{$S_{H} \otimes \text{Id}$} (C);
  \draw[->] (C) to node [swap]{$\langle -,- \rangle$} (D);
\end{tikzpicture}
\end{center}
If $H$ and $A$ have respective right and left actions $\zeta_{H}$ and $\zeta_{A}$ of an algebra $\mathcal{A}$ then we say that $\langle-,- \rangle$ is $\mathcal{A}$-equivariant if the diagram
\begin{center}
\begin{tikzpicture}[node distance=6cm, auto]
  \node (A) {$H \hat{\otimes} \mathcal{A} \hat{\otimes} A$};
  \node (B) [right=2cm of A] {$H \hat{\otimes} A$};
  \node (C) [below=1cm of A] {$H \hat{\otimes} A$};
  \node (D) [below=1cm of B] {$k$};
  \draw[->] (A) to node {$\zeta_{H}\otimes \text{Id}$} (B);
  \draw[->] (B) to node {$\langle -,- \rangle$} (D);
  \draw[->] (A) to node [swap]{$\text{Id} \otimes \zeta_{A}$} (C);
  \draw[->] (C) to node [swap]{$\langle -,- \rangle$} (D);
\end{tikzpicture}
\end{center}
commutes. Likewise, if $H$ and $A$ have respective left and right coactions $\zeta_{H}$ and $\zeta_{A}$ of a coalgebra $C$ then we say that $\langle-,- \rangle$ is $C$-equivariant if the diagram
\begin{center}
\begin{tikzpicture}[node distance=6cm, auto]
  \node (A) {$H \hat{\otimes} A$};
  \node (B) [right=2cm of A] {$C \hat{\otimes} H \hat{\otimes} A$};
  \node (C) [below=1cm of A] {$H \hat{\otimes} A \hat{\otimes} C$};
  \node (D) [below=1cm of B] {$C$};
  \draw[->] (A) to node {$\zeta_{H}\otimes \text{Id}$} (B);
  \draw[->] (B) to node {$\text{Id} \otimes \langle -,- \rangle$} (D);
  \draw[->] (A) to node [swap]{$\text{Id} \otimes \zeta_{A}$} (C);
  \draw[->] (C) to node [swap]{$\langle -,- \rangle \otimes \text{Id}$} (D);
\end{tikzpicture}
\end{center}
commutes.
\end{defn}

\begin{defn}
Let $H$ and $A$ be IndBanach Hopf algebras with a duality pairing $\langle -,- \rangle:H \hat{\otimes} A \rightarrow k$. Suppose we have a pair of convolution invertible maps $\mathscr{R}$ and $\overline{\mathscr{R}}$ in $\text{Hom}(A,H)$ that are both algebra homomorphisms and anti-coalgebra homomorphisms, such that the following diagrams commute:
\begin{center}

\begin{tikzpicture}[node distance=6cm, auto]
  \node (A) {$A \hat{\otimes} A$};
  \node (B) [below=1.25cm of A] {$H \hat{\otimes} A$};
  \node (C) [right=3cm of A] {$A \hat{\otimes} H$};
  \node (D) [below=1.3cm of C] {$k$};
  \node (E) [below=0.35cm of C] {$H \hat{\otimes} A$};
  \draw[->] (A) to node [swap]{$\overline{\mathscr{R}} \otimes \text{Id}$} (B);
  \draw[->] (C) to node [swap]{$\tau$} (E);
  \draw[->] (A) to node {$\text{Id} \otimes \mathscr{R}^{-1}$} (C);
  \draw[->] (E) to node {} (D);
  \draw[->] (B) to node {$\langle - , - \rangle$} (D);
\end{tikzpicture}
\\
\begin{tikzpicture}[node distance=6cm, auto]
  \node (A) {$H \hat{\otimes} A$};
  \node (B) [below=3.85cm of A] {$H \hat{\otimes} H \hat{\otimes} A$};
  \node (C) [right=5cm of A] {$H \hat{\otimes} H \hat{\otimes} A \hat{\otimes} A \hat{\otimes} A$};
  \node (D) [below=3.9cm of C] {$H$};
  \node (E) [below=0.5cm of C] {$H \hat{\otimes} A \hat{\otimes} H \hat{\otimes} A \hat{\otimes} A$};
  \node (F) [below=0.5cm of E] {$A \hat{\otimes} H \hat{\otimes} A \hat{\otimes} H \hat{\otimes} A$};
  \node (G) [below=0.5cm of F] {$H \hat{\otimes} H \hat{\otimes} H$};
  \draw[->] (A) to node [swap]{$\Delta_{H} \otimes \text{Id}$} (B);
  \draw[->] (C) to node [swap]{$\text{Id} \otimes \tau \otimes \text{Id}\otimes \text{Id}$} (E);
  \draw[->] (A) to node {$\Delta_{H} \otimes ((\text{Id} \otimes \Delta_{A})\circ\Delta_{A})$} (C);
  \draw[->] (E) to node [swap]{$\tau \otimes \tau \otimes \text{Id}$} (F);
  \draw[->] (F) to node [swap]{$\mathscr{R} \otimes \langle -,- \rangle \otimes \text{Id} \otimes \mathscr{R}^{-1}$} (G);
  \draw[->] (G) to node [swap]{$\mu_{H} \circ(\mu_{H} \otimes \text{Id})$} (D);
  \draw[->] (B) to node [swap]{$\text{Id} \otimes \langle -,- \rangle$} (D);
\end{tikzpicture}
\\
\begin{tikzpicture}[node distance=6cm, auto]
  \node (A) {$H \hat{\otimes} A$};
  \node (B) [below=2.95cm of A] {$H \hat{\otimes} H \hat{\otimes} A$};
  \node (C) [right=5cm of A] {$H \hat{\otimes} H \hat{\otimes} A \hat{\otimes} A \hat{\otimes} A$};
  \node (D) [below=3cm of C] {$H$};
  \node (E) [below=1.15cm of A] {$H \hat{\otimes} H \hat{\otimes} A$};
  \node (F) [below=0.5cm of C] {$A \hat{\otimes} H \hat{\otimes} H \hat{\otimes} A \hat{\otimes} A$};
  \node (G) [below=0.6cm of F] {$H \hat{\otimes} H \hat{\otimes} H$};
  \draw[->] (A) to node [swap]{$\Delta_{H} \otimes \text{Id}$} (E);
  \draw[->] (E) to node [swap]{$\tau \otimes \text{Id}$} (B);
  \draw[->] (A) to node {$\Delta_{H} \otimes ((\text{Id} \otimes \Delta_{A})\circ\Delta_{A})$} (C);
  \draw[->] (C) to node [swap]{$((\tau \otimes \text{Id}) \circ (\text{Id} \otimes \tau)) \otimes \text{Id} \otimes \text{Id}$} (F);
  \draw[->] (F) to node [swap]{$\overline{\mathscr{R}} \otimes \text{Id} \otimes \langle -,- \rangle \otimes \overline{\mathscr{R}}^{-1}$} (G);
  \draw[->] (G) to node [swap]{$\mu_{H} \circ(\mu_{H} \otimes \text{Id})$} (D);
  \draw[->] (B) to node [swap]{$\text{Id} \otimes \langle -,- \rangle$} (D);
\end{tikzpicture}
\end{center}
In this case we call $H$ and $A$ a \emph{weakly quasi-triangular dual pair}.
\end{defn}

\begin{defn}
For an IndBanach Hopf algebra $H$, let us denote by $H\text{-Mod}$ and $\text{Mod-}H$ the categories of left and right $H$-modules respectively. Given a dual pair of IndBanach Hopf algebras, $H$ and $A$, let us denote by $A\text{-Comod}$ and $\text{Comod-}A$ the categories of left and right $A$-comodules respectively. There are faithful functors $A\text{-Comod} \rightarrow \text{Mod-}H$ and $\text{Comod-}A \rightarrow H\text{-Mod}$ induced by the duality pairing.
\end{defn}

\begin{prop}
\label{WeaklyQTPairComodulesBraided}
For a weakly quasi-triangular dual pair of IndBanach Hopf algebras, $H$ and $A$, $A\text{-Comod}$ and $\text{Comod-}A$ are braided monoidal, with braidings given by the compositions
$$M \hat{\otimes} M' \longrightarrow A \hat{\otimes} M \hat{\otimes} A \hat{\otimes} M' \xrightarrow{\mathscr{R} \otimes \text{Id} \otimes \text{Id} \otimes \text{Id}} H \hat{\otimes} M \hat{\otimes} A \hat{\otimes} M' \longrightarrow M \hat{\otimes} M' \xrightarrow{\tau} M' \hat{\otimes} M,$$
where the third map is $(\langle -,- \rangle \otimes \text{Id}) \circ (\text{Id} \otimes \tau \otimes \text{Id})$, and
$$N \hat{\otimes} N' \xrightarrow{\tau} N' \otimes N \longrightarrow  N' \hat{\otimes} A \hat{\otimes} N \hat{\otimes} A \xrightarrow{\text{Id} \otimes \mathscr{R} \otimes \text{Id} \otimes \text{Id}} N' \hat{\otimes} H \hat{\otimes} N \hat{\otimes} A \longrightarrow N' \hat{\otimes} N,$$
where the last map is  $(\text{Id} \otimes \langle -,- \rangle) \circ (\text{Id} \otimes \tau \otimes \text{Id})$, for left $A$-comodules $M$ and $M'$ and right $A$-comodules $N$ and $N'$.
\end{prop}

\begin{proof}
This follows from Theorem 1.16 of \cite{AaHAiBC}, using the remark from the preliminary section of \cite{DBoBG} that $A$ is dual quasi-triangular under the composition
$$A \hat{\otimes} A \xrightarrow{\mathscr{R} \otimes \text{Id}} H \hat{\otimes} A \xrightarrow{\langle-,-\rangle} k.$$
\end{proof}

\begin{prop}
For a weakly quasi-triangular dual pair of IndBanach Hopf algebra, $H$ and $A$, and an algebra $B$ in $\text{Comod-}A$ there is an algebra structure on $B \hat{\otimes} H$ with multiplication defined by
$$\begin{array}{rcl}
B \hat{\otimes} H \hat{\otimes} B \hat{\otimes} H & \overset{\text{Id} \otimes \Delta_{H} \otimes \text{Id} \otimes \text{Id}}{\xrightarrow{\hspace*{2cm}}} & B \hat{\otimes} H \hat{\otimes} H \hat{\otimes} B \hat{\otimes} H \\
& \overset{\text{Id} \otimes \text{Id} \otimes \tau \otimes \text{Id}}{\xrightarrow{\hspace*{2cm}}} & B \hat{\otimes} H \hat{\otimes} B \hat{\otimes} H \hat{\otimes} H \\
& \overset{\text{Id} \otimes \zeta_{B} \otimes \text{Id} \otimes \text{Id}}{\xrightarrow{\hspace*{2cm}}} & B \hat{\otimes} B \hat{\otimes} H \hat{\otimes} H \\
& \overset{\mu_{M} \otimes \mu_{H}}{\xrightarrow{\hspace*{2cm}}} & B \hat{\otimes} H. 
\end{array}$$
Here, $\zeta_{B}$ is the left action of $H$ on $B$. Furthermore, if $B$ is a braided IndBanach Hopf algebra then we can give $B \hat{\otimes} H$ a braided IndBanach Hopf algebra structure with comultiplication defined by
$$\begin{array}{rcccl}
B \hat{\otimes} H & \overset{\Delta_{B} \otimes \Delta_{H}}{\xrightarrow{\hspace*{1.3cm}}} & B \hat{\otimes} B \hat{\otimes} H \hat{\otimes} H
& \overset{\text{Id} \otimes \Psi_{B,H} \otimes \text{Id}}{\xrightarrow{\hspace*{2cm}}} & B \hat{\otimes} H \hat{\otimes} B \hat{\otimes} H
\end{array}$$
where $\Psi$ is the braiding on $\text{Comod-}A$. Likewise, for $C$ in $\text{Mod-H}$ there is a (Hopf) algebra structure on $H \hat{\otimes} C$ defined analogously.
\end{prop}

\begin{proof}
This follows from Theorem 2.1 of \cite{CPbBGaB}.
\end{proof}

\begin{defn}
We denote by $B \rtimes H$ the IndBanach (Hopf) algebra on $B \hat{\otimes} H$ as described above, and likewise we denote by $H \ltimes C$ the (Hopf) algebra on $H \hat{\otimes} C$. These are the \emph{bosonisations} of $B$ and $H$ or $H$ and $C$ respectively.
\end{defn}

\begin{lem}
\label{OverlineHopfAlgebras}
Let $H$ and $A$ be a weakly quasi-triangular dual pair of IndBanach Hopf algebras, with maps $\mathscr{R}_{H,A}$, $\overline{\mathscr{R}}_{H,A}$ giving the weakly quasi-triangular structure. Let $C$ be a braided IndBanach Hopf algebra in $\text{Comod-}A$ with invertible antipode $S_{C}$. Then there is a weakly quasi-triangular dual pair $\overline{H}$ and $\overline{A}$ with the same Hopf algebra structures as $H$ and $A$ but with weakly quasi-triangular structure given by $\mathscr{R}_{\overline{H},\overline{A}}=\overline{\mathscr{R}}^{-1}_{{H},{A}}$ and $\overline{\mathscr{R}}_{\overline{H},\overline{A}}=\mathscr{R}^{-1}_{{H},{A}}$. Furthermore, there is a braided IndBanach Hopf algebra $\overline{C}$ in $\text{Comod-}\overline{A}$ with the same algebra structure as $C$ but with the opposite comultiplication $\Delta_{\overline{C}}=\Psi^{-1}_{C,C} \circ \Delta_{C}$, where $\Psi$ is the braiding on $\text{Comod-}A$, and antipode $S_{\overline{C}}=S_{C}^{-1}$.
\end{lem}

\begin{proof}
This follows from Lemma 3.1 in \cite{DBoBG}, which was proven in \cite{AaHAiBC}, using Remark 3.9 of \cite{DBoBG}.
\end{proof}

\begin{prop}
\label{DoubleBosonisationVariation}
Let $H$ and $A$ be a weakly quasi-triangular dual pair of IndBanach Hopf algebras, let $B$ be a braided IndBanach Hopf algebra in $A\text{-Comod}$ and let $C$ be a braided IndBanach Hopf algebra in $\text{Comod-}A$ with respective induced right and left actions $\zeta_{B}$ and $\zeta_{C}$ of $H$. Suppose further that we have a $H$-equivariant duality pairing between $B$ and $C$, $\langle - , - \rangle: B \hat{\otimes} C \rightarrow k$. Then there is an IndBanach Hopf algebra structure on $C \hat{\otimes} H \hat{\otimes} B$ such that the maps $\overline{C} \rtimes \overline{H} \rightarrow C \hat{\otimes} H \hat{\otimes} B$ and $H \ltimes B \rightarrow C \hat{\otimes} H \hat{\otimes} B$ are morphisms of Hopf algebras, and the multiplication restricts to the composition
$$\begin{array}{rcl}
B \hat{\otimes} \overline{C} &\xrightarrow{\hspace*{2.5cm}}& B \hat{\otimes} B \hat{\otimes} B \hat{\otimes} \overline{C} \hat{\otimes} \overline{C} \hat{\otimes} \overline{C}\\
&\overset{\tau_{(1 \, \, 2 \, \, 4)(3 \, \, 6 \, \, 5)}}{\xrightarrow{\hspace*{2.5cm}}}& \overline{C} \hat{\otimes} B \hat{\otimes} \overline{C} \hat{\otimes} B \hat{\otimes} \overline{C} \hat{\otimes} B\\
&\overset{\text{Id} \otimes \text{Id} \otimes \text{Id} \otimes \text{Id} \otimes S_{\overline{C}} \otimes \text{Id}}{\xrightarrow{\hspace*{2.5cm}}}& \overline{C} \hat{\otimes} B \hat{\otimes} \overline{C} \hat{\otimes} B \hat{\otimes} \overline{C} \hat{\otimes} B\\
&\overset{\text{Id} \otimes R_{B,\overline{C}} \otimes R^{-1}_{B,\overline{C}} \otimes \text{Id}}{\xrightarrow{\hspace*{2.5cm}}}& \overline{C} \hat{\otimes} B \hat{\otimes} \overline{C} \hat{\otimes} B \hat{\otimes} \overline{C} \hat{\otimes} B\\
&\overset{\langle -,- \rangle \otimes \text{Id} \otimes \text{Id} \otimes \langle -,- \rangle}{\xrightarrow{\hspace*{2.5cm}}}& \overline{C} \hat{\otimes} B \longrightarrow C \hat{\otimes} H \hat{\otimes} B\\
\end{array}$$
between $B \hookrightarrow C \hat{\otimes} H \hat{\otimes} B$ and $C \hookrightarrow C \hat{\otimes} H \hat{\otimes} B$. Here, the first map is induced by the coproducts $\Delta_{B}$ and $\Delta_{\overline{C}}$, $\tau_{(1 \, \, 2 \, \, 4)(3 \, \, 6 \, \, 5)}$ is a reordering given by the permutation $(1 \, \, 2 \, \, 4)(3 \, \, 6 \, \, 5) \in S_{6}$, and $R_{B,\overline{C}}$ and $R_{B,\overline{C}}^{-1}$ are the respective compositions
$$B \hat{\otimes} \overline{C} \rightarrow A \hat{\otimes} B \hat{\otimes} \overline{C} \hat{\otimes} A \overset{\mathscr{R} \otimes \text{Id} \otimes \text{Id} \otimes \overline{\mathscr{R}}}{\xrightarrow{\hspace*{2.5cm}}} H \hat{\otimes} B \hat{\otimes} \overline{C} \hat{\otimes} H \rightarrow B \hat{\otimes} \overline{C}$$
and
$$B \hat{\otimes} \overline{C} \rightarrow A \hat{\otimes} B \hat{\otimes} \overline{C} \hat{\otimes} A \overset{\mathscr{R}^{-1} \otimes \text{Id} \otimes \text{Id} \otimes \overline{\mathscr{R}}^{-1}}{\xrightarrow{\hspace*{2.5cm}}} H \hat{\otimes} B \hat{\otimes} \overline{C} \hat{\otimes} H \rightarrow B \hat{\otimes} \overline{C}$$
where the first maps are the coactions of $A$ and the last maps are the induced actions of $H$.
\end{prop}

\begin{proof}
This follows from Theorem 3.2, along with Remark 3.9, of \cite{DBoBG}.
\end{proof}

\begin{defn}
We will denote by $U(C,H,B)$ the Banach Hopf algebra $C \hat{\otimes} H \hat{\otimes} B$ as described in Proposition \ref{DoubleBosonisationVariation}, the \emph{double bosonisation} of $B$, $H$ and $C$ over $A$.
\end{defn}

\begin{defn}
Let $H$ be an IndBanach Hopf algebra. An \emph{R-matrix} for $H$ is a convolution invertible generalised element of $H \hat{\otimes} H$, $\mathscr{R}:k \rightarrow H \hat{\otimes} H$, such that $(\Delta \otimes \text{Id}) \circ \mathscr{R} = \mathscr{R}_{13} \ast \mathscr{R}_{23}$, $(\text{Id} \otimes \Delta) \circ \mathscr{R} = \mathscr{R}_{13} \ast \mathscr{R}_{12}$, and $\tau \circ \Delta = \mathscr{R}(\Delta)\mathscr{R}^{-1}$. Here we use the notation $\mathscr{R}_{12}$ for the composition
$$k \overset{\mathscr{R}}{\longrightarrow} H \hat{\otimes} H \cong H \hat{\otimes} H \hat{\otimes} k \overset{\text{Id} \otimes \text{Id} \otimes \eta_{H}}{\xrightarrow{\hspace*{1.2cm}}} H \hat{\otimes} H \hat{\otimes} H$$
and likewise for $\mathscr{R}_{13}$ and $\mathscr{R}_{23}$, and $\mathscr{R}(\Delta)\mathscr{R}^{-1}$ for the composition
$$H \cong k \hat{\otimes} H \hat{\otimes} k \overset{\mathscr{R} \otimes \Delta_{H} \otimes \mathscr{R}^{-1}}{\xrightarrow{\hspace*{1.4cm}}} (H \hat{\otimes} H) \hat{\otimes} (H \hat{\otimes} H) \hat{\otimes} (H \hat{\otimes} H) \rightarrow H \hat{\otimes} H$$
where the last map is the multiplication on $H \hat{\otimes} H$. A \emph{quasi-triangular IndBanach Hopf algebra} is an IndBanach Hopf algebra $H$ equipped with an R-matrix $\mathscr{R}$.
\end{defn}

\begin{lem}
Let $H$ and $A$ be Hopf algebras with a dual pairing $\langle-,-\rangle:H \hat{\otimes} A \rightarrow k$. Suppose further that $H$ is quasi-triangular, with R-matrix $\mathscr{R}':k \rightarrow H \hat{\otimes} H$. Then the maps
$$\mathscr{R}:A \cong k \hat{\otimes} A \xrightarrow{\mathscr{R}' \otimes \text{Id}} H \hat{\otimes} H \hat{\otimes} A \xrightarrow{\text{Id} \otimes \tau} H \hat{\otimes} A \hat{\otimes} H \xrightarrow{\langle -,- \rangle \otimes \text{Id}} k \hat{\otimes} H \cong H$$
and
$$\overline{\mathscr{R}}:A \cong k \hat{\otimes} A \xrightarrow{(\mathscr{R}')^{-1} \otimes \text{Id}} H \hat{\otimes} H \hat{\otimes} A \xrightarrow{\text{Id} \otimes \langle -,- \rangle} H \hat{\otimes} k \cong H$$
induce a weak quasi-triangular structure on the dual pair $H$ and $A$.
\end{lem}

\begin{proof}
This follows from the remarks in the preliminary section of \cite{DBoBG}.
\end{proof}

\begin{prop}
\label{QuasiTriangularBraiding}
For a quasi-triangular IndBanach Hopf algebra $H$, $H\text{-Mod}$ and $\text{Mod-}H$ are braided monoidal. The braiding on $M$ and $M'$ in $H\text{-Mod}$ is given by the composition
$$M \hat{\otimes}M' \cong k \hat{\otimes} M \hat{\otimes}M' \overset{\mathscr{R} \otimes \text{Id} \otimes \text{Id}}{\xrightarrow{\hspace*{1.2cm}}} H \hat{\otimes} H\hat{\otimes} M \hat{\otimes}M' \overset{\tau \circ (\zeta_{M} \otimes \zeta_{M'})\circ(\text{Id} \otimes \tau \otimes \text{Id})}{\xrightarrow{\hspace*{3.2cm}}}M' \hat{\otimes}M,$$
whilst the braiding on $N$ and $N'$ in $\text{Mod-}H$ is given by the composition
$$N \hat{\otimes}N' \overset{\tau}{\rightarrow} N' \hat{\otimes}N \overset{ \text{Id} \otimes \text{Id} \otimes \mathscr{R}}{\xrightarrow{\hspace*{1.4cm}}}  N' \hat{\otimes}N\hat{\otimes}H \hat{\otimes} H \overset{(\zeta_{N'} \otimes \zeta_{N})\circ(\text{Id} \otimes \tau \otimes \text{Id})}{\xrightarrow{\hspace*{2.8cm}}}N' \hat{\otimes}N.$$
\end{prop}

\begin{proof}
This is Theorem 1.10 of \cite{AaHAiBC}.
\end{proof}

\section{Non-Archimedean Analytic quantum groups}
\label{NAQGSection}

We will assume {\bf (NA)} throughout this section. The {\bf (A)} case will be covered in the subsequent section.

\subsection{Constructing Non-Archimedean Nichols algebras}
\label{NANicholsAlegbraSection}

\begin{defn}
Let $V$ be an analytically $\mathbb{N}$-graded Banach space, $V(n):=V(t^{n})$, $V=\coprod_{n \in \mathbb{N}}^{\leq 1} V(n)$. For $W \subset V$ a subspace, not necessarily closed, let $W(n)=V(n) \cap W$. We say that $W$ is homogeneous if
$$\bigoplus\nolimits_{n \in \mathbb{N}}W(n) \subset W \subset \coprod\nolimits_{n \in \mathbb{N}}^{\leq 1} W(n).$$ Note that $W=\coprod_{n \in \mathbb{N}}^{\leq 1} W(n)$ if and only if it is closed and homogeneous.
\end{defn}

\begin{defn}
For $C$ a Banach coalgebra, we say that a closed subspace $I\subset C$ is a coideal if $\Delta(C) \subset \overline{I \hat{\otimes} C + C \hat{\otimes} I}$. Then $C/I$ is again a Banach coalgebra.
\end{defn}

\begin{lem}
\label{CoidealsOfTensorAlgebras}
Let $V$ be a Banach space with pre-braiding $c$ of norm at most 1. Suppose we have closed homogeneous coideals $I \subset J \subset T(V)$. If the induced map $T(V)/I \rightarrow T(V)/J$ is injective on $P(T(V)/I)$ then $I=J$.
\end{lem}

\begin{proof}
We proceed similarly to Lemma 5.3.3 of \cite{HAatAoR}. Let us denote by $R=T(V)/I$, $R'=T(V)/J$, by $R(n)$, $R'(n)$ their respective $n$th graded pieces, and by $R(\leq n)=\bigoplus_{i \leq n}R(i)$ and $R'(\leq n)=\bigoplus_{i \leq n}R'(i)$. Since $I \subset J$ we have a strict graded epimorphism $f:R \rightarrow R'$ which restricts to an isometry $R(1) \rightarrow R'(1)$. Suppose that we know that $f$ restricts also to an isometry $R(\leq n) \rightarrow R'(\leq n)$, and let $x \in R(\leq n+1)$. We know that, by Lemma \ref{GradedPiecesOfTensorAlgebraCoideals},
$$
\begin{array}{rcl}
\Delta(R(n+1)) &\subset& \sum_{i=0}^{n+1} R(i) \hat{\otimes} R(n-i)\\
&\subset& R(n+1) \hat{\otimes} k + k \hat{\otimes} R(n+1) + R(\leq n) \hat{\otimes} R(\leq n),
\end{array}
$$
so $\Delta(x)=y \otimes 1 + 1 \otimes y' + z$ for some $y,y' \in R(n+1)$, $z \in R(\leq n) \hat{\otimes} R(\leq n)$. But then $x-y = (\text{Id} \otimes \varepsilon)\Delta(x)-y= \varepsilon(y)\cdot 1 + (\text{Id} \otimes \varepsilon)(z) \in R(\leq n)$ and likewise $x-y' \in R(\leq n)$. So $\Delta(x)=x \otimes 1 + 1 \otimes x + z'$ for some $z' \in {R(\leq n)} \hat{\otimes} {R(\leq n)}$. If $f(x)=0$ then $(f \otimes f)(z')=0$, but, by assumption and by Lemma 3.49 of \cite{SDiBAG}, $f \otimes f$ is injective on $R(\leq n) \hat{\otimes} R(\leq n)$. So $z'=0$ and $x$ is primitive, hence $x=0$. Thus $f$ is an isometry ${R(\leq n)} \rightarrow {R'(\leq n)}$ since the norms on $R(\leq n)$ and $R'(\leq n)$ are the quotient norms from $\coprod_{i \leq n}^{\leq 1} V^{\hat{\otimes} i}$. Taking contracting colimits over $n$ we see that $f$ is isometric. Hence $I=J$.
\end{proof}

\begin{prop}
\label{Radius1NicholsAlgebrasExists}
Let $V$ be a Banach space with pre-braiding $c$ of norm at most $1$. Then an analytic Nichols algebra of $V$ exists.
\end{prop}

\begin{proof}
Let $\mathcal{I}(V)$ be the set of all homogeneous ideals of $T(V)$ contained in $\coprod_{n \geq 2}^{\leq 1} V^{\hat{\otimes}n}$ that are also coideals. Let $\mathcal{I}'(V)$ be the subset of ideals $I$ in $\mathcal{I}(V)$ for which $\overline{I} \hat{\otimes} T(V) + T(V) \hat{\otimes} \overline{I}$ is preserved by $c$. Let $I(V)$ and $I'(V)$ be the sums of all ideals in $\mathcal{I}(V)$ and $\mathcal{I}'(V)$ respectively, and let $\overline{I}(V)$ and $\overline{I}'(V)$ be their respective closures. Clearly then $\overline{I}(V)$ is a homogeneous ideal contained in $\coprod_{n \geq 2}^{\leq 1} V^{\hat{\otimes}n}$. Also,
$$\Delta(\overline{I}(V)) \subset \overline{\Delta(I(V))} \subset \overline{T(V) \hat{\otimes} \bar{I}(V) + \bar{I}(V) \hat{\otimes} T(V)},$$
so $\overline{I}(V)$ is also a coideal in $T(V)$, and hence is also in $\mathcal{I}(V)$. So $\overline{I}(V)=I(V)$ is closed. Likewise $I'(V)$ is closed. We must check that
$$P(T(V)/I(V))=(T(V)/I(V))(1).$$
This follows as in Lemma 5.3.3 of \cite{PHA} since the closed ideal in $T(V)$ generated by $I(V)$ and
$$\left\lbrace x \in \coprod\nolimits_{n \geq 2}^{\leq 1} V^{\hat{\otimes} n} \middle| \Delta(x) \in x \otimes 1 + 1 \otimes x + I(V) \otimes T(V) + T(V) \otimes I(V)\right\rbrace$$
must be in $\mathcal{I}(V)$. Likewise $P(T(V)/I'(V))=(T(V)/I'(V))(1)$. But $I'(V) \subset I(V)$, so by Lemma \ref{CoidealsOfTensorAlgebras} we have $I'(V) = I(V)$. Hence $c$ descends to a braiding on $T(V)/I(V)$. We then have that $T(V)/I(V)$ is an analytically graded braided IndBanach Hopf algebra with $(T(V)/I(V))(0) \cong k$ and generated by $(T(V)/I(V))(1) \cong V$.
\end{proof}

\begin{defn}
For a Banach space $V$ with pre-braiding $c$ of norm at most $1$ we will denote by $\mathfrak{B}^{c}(V)$, or $\mathfrak{B}(V)$ when the braiding is implicit, the Banach Nichols algebra defined in the proof of Proposition \ref{Radius1NicholsAlgebrasExists}. For $0< r$, let us denote by $\mathfrak{B}_{r}(V)$, or $\mathfrak{B}_{r}^{c}(V)$, the analytically graded pre-braided Banach Hopf algebra $\mathfrak{B}^{c}(V_{r})$. For $0\leq \rho$, let us denote by $\mathfrak{B}_{\rho}(V)^{\dagger}$, or $\mathfrak{B}_{\rho}^{c}(V)^{\dagger}$, the dagger graded pre-braided Banach Hopf algebra $\text{"colim"}_{r > \rho}\mathfrak{B}_{r}^{c}(V)$.
\end{defn}

\begin{prop}
\label{DaggerNicholsAlgebrasInFiniteDimensions}
Let $0< r$, $0\leq \rho$ and let $V$ be a finite dimensional Banach space with pre-braiding $c$ of norm at most $1$. Then $\mathfrak{B}_{r}^{c}(V)$ is an analytic Nichols algebra of $V$ and $\mathfrak{B}_{\rho}^{c}(V)^{\dagger}$ is a dagger Nichols algebra of $V$.
\end{prop}

\begin{proof}
The fact that $\mathfrak{B}_{r}^{c}(V)$ is an analytic Nichols algebra follows from Proposition \ref{Radius1NicholsAlgebrasExists}. If $R=\mathfrak{B}_{\rho}(V)^{\dagger}$ then it follows from Proposition \ref{Radius1NicholsAlgebrasExists} that $R(0)=\text{"colim"}_{r>\rho}k_{r} \cong k$ and $P(R)=R(1)=\text{"colim"}_{r>\rho}V_{r} \cong V$. It remains to check that $R(0)$ generates $R$. This follows since, for each $r>\rho$, the composition $\coprod_{n \geq 0}V^{\hat{\otimes}n} \rightarrow T_{r}(V) \rightarrow \mathfrak{B}_{r}^{c}(V)$ is an epimorphism.
\end{proof}

\begin{prop}
\label{NicholsAlgebraUniversalProperty}
Let $V$ be a Banach space with pre-braiding $c$ of norm at most $1$. Let $R$ be an analytically graded pre-braided Banach Hopf algebra with contracting multiplication (\emph{i.e} of norm at most $1$) such that $R(0)\cong k$, $R(1) \cong V$ as pre-braided Banach spaces, and $R$ is generated as an algebra by $R(1)$. Then there is an epimorphism of analytically graded braided Hopf algebras $\mathfrak{B}_{r}^{c}(V) \rightarrow R$ extending $V \overset{\sim}{\longrightarrow} R(1)$ where $r$ is the norm of this isomorphism.
\end{prop}

\begin{proof}
Since $V \overset{\sim}{\longrightarrow} R(0)$ is of norm $r$, the map $V_{r} \overset{\sim}{\longrightarrow} R(0)$ is of norm 1. Since the multiplication is contracting, we obtain contracting maps $V_{r}^{\hat{\otimes} n} \rightarrow R^{\hat{\otimes} n} \rightarrow R$ which induce an algebra homomorphism $T_{r}(V) \rightarrow R$ through which $V \cong R(0) \rightarrow R$ factors. Since $R(0)$ generates $R$, this morphism is epic. Furthermore, by construction of the analytically graded pre-braided bialgebra structure on $T_{r}(V)$ and since
$$\Delta(R(1)) \subset R(0)\hat{\otimes} R(1) + R(1) \hat{\otimes} R(0)$$
implies that $R(1)\subset P(R)$, the map $T_{r}(V) \rightarrow R$ is a morphism of bialgebras. Hence the kernel of this map is a closed homogeneous ideal and coideal, so is contained in $I(V_{r})$. Hence we obtain an epimorphism $\mathfrak{B}_{r}^{c}(V)=T(V_{r})/I(V_{r}) \rightarrow R$.
\end{proof}

\begin{remark}
Note that the assignment $(V,c) \mapsto \mathfrak{B}_{r}^{c}(V)$ is functorial if we restrict to morphisms $(V,c) \rightarrow (V',c')$ given by maps $V \rightarrow V'$ of norm at most $1$ that respect the braiding. Hence different norms on finite dimensional Banach spaces and different values of $r$ may give non-isomorphic analytic Nichols algebras. This differs from the algebraic case where Nichols algebras always exist uniquely.
\end{remark}

\begin{defn}
\label{BilinearForm}
A \emph{bilinear form} on a braided IndBanach space $(V,c)$ is a morphism $\langle -,- \rangle : V \hat{\otimes} V \rightarrow k$ such that the diagram
\begin{center}
\begin{tikzpicture}[node distance=6cm, auto]
  \node (A) {$V \hat{\otimes} V$};
  \node (B) [right=0.5cm of A] {$V^{\ast} \hat{\otimes} V^{\ast}$};
  \node (C) [below=0.7cm of A] {$V \hat{\otimes} V$};
  \node (D) [below=0.7cm of B] {$V^{\ast} \hat{\otimes} V^{\ast}$};
  \node (E) [right=0.5cm of B] {$(V \hat{\otimes} V)^{\ast}$};
  \node (F) [below=0.62cm of E] {$(V \hat{\otimes} V)^{\ast}$};
  \draw[->] (A) to node {} (B);
  \draw[->] (E) to node {$c^{\ast}$} (F);
  \draw[->] (A) to node [swap]{$c$} (C);
  \draw[->] (C) to node [swap]{} (D);
  \draw[->] (B) to node {} (E);
  \draw[->] (D) to node {} (F);
\end{tikzpicture}
\end{center}
commutes for both of the induced maps $V \rightarrow V^{\ast}$. A bilinear form is \emph{symmetric} if $\langle - , - \rangle \circ \tau = \langle - , - \rangle$, and is \emph{non-degenerate} if the induced maps $V \rightarrow V^{\ast}$ are both injective.
\end{defn}

\begin{lem}
\label{ExtendBilinearForm}
Let $V$ be a Banach space with pre-braiding $c$ of norm at most 1, and suppose we have a non-degenerate symmetric bilinear form $\langle - , - \rangle: V \hat{\otimes} V \rightarrow k$ of norm $C>0$. Then for each $r,s >0$ with $C \leq rs$ there is a unique extension of this bilinear form to a dual pairing of Hopf algebras $T_{r}^{c}(V) \hat{\otimes} T_{s}^{c}(V) \rightarrow k$. Furthermore, the restriction to $V^{\hat{\otimes} n} \hat{\otimes} V^{\hat{\otimes} m}$ is symmetric if $n=m$ and is zero if $n \neq m$.
\end{lem}

\begin{proof}
This is proved as in \cite{ItQG}. We construct this extension as follows. Our bilinear form on $V$ induces a continuous map $V_{r} \rightarrow V_{s}^{\ast}$ of norm $\frac{C}{rs}$, whilst the natural projection $T_{s}^{c}(V) \rightarrow V_{s}$ induces a map $V_{s}^{\ast} \rightarrow T_{s}^{c}(V)^{\ast}$ of norm $1$, and composition gives us a map $V_{r} \rightarrow T_{s}^{c}(V)^{\ast}$ of norm at most $\frac{C}{rs} \leq 1$. The coalgebra structure on $T_{s}^{c}(V)$ induces an algebra structure on $T_{s}^{c}(V)^{\ast}$ whose multiplication is contracting, and we get a unique continuous algebra homomorphism $T_{r}^{c}(V) \rightarrow T_{s}^{c}(V)^{\ast}$ extending $V_{r} \rightarrow T_{s}^{c}(V)^{\ast}$. This gives us our desired bilinear form on $T_{r}^{c}(V) \hat{\otimes} T_{s}^{c}(V)$. For this to be a dual pairing we must also check that the diagrams
\begin{center}
\begin{tikzpicture}[node distance=6cm, auto]
  \node (A) {$T_{r}^{c}(V)$};
  \node (B) [right=2cm of A] {$T_{s}^{c}(V)^{\ast}$};
  \node (C) [below=1cm of A] {$T_{r}^{c}(V) \hat{\otimes} T_{r}^{c}(V)$};
  \node (D) [below=1cm of B] {$(T_{s}^{c}(V) \hat{\otimes} T_{s}^{c}(V))^{\ast}$};
  \draw[->] (A) to node {} (B);
  \draw[->] (B) to node {$\mu^{\ast}$} (D);
  \draw[->] (A) to node [swap]{$\Delta$} (C);
  \draw[->] (C) to node [swap]{} (D);
\end{tikzpicture}
\begin{tikzpicture}[node distance=6cm, auto]
  \node (A) {$T_{r}^{c}(V)$};
  \node (B) [right=0.7cm of A] {$T_{s}^{c}(V)^{\ast}$};
  \node (C) [below=1cm of A] {$k$};
  \node (D) [below=1cm of B] {$k^{\ast}$};
  \draw[->] (A) to node {} (B);
  \draw[->] (B) to node {$\eta^{\ast}$} (D);
  \draw[->] (A) to node [swap]{$\varepsilon$} (C);
  \draw[->] (C) to node {$\sim$} (D);
\end{tikzpicture}
\end{center}
commute. By assumption all of these morphisms are algebra homomorphisms, and so it is enough to check these diagrams commute on $V$, which is trivial. Since the algebra homomorphism $T_{r}^{c}(V) \rightarrow T_{s}^{c}(V)^{\ast}$ is graded we have $\langle V^{\hat{\otimes} n}, V^{\hat{\otimes} m} \rangle=\{0\}$ for $n \neq m$. Since this is a duality pairing, symmetry on $V^{\hat{\otimes} n} \hat{\otimes} V^{\hat{\otimes} n}$ can be reduced to the case where $n=1$ where it is true by assumption.
\end{proof}

\begin{prop}
\label{BilinearFormGivesNicholsAlgebra}
Let $V$ be a Banach space with pre-braiding $c$ of norm at most 1, and suppose we have a non-degenerate symmetric bilinear form $\langle - , - \rangle: V \hat{\otimes} V \rightarrow k$ of norm $C$. Then for each $0< r$, let $I_{r}$ be the radical in $T_{r}^{c}(V)$ of the extension of this bilinear form to $T_{r}^{c}(V)\hat{\otimes} T_{s}^{c}(V)$ for some $s>0$ such that $C\leq rs$. Then $I_{r}$ is a closed ideal and coideal of $T_{r}^{c}(V)$, independent of the choice of $s$, and $P(T_{r}^{c}(V)/I_{r})=V$. Hence $T_{r}^{c}(V)/I_{r}$ is an analytic Nichols algebra of $V$.
\end{prop}

\begin{proof}
The fact that $I_{r}$ is a closed homogeneous ideal and coideal of $T_{r}^{c}(V)$ follows from Lemma \ref{ExtendBilinearForm}. It is independent of the choice of $s$ since the vector subspace $\bigoplus_{n \geq 0} V^{\otimes n}$ is dense in $T_{s}^{c}(V)$ for all $s$, hence
$$I_{r}=\left\lbrace x \in T_{r}^{c}(V) \middle| \langle x,y \rangle = 0 \text{ for all } y \in \bigoplus\nolimits_{n \geq 0} V^{\otimes n}\right\rbrace.$$
Since the bilinear form  on $V$ is non-degenerate, $I_{r} \subset \coprod_{n \geq 2}^{\leq 1} V_{r}^{\hat{\otimes}n}$. Clearly the quotient $T_{r}(V)/I_{r}$ is generated by $V$, and so it remains to check that the subspace of primitive elements is just $V$. Given $x \in T_{r}^{c}(V)$ homogeneous of degree $n \geq 2$ (\emph{i.e.} in $V_{r}^{\hat{\otimes}n}$) such that its image in $T_{r}^{c}(V)/I_{r}$ is primitive, we must have that $\langle x,yy' \rangle = \langle 1,y \rangle \langle x,y' \rangle + \langle x,y \rangle \langle 1,y' \rangle =0$ for all $y,y'$ homogeneous of degree at least $1$. It then follows that $x$ must be in the radical $I_{r}$ since $\langle x,z \rangle=0$ for any $z$ homogeneous of degree at most $1$. By the assumption in Definition \ref{BilinearForm} the diagram
\begin{center}
\begin{tikzpicture}[node distance=6cm, auto]
  \node (A) {$T_{r}(V) \hat{\otimes} T_{r}(V)$};
  \node (B) [right=0.5cm of A] {$(T_{s}(V) \hat{\otimes} T_{s}(V))^{\ast}$};
  \node (C) [below=0.5cm of A] {$T_{r}(V) \hat{\otimes} T_{r}(V)$};
  \node (D) [below=0.5cm of B] {$(T_{s}(V) \hat{\otimes} T_{s}(V))^{\ast}$};
  \draw[->] (A) to node {} (B);
  \draw[->] (B) to node {$\tilde{c}^{\ast}$} (D);
  \draw[->] (A) to node [swap]{$\tilde{c}$} (C);
  \draw[->] (C) to node [swap]{} (D);
\end{tikzpicture}
\end{center}
commutes, and so $I_{r} \hat{\otimes} T_{r}(V) + T_{r}(V) \hat{\otimes} I_{r}$ is preserved by $\tilde{c}$ and the braiding on $T_{r}^{c}(V)$ descends to a braiding of $T_{r}^{c}(V)/I_{r}$. Hence this is an analytic Nichols algebra of $V$.
\end{proof}

\begin{prop}
\label{NicholsAlgebrasEquivalentDefinition}
Let $V$ be a Banach space with pre-braiding $c$ of norm at most $1$, and suppose we have a non-degenerate symmetric bilinear form $\langle - , - \rangle: V \hat{\otimes} V \rightarrow k$. Retaining the notation of Proposition \ref{BilinearFormGivesNicholsAlgebra}, the induced map $\mathfrak{B}_{r}^{c}(V) \rightarrow T_{r}^{c}(V)/I_{r}$ is an isomorphism for each $0< r$. In particular, $I_{r}$ is independent of the choice of bilinear form.
\end{prop}

\begin{proof}
This follows from Lemma \ref{CoidealsOfTensorAlgebras}, noting that $I_{r} \subset I(V_{r})$.
\end{proof}

\begin{prop}
Let $V$ be a Banach space with pre-braiding $c$ of norm at most $1$, and suppose we have a non-degenerate symmetric bilinear form $\langle - , - \rangle: V \hat{\otimes} V \rightarrow k$. Then for each $0 \leq \rho$, $\mathfrak{B}_{\rho}^{c}(V)^{\dagger}$ is a dagger Nichols algebra of $V$.
\end{prop}

\begin{proof}
This follows as in the proof of Proposition \ref{DaggerNicholsAlgebrasInFiniteDimensions}.
\end{proof}

\begin{prop}
\label{WeakClassicalNicholsAlgebrasDense}
Let $V$ be a Banach space with a pre-braiding $c$ of norm at most 1 that restricts to an algebraic braiding $V \otimes V \rightarrow V \otimes V$ of vector spaces, and let $\langle - , - \rangle: V \hat{\otimes} V \rightarrow k$ be a non-degenerate symmetric bilinear form. Then the algebraic Nichols algebra of $V$, as defined in Definition 2.1 of \cite{PHA}, is dense in the Banach space $\mathfrak{B}_{r}^{c}(V)$ for each $r> 0$.
\end{prop}

\begin{proof}
The algebraic Nichols algebra can be constructed as the quotient of the tensor algebra of vector spaces, $\bigoplus_{n \geq 0} V^{\otimes n}$, by the radical of the restriction of the induced bilinear form on $T_{r}(V) \hat{\otimes} T_{s}(V)$ for some sufficiently large $s>0$. The result then follows since $\bigoplus_{n \geq 0} V^{\otimes n}$ is dense in $T_{r}(V)$ for all $r>0$.
\end{proof}

\begin{prop}
\label{ClassicalNicholsAlgebrasDense}
Let $V$ be a Banach space with pre-braiding $c$ of norm at most $1$, and suppose that $R$ is a Banach analytic Nichols algebra of $V$. Then the algebraic Nichols algebra of $V$, as defined in Definition 2.1 of \cite{PHA}, is dense in $R$. Furthermore, if $V$ is finite dimensional, $R(n)$ is isomorphic as a vector space to the $n$th graded piece of the algebraic Nichols algebra.
\end{prop}

\begin{proof}
In the category of vector spaces, there is a unique braided $\mathbb{N}$-graded Hopf algebra structure on $\bigoplus_{n \geq 0} V^{\otimes n}$ for which $V$ is primitive and the braiding restricts to $c$.  The inclusion $V \hookrightarrow R$ induces a morphism of braided graded Hopf algebras $\bigoplus_{n \geq 0} V^{\otimes n} \rightarrow R$. The image of this morphism, which we denote by $R_{0}$, is a braided graded Hopf algebra, generated by $V$, that is dense in $R$. Then $V \subset P(R_{0}) \subset P(R) = V$ and so $P(R_{0})=V$. Likewise $R_{0}(0)=k$ and $R_{0}(1)=V$. Thus $R_{0}$ is an algebraic Nichols algebra. Furthermore, each $R_{0}(n)$ must be dense in $R(n)$. If these are finite dimensional then they must be equal.
\end{proof}

\subsection{Constructing non-Archimedean analytic quantum groups}
\label{Constructingnon-Archimedeananalyticquantumgroups}

The main motivation behind Majid's work in \cite{DBoBG}, which we summarised in Section \ref{Analytic Bosonisation}, is the reconstruction of Lusztig's from of the quantum enveloping algebra. We will use the same technique to construct analytic analogues of these quantum enveloping algebras.
\\

Throughout the following, we will fix an element $q \in k \setminus \{0\}$ of norm 1 that is not a root of unity. We fix the following root datum of a Lie algebra.

\begin{defn}
\label{KacMoodyRootDatum}
Let $\mathfrak{g}$ be the Lie algebra defined by the data of
\begin{itemize}
\item[i)] a free $\mathbb{Z}$-module $\Phi$, the \emph{weight lattice}, a submodule $\Psi \subset \Phi$, the \emph{root lattice}, and a free basis $\{\alpha_{i} \mid i \in I\}$ of $\Psi$, the \emph{simple roots}, indexed over some set $I$;
\item[ii)] a symmetric bilinear form $( \cdot , \cdot ) : \Phi \times \Phi \rightarrow \mathbb{Q}$ such that $( \alpha_{i}, \alpha_{i}) \in 2 \mathbb{N}$ and $( \alpha_{i}, \alpha_{j}) \leq 0$ for $i, j \in I, i \neq j$; and
\item[iii)] \emph{simple coroots} $\lambda_{i} \in \Phi^{\ast}=\text{Hom}_{\mathbb{Z}}(\Phi, \mathbb{Z})$ such that $\lambda_{i}(\alpha) = \frac{2(\alpha_{i}, \alpha)}{(\alpha_{i}, \alpha_{i})}$ for $i \in I, \alpha \in \Phi$.
\end{itemize}
Then $\mathfrak{g}$ is generated over $\mathbb{Q}$ by elements $e_{i}, f_{i}, h_{i}$ for $i \in I$ subject to the relations
$$\begin{array}{cccc}
[h_{i},h_{j}] = 0, & [e_{i},f_{i}] = \delta_{ij} h_{i}, &
[h_{i},e_{j}] = \lambda_{i}(\alpha_{j})e_{j}, & [h_{i},f_{j}] = - \lambda_{i}(\alpha_{j})f_{j} \end{array},$$
and for $i \neq j$,
$$\begin{array}{cc}
(\text{ad}e_{i})^{1-\lambda_{i}(\alpha_{j})}e_{j} = 0, & (\text{ad}f_{i})^{1-\lambda_{i}(\alpha_{j})}f_{j} = 0, \end{array}$$
where ad is the \emph{adjoint map} $(\text{ad}x)(y) = [x,y]$.
\end{defn}

\begin{defn}
\label{CartanPart}
Let $H=\coprod_{\lambda \in \Phi^{\ast}}^{\leq 1} k \cdot K_{\lambda}$ be the Banach group Hopf algebra of $\Phi^{\ast}$ with
$$K_{\lambda} \cdot K_{\lambda'}=K_{\lambda + \lambda'}, \quad \Delta_{H}(K_{\lambda})=K_{\lambda} \otimes K_{\lambda}\quad\text{and}\quad S(K_{\lambda})=K_{-\lambda}.$$
We use the notation
$$t_{i}:=K_{\frac{(\alpha_{i},\alpha_{i})}{2}\lambda_{i}}$$
which we borrow from \cite{OCB}. Let $H'$ be the closed sub-Hopf algebra generated by $\{t_{i} \mid i \in I\}$, $H'=\coprod_{\underline{n} \in \mathbb{Z}^{I}}^{\leq 1} k \cdot \underline{t}^{\underline{n}}$.
\end{defn}

\begin{lem}
\label{HAPairingWeakQuasiTriangular}
There is a duality pairing $H \hat{\otimes} H' \rightarrow k$ defined by
$$ K_{\lambda} \otimes \underline{t}^{\underline{n}} \mapsto q^{\lambda(\sum n_{i} \alpha_{i})}$$
and simultaneous algebra homomorphisms and coalgebra anti-homomorphisms
$$\mathscr{R}:H' \rightarrow H, t_{i} \mapsto t_{i}, \, \, \, \overline{\mathscr{R}}:H' \rightarrow H, t_{i} \mapsto t_{i}^{-1},$$
for $i \in I$, making $H$ and $H'$ a weakly quasi-triangular dual pair.
\end{lem}

\begin{proof}
This is Lemma 4.1 of \cite{DBoBG}.
\end{proof}

\begin{defn}
We will say that a Banach $H$-modules $M$ is a \emph{Banach weight space} of weight $\alpha \in \Phi$ if $K_{\lambda} \cdot m = q^{\lambda(\alpha)}m$ for all $m \in M$. We say that an Banach $H$-module $M$ \emph{decomposes into Banach weight space} if it is a contracting coproduct of Banach weight spaces $M_{\alpha}$ of weights $\alpha \in \Phi$, $M \cong \coprod_{\alpha \in \Phi}^{\leq 1} M_{\alpha}$. The weights of $M$ will be those $\alpha \in \Phi$ such that $M_{\alpha} \neq 0$. We will say that an IndBanach $H$-module $M$ \emph{decomposes locally into Banach weight spaces} if it can be written as a colimit of Banach $H$-modules that decompose into Banach weight spaces. Given a subset $X \subset \Phi$ we denote by $H\text{-Mod}_{X}$ the full subcategory of $H\text{-Mod}$ consisting of modules that decompose locally into Banach weight spaces with weights in $X$.
\end{defn}

\begin{remark}
Lemma \ref{HAPairingWeakQuasiTriangular} above, along with Proposition \ref{WeaklyQTPairComodulesBraided} and Proposition 4.4 of \cite{TRTfIBS}, says precisely that $H\text{-Mod}_{\Psi}$ is braided monoidal. In the next section we shall extend this braiding, under certain conditions, to $H\text{-Mod}_{\Phi}$.
\end{remark}

\begin{defn}
\label{AnalyticQuantumGroupPositivePart}
Let $V=\coprod_{i \in I}^{\leq 1} k \cdot v_{i}$ with basis $\{v_{i} \mid i \in I \}$. This is a left $H'$-comodule with coaction $v_{i} \mapsto t_{i} \otimes v_{i}$. The weakly quasi-triangular structure of $H$ and $H'$ then induces a braiding $c$ on $V$ given by
$$c(v_{i} \otimes v_{j})=q^{(\alpha_{i},\alpha_{j})}v_{j} \otimes v_{i}$$
of norm $\|c\| = 1$. Let $\langle -,- \rangle$ be the non-degenerate bilinear form on $V$ defined by
$$\langle v_{i},v_{j} \rangle = \delta_{i,j}\frac{1}{(q_{i}-q_{i}^{-1})} \quad \text{for} \quad q_{i}=q^{\frac{(\alpha_{i},\alpha_{i})}{2}}.$$
Given $0 < r$ or $0 \leq \rho$ we denote by $\mathbf{f}_{r}^{\text{an}}$ and $\mathbf{f}_{\rho}^{\dagger}$ the Nichols algebras $\mathfrak{B}_{r}^{c}(V)$ and $\mathfrak{B}_{\rho}^{c}(V)^{\dagger}$ respectively. We will also use the notation $U_{q}^{+}(\mathfrak{g})_{r}^{\text{an}}=U_{q}^{-}(\mathfrak{g})_{r}^{\text{an}}=\mathbf{f}_{r}^{\text{an}}$ and $U_{q}^{+}(\mathfrak{g})_{\rho}^{\dagger}=U_{q}^{-}(\mathfrak{g})_{\rho}^{\dagger}=\mathbf{f}_{\rho}^{\dagger}$. Note that these are braided IndBanach Hopf algebras in $H'\text{-Comod}=\text{Comod-}H'$.
\end{defn}

\begin{lem}
\label{PositivePartDense}
For each $0<r$, the positive part of the quantum group, as defined in \cite{ItQG} and \cite{DBoBG}, is dense in the Banach space $\mathbf{f}_{r}^{\text{an}}$.
\end{lem}

\begin{proof}
This follows from Proposition \ref{ClassicalNicholsAlgebrasDense} along with Theorem 4.2 of \cite{PHA}, which is a restatement of the constructions in Chapter 1 of \cite{ItQG}.
\end{proof}

\begin{lem}
Suppose $0< r,s$ such that $1 \leq |q_{i}-q_{i}^{-1}|rs$ for all $i \in I$. Then there is a duality pairing $\mathbf{f}_{s}^{\text{an}} \hat{\otimes} \mathbf{f}_{r}^{\text{an}} \rightarrow k$ as Banach Hopf algebras in $H\text{-Mod}=\text{Mod-}H$ extending $\langle-,-\rangle$ in Definition \ref{AnalyticQuantumGroupPositivePart}. Likewise, for $0 \leq \rho, \sigma$ with $1 \leq |q_{i}-q_{i}^{-1}|\rho \sigma$ for all $i \in I$ there is a duality pairing $\mathbf{f}_{\sigma}^{\dagger} \hat{\otimes} \mathbf{f}_{\rho}^{\dagger} \rightarrow k$.
\end{lem}

\begin{proof}
This follows from Lemma \ref{ExtendBilinearForm}.
\end{proof}

\begin{defn}
\label{AnalyticQuantumGroups}
\label{DaggerQuantumGroups}
For $0< r,s$ with $1 \leq |q_{i}-q_{i}^{-1}|rs$ for all $i \in I$ we denote by $U_{q}(\mathfrak{g})_{r,s}^{\text{an}}$ the \emph{analytic quantum group} $U(\mathbf{f}_{r}^{\text{an}},H,\mathbf{f}_{s}^{\text{an}})$. We will denote by $U_{q}^{\leq 0}(\mathfrak{g})_{r}^{\text{an}}$ and $U_{q}^{\geq 0}(\mathfrak{g})_{r}^{\text{an}}$ the respective sub-Hopf algebras $\mathbf{f}_{r}^{\text{an}} \rtimes H$ and $\overline{H} \ltimes \overline{\mathbf{f}_{s}^{\text{an}}}$ of $U_{q}(\mathfrak{g})_{r,s}^{\text{an}}$. Let us denote by $F_{i}$ the element $v_{i} \otimes 1 \in \mathbf{f}_{r}^{\text{an}} \rtimes H$, and by $E_{i}$ the element $1 \otimes v_{i}\in \overline{H} \ltimes \overline{\mathbf{f}_{s}^{\text{an}}}$ for $i \in I$. For $0 \leq \rho, \sigma$ with $1 \leq |q_{i}-q_{i}^{-1}|\rho \sigma$ for all $i \in I$ we denote by $U_{q}(\mathfrak{g})_{\rho,\sigma}^{\dagger}$ the \emph{dagger quantum group} $U(\mathbf{f}_{\rho}^{\dagger},H,\mathbf{f}_{\sigma}^{\dagger})$. We will denote by $U_{q}^{\leq 0}(\mathfrak{g})_{\rho}^{\dagger}$ and $U_{q}^{\geq 0}(\mathfrak{g})_{\sigma}^{\dagger}$ the respective sub-Hopf algebras $\mathbf{f}_{\rho}^{\dagger} \rtimes H$ and $\overline{H} \ltimes \overline{\mathbf{f}_{\sigma}^{\dagger}}$ of $U_{q}(\mathfrak{g})_{\rho,\sigma}^{\dagger}$.
\end{defn}

\begin{prop}
\label{NAQGGraded}
$U_{q}(\mathfrak{g})_{r,s}^{\text{an}}$ and $U_{q}(\mathfrak{g})_{\rho,\sigma}^{\dagger}$ are analytically graded by $\mathbb{Z}I \cong \Psi$ (\emph{i.e} are graded by the Banach group Hopf algebra of $\Psi$).
\end{prop}

\begin{proof}
Note that $\mathbf{f}_{r}^{\text{an}}$ and $\mathbf{f}_{s}^{\text{an}}$ are $H'$-comodules, both left and right since $H'$ is cocommutative. If we give $H$ the trivial $H'$-coaction,
$$H \cong k \hat{\otimes} H \overset{\eta_{H'} \otimes \text{Id}_{H}}{\longrightarrow} H' \hat{\otimes} H,$$
then all of the morphisms involved in defining $U(\mathbf{f}_{r}^{\text{an}},H,\mathbf{f}_{s}^{\text{an}})$ are $H'$-comodule homomorphisms. The result then follows since $H'$ is isomorphic to the Banach group Hopf algebra of $\Phi$.
\end{proof}

\begin{prop}
\label{QuantumGroupDenseInNAAnalyticVersion}
For each $0< r,s$ with $1 \leq |q_{i}-q_{i}^{-1}|rs$ for all $i \in I$, the quantum enveloping algebra $U_{q}(\mathfrak{g})$ is dense in the Banach space $U_{q}(\mathfrak{g})_{r,s}^{\text{an}}$, and the Hopf structure on $U_{q}(\mathfrak{g})_{r,s}^{\text{an}}$ restricts to the usual Hopf structure on $U_{q}(\mathfrak{g})$.
\end{prop}

\begin{proof}
This follows from Proposition 4.3 of \cite{DBoBG}, Lemma \ref{PositivePartDense} and the triangular decomposition of $U_{q}(\mathfrak{g})$ given by the Poincar\'e-Birkhoff-Witt Theorem.
\end{proof}

\begin{prop}
\label{EpimorphismFromAnalyticQuantumGroup}
Suppose $\mathcal{U}$ is a Banach Hopf algebra in which $U_{q}(\mathfrak{g})$ is a dense sub-Hopf algebra, with $\mathcal{U}$ analytically $\mathbb{Z}I$ graded (\emph{i.e} are graded by the Banach group Hopf algebra of $\Psi$) extending the grading on $U_{q}(\mathfrak{g})$. Then there is an epimorphism of Banach Hopf algebras
$$U_{q}(\mathfrak{g})_{r,s}^{\text{an}} \rightarrow \mathcal{U}$$
for some $r,s>0$ with $|q_{i}-q_{i}^{-1}|rs\geq 1$ for all $i \in I$.
\end{prop}

\begin{proof}
Let $\mathcal{U}^{-}$, $\mathcal{U}^{0}$, and $\mathcal{U}^{+}$ be the respective closured of $U_{q}^{-}(\mathfrak{g})$, $U_{q}^{0}(\mathfrak{g})$, and $U_{q}^{+}(\mathfrak{g})$ in $\mathcal{U}$. The $\mathbb{Z}I$ grading on $U_{q}^{+}(\mathfrak{g})$ is concentrated in degrees in the submonoid $\mathbb{N}I$, and hence so is the analytic grading on $\mathcal{U}^{+}$. The homomorphism of monoids from this submonoid to $\mathbb{N}$, $\sum_{i \in I}n_{i} \cdot i \mapsto \sum_{i \in I}n_{i}$ gives $\mathcal{U}^{+}$ an analytic $\mathbb{N}$ grading compatible with that of the algebraic Nichols algebra $U_{q}^{+}(\mathfrak{g})$. Since $\mathcal{U}^{+}$ is a completion of the algebraic Nichols algebra $U_{q}^{+}(\mathfrak{g})$ it must be an analytic Nichols algebra of $V$ as in Definition \ref{AnalyticQuantumGroupPositivePart}, and hence by Proposition \ref{EpimorphismFromAnalyticQuantumGroup} we have an epimorphism $\mathbf{f}_{s} \rightarrow \mathcal{U}^{+}$ for some $s>0$. Likewise we obtain an epimorphism $\mathbf{f}_{r} \rightarrow \mathcal{U}^{-}$ for some $r>0$, and without loss of generality we may take $r$ and $s$ such that $rs|q_{i}-q_{i}^{-1}| \geq 1$ for all $i \in I$. Finally, each $K_{\lambda} \in \mathcal{U}_{0}$ is grouplike, so since comultiplication on $\mathcal{U}_{0}$ is bounded by some constant $C>0$ we have $\|K_{\lambda}\|^{2}=\|K_{\lambda} \otimes K_{\lambda}\| \leq C\|K_{\lambda}\|$, so $\|K_{\lambda}\|\leq C$. Thus we can define a bounded epimorphism $H \rightarrow \mathcal{U}^{0}$ by mapping each $K_{\lambda}$ to the associated element of $U_{q}^{0}(\mathfrak{g})$ inside $\mathcal{U}^{0}$. This gives us an epimorphism
$$U_{q}(\mathfrak{g})^{\text{an}}_{r,s} = \mathbf{f}^{\text{an}}_{r} \hat{\otimes} H \hat{\otimes} \mathbf{f}^{\text{an}}_{s} \twoheadrightarrow \mathcal{U}^{-} \hat{\otimes} \mathcal{U}^{0} \hat{\otimes} \mathcal{U}^{+} \overset{\mu}{\twoheadrightarrow} \mathcal{U}$$
that restricts to a Hopf algebra homomorphism on the dense subspace $U_{q}(\mathfrak{g})$.
\end{proof}

\subsection{Quasi-triangularity and the quasi-R-matrix}
\

Suppose throughout this section that the symmetrised Cartan matrix associated to our root datum in Definition \ref{KacMoodyRootDatum}, $A=((\alpha_{i},\alpha_{j}))_{i,j \in I}$, is invertible (over $\mathbb{Q}$). Suppose further that $q=\text{exp}(\hslash)$ for some $\hslash \in k$ of sufficiently small norm such that ${\frac{1}{|n!|}(|\hslash| \cdot \text{max}_{i,j}|A^{-1}_{i,j}|)^{n}}$ converges to $0$, where $A_{i,j}^{-1}$ are the entries of the inverse of the Cartan matrix.

\begin{remark}
If $k$ is an extension of $\mathbb{Q}_{p}$, the requirement that ${\frac{1}{|n!|}(|\hslash| \cdot \text{max}_{i,j}|A^{-1}_{i,j}|)^{n}}$ converges to $0$ is equivalent to $|\hslash| \cdot \text{max}_{i,j}|A^{-1}_{i,j}|<p^{\frac{1}{1-p}}$. If $k$ is such that $|n|=1$ for all integers $n \in \mathbb{Z}$ then this is just the assumption that $|\hslash|<1$.
\end{remark}

\begin{defn}
Let us denote by $\mathscr{H}$ the Banach space ${\mathscr{H}:=\coprod_{\alpha \in \mathbb{Z}I}^{\leq 1} k \cdot H_{\alpha}}$. This has an algebra structure induced by the group structure of $\mathbb{Z}I$, and we make $\mathscr{H}$ a Hopf algebra by defining
$$\Delta_{\mathscr{H}}(H_{i})=1 \otimes H_{i} + H_{i} \otimes 1 \quad  \text{for} \quad i \in I.$$
We will denote by $e:H' \rightarrow \mathscr{H}$ the morphism of Hopf algebras determined by
$$t_{i} \mapsto \text{exp}\left(\frac{(\alpha_{i},\alpha_{i})}{2}\hslash H_{i}\right)=\sum_{n \geq 0} \frac{(\alpha_{i},\alpha_{i})^{n}\hslash^{n}}{2^{n} \cdot n!}H_{ni}.$$
\end{defn}

\begin{lem}
\label{CartanPartQuasiTriangular}
Under our assumptions, $\mathscr{H}$ is quasi-triangular with R-matrix
$$\mathscr{R}_{\mathscr{H}}=\text{exp}(\hslash \sum_{i,j}A^{-1}_{i,j}H_{i} \otimes H_{j}).$$
\end{lem}

\begin{proof}
Since $\mathscr{H}$ is both commutative and cocommutative, it is trivial to check that $\mathscr{R}_{\mathscr{H}}$ is an R-matrix. It remains only to check that this converges, which follows from our assumptions since
$$\|\hslash \sum_{i,j}A^{-1}_{i,j}H_{i} \otimes H_{j}\| \leq |\hslash| \cdot \text{max}_{i,j}|A^{-1}_{i,j}|.$$
\end{proof}

\begin{prop}
\label{HModuleBraiding}
There is a braiding on $H\text{-Mod}_{\Phi}$ extending that of Lemma \ref{HAPairingWeakQuasiTriangular}.
\end{prop}

\begin{proof}
Suppose $M_{\alpha}$ and $N_{\alpha'}$ are Banach $H$-modules, and hence $H'$-modules, of weights $\alpha$ and $\alpha'$ in $\Phi$ respectively. Then $M_{\alpha}$ is a $\mathscr{H}$-module where $H_{i}$ acts by the scalar $(\alpha_{i},\alpha)$, and likewise so is $N_{\alpha'}$. The action of $\mathscr{R}_{\mathscr{H}}$ induces the braiding
$$M_{\alpha} \hat{\otimes} N_{\alpha'} \rightarrow N_{\alpha'} \hat{\otimes} M_{\alpha}, \quad m \otimes n \mapsto q^{\sum_{i,j} A_{i,j}^{-1}(\alpha_{i},\alpha)(\alpha_{j},\alpha')}n \otimes m.$$
This braiding commutes not only with the action of $H'$ but also with that of $H$. If $\alpha=\sum n_{i} \alpha_{i}$ and $\alpha'=\sum n'_{i} \alpha_{i}$ in $\Psi$ then
$$\begin{array}{rcl}
\sum_{i,j} A_{i,j}^{-1}(\alpha_{i},\alpha)(\alpha_{j},\alpha') &=& \sum_{i,j,i',j'} A_{i,j}^{-1}n_{i'}n'_{j'}(\alpha_{i},\alpha_{i'})(\alpha_{j},\alpha_{j'})\\
&=& \sum_{i,j,i',j'} A_{i,j}^{-1}n_{i'}n'_{j'}(\alpha_{i},\alpha_{i'})A_{j,j'}\\
&=& \sum_{i,i'} n_{i'}n'_{i}(\alpha_{i},\alpha_{i'}) = (\alpha,\alpha'),
\end{array}$$
so the braiding restricts to the one given by Lemma \ref{HAPairingWeakQuasiTriangular} on $H\text{-Mod}_{\Psi}$.
\end{proof}

\label{RMatrixDoesn'tConverge}
Proposition 3.6 of \cite{DBoBG} gives criterion for when a double-bosonisation is quasi-triangular. The Banach version of this would be the following.

\begin{theorem*}
Let $H$ be a quasi-triangular Banach Hopf algebra, let $B$ be a braided Banach Hopf algebra $B$ in $H\text{-Mod}$ and let $C$ be a braided Banach Hopf algebra in $\text{Mod-}H$, equipped with a duality pairing between $B$ and $C$, $\langle - , - \rangle: B \hat{\otimes} C \rightarrow k$. Suppose there exists linearly independent subsets $\{b_{n} \mid n \in \mathbb{N}\} \subset B$ and $\{c_{n} \mid n \in \mathbb{N}\} \subset C$ that span dense subspaces and satisfy
$$\langle b_{n},c_{m} \rangle = \delta_{n,m} \quad \text{and} \quad \|b_{n}\|\cdot \|c_{n}\| \rightarrow 0.$$
Then $U(C,H,B)$ is quasi-triangular with R-matrix
$$\mathscr{R}_{U(C,H,B)}:=\text{exp}_{C,B} \cdot \mathscr{R}_{H}=\sum_{x \in X} (c_{x} \otimes \mathscr{R}_{2}^{(1)}\mathscr{R}_{1}^{(1)} \otimes 1) \otimes (1 \otimes \mathscr{R}_{1}^{(2)} \otimes S_{B}(b_{x}) \cdot \mathscr{R}_{2}^{(2)})$$
where $\text{exp}_{C,B}=\sum_{x \in X} c_{x} \otimes S_{B}(b_{x})$, and $\mathscr{R}_{1}=\sum \mathscr{R}_{1}^{(1)} \otimes \mathscr{R}_{1}^{(2)}$ and $\mathscr{R}_{2}=\sum \mathscr{R}_{2}^{(1)}\otimes \mathscr{R}_{2}^{(2)}$ are copies of the R-matrix $\mathscr{R}_{H}$ on $H$.
\end{theorem*}

\noindent Unfortunately, as in \cite{DBoBG}, the conditions required for this theorem do not hold if $B$ and $C$ are infinite dimensional, hence it does not apply to $U(\mathbf{f}_{r}^{\text{an}}, \mathscr{H}, \mathbf{f}_{s}^{\text{an}})$. Indeed, if such sets were to exist for infinite dimensional $B$ and $C$ then $\sum b_{n} \otimes c_{n}$ would converge in $B \otimes C$ but its image under $\langle -,- \rangle$ would not.\\

If we allow ourselves to work over formal powerseries in $\hslash$ then we may define an R-matrix for some analytic quantum groups, which we will see in Section \ref{WorkingOverk[[h]]}.

\subsection{Quantum groups as Drinfel'd doubles}

\begin{lem}
\label{Drinfel'dDoubleConstruction}
Suppose we have IndBanach Hopf algebras $B$ and $C$ with a duality pairing $\langle -,- \rangle : B \hat{\otimes} C^{\text{op}} \rightarrow k$. Then there is a Hopf algebra structure on $C \hat{\otimes} B$ such that $B \rightarrow C \hat{\otimes} B$ and $C \rightarrow C \hat{\otimes} B$ are morphisms of Hopf algebras and the multiplication restricts to the composition
$$\begin{array}{rcl}
B \hat{\otimes} C &\longrightarrow& B \hat{\otimes} B \hat{\otimes} B \hat{\otimes} C \hat{\otimes} C \hat{\otimes} C\\
&\longrightarrow& B \hat{\otimes} C \hat{\otimes} C \hat{\otimes} B \hat{\otimes} B \hat{\otimes} C\\
&\longrightarrow& C \hat{\otimes} B
\end{array}$$
where the first morphism is given by the respective comultiplications, the second is a reordering given by the permutation $(2 \quad 4)(3 \quad 5) \in S_{6}$ and the last is given by $\langle -,- \rangle \otimes \text{Id} \otimes \text{Id} \otimes \langle -,- \rangle$.
\end{lem}

\begin{proof}
This construction is outlined in Section 8.2.1 of \cite{QGaTR} and Section IX.5 of \cite{QG} for (finite dimensional) vector spaces, but is easily generalised to the category of IndBanach spaces.
\end{proof}

\begin{defn}
We will denote by $D(B,C)$ the Hopf algebra on $C \hat{\otimes} B$ described in the previous lemma, the \emph{relative Drinfel'd double} of $B$ and $C$.
\end{defn}

\begin{lem}
\label{DualityPairingforQuantumDouble}
There is a duality pairing $(\overline{H} \ltimes \overline{\mathbf{f}^{\text{an}}_{s}}) \hat{\otimes} (\mathbf{f}^{\text{an}}_{r} \rtimes H')^{\text{op}} \rightarrow k$ such that
$$
\begin{array}{rclrcl}
K_{\lambda} \otimes t_{j} &\mapsto& q^{-\lambda(\alpha_{j})},&
K_{\lambda} \otimes F_{j} &\mapsto& 0,\\
E_{i} \otimes t_{j} &\mapsto& 0,&
E_{i} \otimes F_{j} &\mapsto& -\delta_{i,j} \frac{1}{q_{i}-q_{i}^{-1}}.
\end{array}
$$
\end{lem}

\begin{proof}
This follows as in Section 6.3.1 of \cite{QGaTR}. We first define bounded algebra homomorphisms $\phi_{i}$ and $\psi_{i}$ from $H \ltimes \overline{\mathbf{f}^{\text{an}}_{s}}$ to $k$ such that
$$\phi_{i}(K_{\lambda} \otimes v_{i})=\frac{1}{q_{i}^{-1}-q_{i}}, \quad \phi_{i}(K_{\lambda} \otimes x)=0$$
and
$$\psi_{i}(K_{\lambda} \otimes y)=q^{-\lambda(\alpha_{i})} \varepsilon(y)$$
for all $i \in I$, $x \in \mathbf{f}_{s}^{\text{an}}(\alpha)$, $\alpha_{i} \neq \alpha \in \Psi$, and $y \in \mathbf{f}_{s}^{\text{an}}$. Here $\varepsilon$ is the counit on $\mathbf{f}_{s}^{\text{an}}$. We define contracting morphisms $V_{r} \rightarrow (\overline{H} \ltimes \overline{\mathbf{f}^{\text{an}}_{s}})^{\ast}$ and $\coprod_{i\in I}^{\leq 1}t_{i} \rightarrow (\overline{H} \ltimes \overline{\mathbf{f}^{\text{an}}_{s}})^{\ast}$ defined respectively by $v_{i} \mapsto \phi_{i}$ and $t_{i} \mapsto \psi_{i}$. These induce algebra homomorphisms $T_{r}(V) \rightarrow (\overline{H} \ltimes \overline{\mathbf{f}^{\text{an}}_{s}})^{\ast}$, $H' \rightarrow (\overline{H} \ltimes \overline{\mathbf{f}^{\text{an}}_{s}})^{\ast}$. It follows from the proof of Proposition 34 in \emph{loc. cit.} that the map $T_{r}(V) \rightarrow (\overline{H} \ltimes \overline{\mathbf{f}^{\text{an}}_{s}})^{\ast}$ factors as a composition $T_{r}(V) \rightarrow \mathbf{f}_{r}^{\text{an}} \rightarrow (H \ltimes \overline{\mathbf{f}^{\text{an}}_{s}})^{\ast}$. From this we obtain a morphism $\mathbf{f}^{\text{an}}_{r} \rtimes H' \rightarrow (\overline{H} \ltimes \overline{\mathbf{f}^{\text{an}}_{s}})^{\ast}$ and hence a bilinear form as desired. The fact that this is a duality pairing is checked on a dense subspace in \cite{QGaTR}, and extends by continuity.
\end{proof}

\begin{remark}
It is shown in Proposition 34 of \cite{QGaTR} that this pairing can be expressed as the composition
$$
\begin{array}{rcccl}
H \hat{\otimes} \mathbf{f}^{\text{an}}_{s} \hat{\otimes} \mathbf{f}^{\text{an}}_{r} \hat{\otimes} H' &\overset{\text{Id} \otimes \text{Id} \otimes S \otimes\text{Id}}{\xrightarrow{\hspace*{1.5cm}}}& H \hat{\otimes} \mathbf{f}^{\text{an}}_{s} \hat{\otimes} \mathbf{f}^{\text{an}}_{r} \hat{\otimes} H' 
&\overset{\text{Id} \otimes \langle -,- \rangle \otimes\text{Id}}{\xrightarrow{\hspace*{1.5cm}}}& H \hat{\otimes} k \hat{\otimes} H' \\
&\overset{\text{Id} \otimes S}{\xrightarrow{\hspace*{1.5cm}}}& H \hat{\otimes} H' 
&\overset{\langle -,- \rangle}{\xrightarrow{\hspace*{1.5cm}}}&  k.
\end{array}
$$
\end{remark}

\begin{defn}
\label{CrossedBimodules}
For an IndBanach Hopf algebra $C$, we will denote by $_{C}\text{Cross}^{C}$ the category of IndBanach spaces $V$ equipped with both a left action and right coaction of $C$, $\mu_{V}: C \hat{\otimes} V \rightarrow V$ and $\Delta_{V}: V \rightarrow V \hat{\otimes} C$, such that the following diagram commutes:
\begin{center}
\begin{tikzpicture}[node distance=6cm, auto]
  \node (A) {$C \hat{\otimes} V$};
  \node (B) [right=3cm of A] {$C \hat{\otimes} C \hat{\otimes} V$};
  \node (C) [below=0.5cm of B] {$C \hat{\otimes} V$};
  \node (D) [below=0.5cm of A] {$C \hat{\otimes} C \hat{\otimes} V \hat{\otimes} C$};
  \node (E) [below=0.5cm of C] {$V \hat{\otimes} C$};
  \node (F) [below=0.5cm of D] {$C \hat{\otimes} V \hat{\otimes} C \hat{\otimes} C$};
  \node (G) [below=0.5cm of F] {$V \hat{\otimes} C$};
  \node (H) [below=0.5cm of E] {$V \hat{\otimes} C \hat{\otimes} C.$};
  \draw[->] (A) to node {$\Delta_{C} \otimes \text{Id}$} (B);
  \draw[->] (B) to node {$\text{Id} \otimes \mu_{V}$} (C);
  \draw[->] (C) to node {$\tau$} (E);
  \draw[->] (E) to node {$\Delta_{V} \otimes \text{Id}$} (H);
  \draw[->] (H) to node {$\text{Id} \otimes \mu_{C}$} (G);
  \draw[->] (A) to node {$\Delta \otimes \Delta_{V}$} (D);
  \draw[->] (D) to node {$\text{Id} \otimes \tau \otimes \text{Id}$} (F);
  \draw[->] (F) to node {$\mu_{V} \otimes \mu_{C}$} (G);
\end{tikzpicture}
\end{center}
We will refer to these as \emph{crossed bimodules}, although they are sometimes also referred to as \emph{Yetter-Drinfel'd modules}.
\end{defn}

\begin{lem}
\label{CrossedBimodulesBraided}
The category $_{C}\text{Cross}^{C}$ is pre-braided monoidal with braiding given by
$$M \hat{\otimes} N \overset{\tau}{\rightarrow} N \hat{\otimes} M \overset{\Delta_{N} \otimes \text{Id}}{\longrightarrow} N \hat{\otimes} C \hat{\otimes} M \overset{\text{Id} \otimes \mu_{M}}{\longrightarrow} N \hat{\otimes} M$$
for $M$ and $N$ in $_{C}\text{Cross}^{C}$. If $C$ has an invertible antipode then this is a braiding.
\end{lem}

\begin{proof}
This is Theorem 5.2 and Theorem 7.2 of \cite{QGaRoMC}.
\end{proof}

\begin{lem}
\label{ModulesofQuantumDouble}
There is a faithful functor $_{C}\text{Cross}^{C} \rightarrow D(B,C)\text{-Mod}$.
\end{lem}

\begin{proof}
This follows from the proof of Proposition 6 in Section 13.1 of \cite{QGaTR}. An object $V$ of $_{C}\text{Cross}^{C}$ already has an action of $C$. It gains an action of $B$ using the duality pairing, giving a map
$$\mu_{V}':B \hat{\otimes} V \overset{\text{Id} \otimes \Delta_{V}}{\xrightarrow{\hspace*{0.8cm}}} B \hat{\otimes} V \hat{\otimes} C \overset{(\text{Id} \hat{\otimes} \langle -,- \rangle) \circ (\tau \otimes \text{Id})}{\xrightarrow{\hspace*{2.5cm}}} V.$$
Then $B$ and $C$ generate $D(B,C)$, and the fact that $\mu_{V}$ and $\mu'_{V}$ together give a well defined action of $D(B,C)=C \hat{\otimes} B$ on $V$ (given by $\mu_{V} \circ (\text{Id} \otimes \mu'_{V})$) follows from the commutativity of the diagram in Definition \ref{CrossedBimodules} as in the proof of Proposition 6 of \emph{loc. cit.}
\end{proof}

\begin{prop}
\label{NAAnalyticQuantumGroupAsDrinfel'dDouble}
There is a strict epimorphism of braided analytically graded Banach Hopf algebras
$$D(\overline{H} \ltimes \overline{\mathbf{f}^{\text{an}}_{s}},\mathbf{f}^{\text{an}}_{r} \rtimes H') \rightarrow U_{q}(\mathfrak{g})^{\text{an}}_{r,s}$$
whose kernel is the closed two sided ideal generated by
$$\{1 \otimes t_{i} \otimes 1 \otimes 1 - 1 \otimes 1 \otimes t_{i} \otimes 1 \mid i \in I\} \subset \mathbf{f}^{\text{an}}_{r} \hat{\otimes} H \hat{\otimes} H' \hat{\otimes} \mathbf{f}^{\text{an}}_{s}.$$
\end{prop}

\begin{proof}
This strict epimorphism can be written as
$$\mathbf{f}^{\text{an}}_{r} \hat{\otimes} H' \hat{\otimes} H \hat{\otimes} \mathbf{f}^{\text{an}}_{s} \xrightarrow{\text{Id} \otimes \mu_{H} \otimes \text{Id}} \mathbf{f}^{\text{an}}_{r} \hat{\otimes} H \hat{\otimes} \mathbf{f}^{\text{an}}_{s}.$$
Corollary 15 of Section 8.2.4 of \cite{QGaTR}, along with Proposition \ref{QuantumGroupDenseInNAAnalyticVersion}, show that this restricts to a morphism of braided graded Hopf algebras between the dense subspaces
$$U_{q}^{\leq 0}(\mathfrak{g}) \otimes U_{q}^{\geq 0}(\mathfrak{g}) \rightarrow U_{q}(\mathfrak{g})$$
whose kernel is generated by $t_{i} \otimes 1 - 1 \otimes t_{i}$. The result then follows by continuity.
\end{proof}

\begin{defn}
Let us denote by $\mathcal{C}$ the full subcategory of $_{(\mathbf{f}^{\text{an}}_{r} \rtimes H')}\text{Cross}^{(\mathbf{f}^{\text{an}}_{r} \rtimes H')}$ consisting of IndBanach spaces $V$ equipped with both a left action and right coaction of $\mathbf{f}^{\text{an}}_{r} \rtimes H'$, $\mu_{V}: (\mathbf{f}^{\text{an}}_{r} \rtimes H') \hat{\otimes} V \rightarrow V$ and $\Delta_{V}: V \rightarrow V \hat{\otimes} (\mathbf{f}^{\text{an}}_{r} \rtimes H')$, such that the following additional diagram commutes:
\begin{center}
\begin{tikzpicture}[node distance=6cm, auto]
  \node (A) {$H' \hat{\otimes} V$};
  \node (B) [below=1cm of A] {$(\mathbf{f}^{\text{an}}_{r} \rtimes H') \hat{\otimes} V$};
  \node (C) [right=1cm of A] {$(\overline{H} \ltimes \overline{\mathbf{f}^{\text{an}}_{s}}) \hat{\otimes} V$};
  \node (D) [below=1.05cm of C] {$V$};
  \draw[->] (A) to node {} (B);
  \draw[->] (C) to node {$\mu_{V}'$} (D);
  \draw[->] (A) to node {} (C);
  \draw[->] (B) to node {$\mu_{V}$} (D);
\end{tikzpicture}
\end{center}
where $\mu_{V}'$ is the action described in the proof of Lemma \ref{ModulesofQuantumDouble}.
\end{defn}

\begin{lem}
\label{FullyFaithfulFunctorCtoUqanMod}
There is a fully faithful functor $\mathcal{C} \rightarrow U_{q}(\mathfrak{g})^{\text{an}}_{r,s}\text{-Mod}$.
\end{lem}

\begin{proof}
By Lemma \ref{ModulesofQuantumDouble} and Proposition \ref{NAAnalyticQuantumGroupAsDrinfel'dDouble} there is a faithful functor $\mathcal{C} \rightarrow U_{q}(\mathfrak{g})^{\text{an}}_{r,s}\text{-Mod}$. Let $f:M \rightarrow N$ be a morphism in $U_{q}(\mathfrak{g})^{\text{an}}_{r,s}\text{-Mod}$ between two objects in the image of $\mathcal{C}$. Then the action of $\overline{H} \ltimes \overline{\mathbf{f}_{s}^{\text{an}}}$ on both $M$ and $N$ is induced by a coaction of $\mathbf{f}_{r}^{\text{an}} \rtimes H'$ via their dual pairing. Furthermore, $f$ commutes with the action of $U_{q}(\mathfrak{g})^{\text{an}}_{r,s}$, and hence with the actions of $\mathbf{f}_{r}^{\text{an}} \rtimes H'$ and $\overline{H} \ltimes \overline{\mathbf{f}_{s}^{\text{an}}}$. The morphism $f$ must therefore preserve the locally Banach weight space decompositions of $M$ and $N$, so it preserves their respective coactions of $H'$. Since the pairing of $\mathbf{f}_{r}^{\text{an}}$ and $\mathbf{f}_{s}^{\text{an}}$ is non-degenerate, $f$ must also preserve the coaction of $\mathbf{f}_{r}^{\text{an}}$ that induces the action of $\mathbf{f}_{s}^{\text{an}}$. Hence $f$ preserves the coaction of $\mathbf{f}_{r}^{\text{an}} \rtimes H'$, so comes from a morphism in $\mathcal{C}$. So the functor is fully faithful.
\end{proof}

\begin{defn}
\label{IntegrableRepresentations}
We will denote by $\mathcal{O}_{\Psi}$ the essential image of $\mathcal{C}$ in $U_{q}(\mathfrak{g})^{\text{an}}_{r,s}\text{-Mod}$. This is the full subcategory of $U_{q}(\mathfrak{g})^{\text{an}}_{r,s}$ modules whose action of $H$ gives a representation in $\text{Comod-}H' = H\text{-Mod}_{\Psi}$ and whose action of $U_{q}^{-}(\mathfrak{g})^{\text{an}}_{s}$ comes from a coaction of $\mathbf{f}_{r}^{\text{an}}$ via their dual pairing. Let us denote by $\mathcal{O}$ the full subcategory of $U_{q}(\mathfrak{g})^{\text{an}}_{r,s}\text{-Mod}$ consisting of modules whose action of $U_{q}^{-}(\mathfrak{g})^{\text{an}}_{s}$ comes from a coaction of $\mathbf{f}_{r}^{\text{an}}$ and whose action of $H$ gives a representation in $H\text{-Mod}_{\Phi}$. Note that $\mathcal{O}_{\Psi} \subset \mathcal{O}$.
\end{defn}

\begin{remark}
Note that the conditions for a $U_{q}(\mathfrak{g})^{\text{an}}_{r,s}$ module to be in $\mathcal{O}$ resemble the conditions for a $U_{q}(\mathfrak{g})$ module to be integrable, with the requirement that the action of $U_{q}^{+}(\mathfrak{g})^{\text{an}}_{s}$ comes from a coaction of $\mathbf{f}_{r}^{\text{an}}$ taking the place of a locally finite dimensional action of $U_{q}^{+}(\mathfrak{g})$.
\end{remark}

\begin{corollary}
The category $\mathcal{O}_{\Psi}$ is braided monoidal.
\end{corollary}

\begin{proof}
This is a consequence of Lemma \ref{CrossedBimodulesBraided} and Lemma \ref{FullyFaithfulFunctorCtoUqanMod}.
\end{proof}

\begin{remark}
A short computation shows that the braiding on $\mathcal{O}_{\Psi}$ can be expressed as performing the braiding on $H'$-comodules given by Lemma \ref{HAPairingWeakQuasiTriangular} followed by the action of $\text{exp}_{\mathbf{f}_{r}^{\text{an}},\mathbf{f}_{r}^{\text{an}}}$ as described at the end of Section \ref{RMatrixDoesn'tConverge}, which is well defined despite $\text{exp}_{\mathbf{f}_{r}^{\text{an}},\mathbf{f}_{r}^{\text{an}}}$ not converging. This is expected given the description of the R-matrix at the end of Section \ref{RMatrixDoesn'tConverge}.
\end{remark}

\begin{prop}
Suppose that the symmetrised Cartan matrix associated to our root datum, $A=((\alpha_{i},\alpha_{j}))_{i,j \in I}$, has an inverse over $\mathbb{Q}$ with entries $A_{i,j}^{-1}$. Suppose further that $q=\text{exp}(\hslash)$ for some $\hslash \in k$ of sufficiently small norm such that ${\frac{1}{|n!|}(|\hslash| \cdot \text{max}_{i,j}|A^{-1}_{i,j}|)^{n}}$ converges to $0$. Then there is a braiding on the category $\mathcal{O}$ extending that of $\mathcal{O}_{\Psi}$.
\end{prop}

\begin{proof}
Given $M$ and $N$ in $\mathcal{O}$, the braiding is given by the composition
$$M \hat{\otimes} N \xrightarrow{\tau} N \hat{\otimes} M \longrightarrow N \hat{\otimes} \mathbf{f}_{r}^{\text{an}} \hat{\otimes} M \longrightarrow N \hat{\otimes} M \xrightarrow{\mathscr{R}_{\mathscr{H}}} N \hat{\otimes} M$$
where the second morphism is the coaction of $\mathbf{f}_{r}^{\text{an}}$ on $N$, the third is its action on $M$, and the last in the action of the R-matrix as in the proof of Proposition \ref{HModuleBraiding}. The computations in the proof of Proposition 3.6 of \cite{DBoBG}, alongside the previous remark, ensure that this is a braiding. This is an extension of the braiding in Lemma \ref{CrossedBimodulesBraided}.
\end{proof}

\begin{example}
Let $\mathfrak{g}=\mathfrak{sl}_{2}$, and fix $r,s>0$ with $|q_{i}-q_{i}^{-1}|rs\geq 1$ for all $i \in I$. Let $\lambda \in \Phi \cong \mathbb{Z}$, and let $W_{\lambda}=k\{\frac{x}{r}\}=\coprod_{n \geq 0}^{\leq 1} k_{r^{n}} \cdot x^{n}$ with the action of $U_{q}(\mathfrak{sl}_{2})_{r,s}^{\text{an}}$ given by
$$K \cdot x^{n} = q^{\lambda -2n}x^{n}, \quad F \cdot x^{n}= x^{n+1}, \quad E \cdot x^{n} = [n][\lambda - (n-1)] x^{n-1}$$
where $x^{-1}:=0$. Alternatively, taking $y^{n}=\frac{x^{n}}{[n]!}$, $W=\{\sum a_{n}y^{n} \mid \frac{|a_{n}|r^{n}}{[n]!} \rightarrow 0\}$ with action
$$K \cdot y^{n} = q^{\lambda -2n}y^{n}, \quad F \cdot y^{n}= [n+1]y^{n+1}, \quad E \cdot y^{n} = [\lambda - (n-1)] y^{n-1}$$
where $y^{-1}:=0$. Note that $W_{\lambda}$ is isomorphic to the quotient of $U_{q}(\mathfrak{sl}_{2})_{r,s}^{\text{an}}$ by the closed left ideal generated by $E$ and $K-q^{\lambda}$, and so we call it an \emph{analytic Verma module} of weight $\lambda$. This gives a representation in $\mathcal{O}$ where the coaction of $\mathbf{f}_{r}^{an}$ is given by
$$\begin{array}{rrl}
x^{n} &\longmapsto& \sum (-1)^{k}\frac{(q-q^{-1})^{k}}{[k]!} F^{k} \otimes E^{k}x^{n}\\
&& =\sum_{k=0}^{n} (-1)^{k}(q-q^{-1})^{k}\frac{[n]![\lambda -n +k]!}{[k]![n-k]![\lambda-n]!} F^{k} \otimes x^{n-k},\\
y^{n} &\longmapsto& \sum_{k=0}^{n} (-1)^{k}(q-q^{-1})^{k}\frac{[\lambda -n +k]!}{[k]![\lambda-n]!} F^{k} \otimes y^{n-k}.
\end{array}$$
Given $\lambda, \lambda' \in \Phi$ the braiding $W_{\lambda} \hat{\otimes} W_{\lambda'} \rightarrow W_{\lambda'} \hat{\otimes} W_{\lambda}$ is given by
$$\begin{array}{rrl}
x^{n} \otimes x^{m} &\longmapsto& \sum_{k=0}^{n} (-1)^{k}(q-q^{-1})^{k}\frac{[n]![\lambda -n +k]!}{[k]![n-k]![\lambda-n]!} q^{2(m+k)(n-k)} \ x^{m+k} \otimes x^{n-k},\\
y^{n} \otimes y^{m} &\longmapsto& \sum_{k=0}^{n} (-1)^{k}(q-q^{-1})^{k}\frac{[m+k]![\lambda -n +k]!}{[k]![m]![\lambda-n]!} q^{2(m+k)(n-k)} \ y^{m+k} \otimes y^{n-k}.
\end{array}$$
\end{example}

\begin{example}
Note that objects in $\mathcal{O}$ are not necessarily generated by highest weight vectors. We give an example of a representation with no such highest weights. Let $\mathfrak{g}=\mathfrak{sl}_{2}$, fix $r,s>0$ with $|q_{i}-q_{i}^{-1}|rs\geq 1$ for all $i \in I$, and suppose that $q=\text{exp}(\hslash)$ for $|\hslash| \ll 1$. Then
$$q-q^{-1}=\left(\sum_{n=0}^{\infty}\frac{(\hslash)^{n}}{n!}\right)-\left(\sum_{n=0}^{\infty}\frac{(-\hslash)^{n}}{n!}\right)=2\hslash \left(\sum_{k=0}^{\infty}\frac{(\hslash)^{2k}}{k!}\right)$$
so $|q-q^{-1}|=|2\hslash|$, and
$$\begin{array}{rcl}
[n]-n&=&q^{n-1}+q^{n-3}+...+q^{-n+1}-n\\
&=& \hslash\sum_{k=1}^{\infty} \frac{1}{k!}[(n-1)^{k}+(n-3)^{k}+...+(-n+1)^{k}]\hslash^{k-1}
\end{array}$$
has norm strictly smaller than 1 for $|\hslash|$ sufficiently small, so $|[n]|=|n|$. Fix some $\lambda \in \mathbb{Z}_{\geq 0}$. Let us define
$$M_{\lambda}:= \coprod\nolimits^{\leq 1}_{i,j \geq 0} k_{r^{j-i}} \cdot x_{i,j}=\left\lbrace\sum \alpha_{i,j}x_{i,j} \middle| |\alpha_{i,j}|r^{j-i} \rightarrow 0\right\rbrace.$$
Then $M_{\lambda}$ becomes a $U_{q}(\mathfrak{sl}_{2})^{\text{an}}_{r,s}$ module with
$$\begin{array}{rclr}
K \cdot x_{i,j} &=& q^{\lambda+2i-2j}x_{i,j},\\
E \cdot x_{i,j} &=& x_{i+1,j},\\
F \cdot x_{0,j} &=& x_{0,j+1},\\
F \cdot x_{i,j} &=& x_{i,j+1}-[i][\lambda+i-1-2j]x_{i-1,j} & \text{for } i>0,
\end{array}$$
so that $x_{i,j} = E^{i}F^{j}x_{0,0}$ and $x_{0,0}$ is of weight $\lambda$. Note that this action is bounded since
$$\|K \cdot x_{i,j}\| = \|x_{i,j}\|,$$
$$\|E \cdot x_{i,j}\| = r^{j-i-1} = \frac{1}{rs} \|E\| \cdot \|x_{i,j}\| \leq \|E\| \cdot \|x_{i,j}\|,$$
$$\|F \cdot x_{i,j}\| \leq \text{max}\{1, |[i][\lambda+i-1-2j]|\}r^{j+1-i} \leq r^{j+1-i} = \|F\|\cdot \|x_{i,j}\|.$$
The map $M_{\lambda} \rightarrow \mathbf{f}^{\text{an}}_{r} \hat{\otimes} M_{\lambda}$ given by
$$x_{i,j} \mapsto \sum_{k \geq 0} (-1)^{k} \frac{(q-q^{-1})^{k}}{[k]!} F^{k} \otimes x_{i+k,j}$$
is well defined and bounded since $|\frac{(q-q^{-1})^{k}}{[k]!}|\frac{r^{k}r^{j-i-k}}{r^{j-i}}=\frac{|2\hslash|^{k}}{|k!|} \rightarrow 0$. This shows that $M_{\lambda}$ is indeed in $\mathcal{O}$, since
$$\sum_{k \geq 0} (-1)^{k} \frac{(q-q^{-1})^{k}}{[k]!} \langle E^{n},F^{k} \rangle \otimes x_{i+k,j}=x_{i+n,j}=E^{n}\cdot x_{i,j}.$$
\end{example}

\

\subsection{Rigidity results}
\

Classical rigidity results of Chevalley, Eilenberg and Cartan from the 1940s assert that there are no non-trivial formal deformations (as an algebra) of the universal enveloping algebra of a semisimple Lie algebra $\mathfrak{g}$. The proof relies on the vanishing of the second Lie algebra cohomology group. In this section we prove an analogous result that relies on a bounded cohomology vanishing result that has yet to be proven. We proceed as in Chapter XVIII of \cite{QG}.\\

We fix a set of root datum as in Definition \ref{KacMoodyRootDatum}.

\begin{defn}
Let $\mathscr{H}_{0}:=\coprod^{\leq 1}_{\alpha \in \mathbb{Z}I}k H_{\alpha}$ be the Banach Hopf algebra generated by $H_{i}$ for $i \in I$, where
$$\Delta_{\mathscr{H}_{0}}(H_{i})=1 \otimes H_{i} + H_{i} \otimes 1.$$
Let $V_{0}:=\coprod^{\leq 1}_{i \in I}k \cdot v_{i}$ and let $T_{r}(V_{0})=\coprod_{n \geq 0}^{\leq 1} (V_{0}^{\hat{\otimes} n})_{r^{n}}$. For $r>0$ we will denote by both $U^{-}(\mathfrak{g})_{r}^{\text{an}}$ and $U^{+}(\mathfrak{g})_{r}^{\text{an}}$ the quotient of $T_{r}(V_{0})$ by the closed ideal generated by the Serre relations
$$\sum_{k=0}^{1-(\alpha_{i},\alpha_{j})}(-1)^{k} {1-(\alpha_{i},\alpha_{j}) \choose k} v_{i}^{1-(\alpha_{i},\alpha_{j})-k}v_{j}v_{i}^{k}=0$$
for $i \neq j$.
\end{defn}

\begin{theorem}
Let $r,s>0$ such that $1 \leq rs$. Then there is a Banach algebra structure on
$$U(\mathfrak{g})_{r,s}^{\text{an}}:=U^{-}(\mathfrak{g})_{r}^{\text{an}} \hat{\otimes} \mathscr{H}_{0} \hat{\otimes} U^{+}(\mathfrak{g})_{s}^{\text{an}}$$
such that $U^{-}(\mathfrak{g})_{r}^{\text{an}}$, $\mathscr{H}_{0}$ and $U^{+}(\mathfrak{g})_{s}^{\text{an}}$ are all subalgebras, and
$$[H_{i},E_{j}]=(\alpha_{i},\alpha_{j})E_{j}, \quad [H_{i},F_{j}]=-(\alpha_{i},\alpha_{j})F_{j}, \quad [E_{i},F_{j}]=\delta_{i,j} H_{i},$$
where $F_{i}=v_{i} \otimes 1 \otimes 1$ and $E_{i}=1 \otimes 1 \otimes v_{i}$. This becomes a Banah Hopf algebra with $\mathscr{H}_{0}$ as a sub-Hopf algebra and
$$\begin{array}{rclrcl}
\Delta(E_{i})&=&E_{i} \otimes 1 + 1 \otimes E_{i},& S(E_{i})&=&-E_{i},\\
\Delta(F_{i}) &=& F_{i} \otimes 1 + 1 \otimes F_{i},& S(F_{i})&=&-F_{i}.
\end{array}$$
\end{theorem}

\begin{proof}
By construction and the triangular decomposition given by the Poincar\'e-Birkhoff-Witt Theorem the enveloping algebra $U(\mathfrak{g})$ sits as a dense subspace of $U(\mathfrak{g})_{r,s}^{\text{an}}$ on which this Hopf algebra structure is well defined. It is therefore enough to check that it extends continuously. Since
$$\begin{array}{rcccccl}
\|H_{i} \cdot F_{j}\| &=& \|F_{j}H_{i} - (\alpha_{i},\alpha_{j})F_{j}\| &\leq& r &=& \|H_{i}\| \|F_{j}\|\\
\|E_{i} \cdot H_{j}\| &=& \|H_{j}E_{i} + (\alpha_{i},\alpha_{j})E_{j}\| &\leq& s &=& \|E_{i}\| \|H_{j}\|\\
\|E_{i} \cdot F_{j}\| &=& \|F_{j}E_{i} - \delta_{i,j}H_{i}\| &\leq& rs &=& \|E_{i}\| \|F_{j}\|
\end{array}$$
and since
$$\|\Delta(x)\|=\|1 \otimes x + x \otimes 1\| \leq \|x\|$$
for each generator $x \in \{F_{i},H_{i},E_{i}\mid i \in I\}$ we may define bounded linear transformations
$$\left( T_{r}(V_{0}) \hat{\otimes} \mathscr{H}_{0} \hat{\otimes} T_{s}(V_{0}) \right) \hat{\otimes} \left( T_{r}(V_{0}) \hat{\otimes} \mathscr{H}_{0} \hat{\otimes} T_{s}(V_{0}) \right) \rightarrow U(\mathfrak{g})_{r,s}^{\text{an}}$$
$$T_{r}(V_{0}) \hat{\otimes} \mathscr{H}_{0} \hat{\otimes} T_{s}(V_{0}) \rightarrow U(\mathfrak{g})_{r,s}^{\text{an}} \hat{\otimes} U(\mathfrak{g})_{r,s}^{\text{an}}$$
which descend to the described multiplication and comultiplication maps.
\end{proof}

\begin{defn}
We will denote by $\mathfrak{g}_{r,s}$ the closed Lie subalgebra of $U(\mathfrak{g})_{r,s}^{\text{an}}$ generated by $\{F_{i},H_{i},E_{i} \mid i \in I\}$. This is a Banach Lie algebra with $\| [x,y]\| \leq \|x\| \cdot \|y\|$ for $x,y \in \mathfrak{g}_{r,s}$.
\end{defn}

\begin{prop}
Suppose we have a Banach algebra $A$ and a morphism of Banach Lie algebras $\mathfrak{g}_{r,s} \rightarrow A$ of norm at most 1. Then this extends to a unique contracting morphism of Banach algebras $U(\mathfrak{g})_{r,s}^{\text{an}} \rightarrow A$. In particular, given a Banach $\mathfrak{g}_{r,s}$ module $M$ whose action satisfies $\| x \cdot m\| \leq \|x\| \cdot \|m\|$ for $x \in \mathfrak{g}$, $m \in M$, this extends to a unique action of $U(\mathfrak{g})_{r,s}^{\text{an}}$ on $M$ such that the morphism $U(\mathfrak{g})_{r,s}^{\text{an}} \hat{\otimes} M \rightarrow M$ is contracting.
\end{prop}

\begin{proof}
Taking the images of $F_{i}$, $H_{i}$ and $E_{i}$ in $A$ gives us a contracting map
$$(T_{r}(V_{0}) \hat{\otimes} \mathscr{H}_{0} \hat{\otimes} T_{s}(V_{0})) \rightarrow A.$$
The images of $F_{i}$ and $E_{i}$ must satisfy the Serre relations in $A$, and hence this descends to a map $U(\mathfrak{g})_{r,s} \rightarrow A$. This restricts to a well defined algebra homomorphism on the dense subspace $U(\mathfrak{g})$, hence the result follows by continuity. Applying this to a morphism $\mathfrak{g}_{r,s} \rightarrow \text{Hom}(M,M)$ gives the rest of this result. 
\end{proof}

\begin{defn}
Let $M$ be a left Banach $\mathfrak{g}_{r,s}$ module whose action satisfies $\| x \cdot m\| \leq \|x\| \cdot \|m\|$ for $x \in \mathfrak{g}_{r,s}$, $m \in M$. Then we define the complex
$$C_{b}^{n}(\mathfrak{g}_{r,s},M):=\text{Hom}(\Lambda^{n}\mathfrak{g}_{r,s},M),$$
the space of bounded antisymmetric n-linear maps from $\mathfrak{g}_{r,s}$ to $M$, and let $\delta_{n}:C_{b}^{n}(\mathfrak{g}_{r,s},M) \rightarrow C_{b}^{n+1}(\mathfrak{g}_{r,s},M)$ be the map
$$\begin{array}{rcl}
\delta_{n}(f)(x_{1},.., x_{n+1})&=&\sum_{1 \leq i \leq j \leq n+1} (-1)^{i+j}f([x_{i},x_{j}],x_{1},..,\hat{x_{i}},..,\hat{x_{j}},..,x_{n+1})\\
&& \quad + \sum_{1 \leq i \leq n+1} (-1)^{i+1}x_{i}f(x_{1},..,\hat{x_{i}},..,x_{n+1}).
\end{array}$$
Note that $\|\delta_{n}\| \leq 1$ and $\delta_{n+1} \circ \delta_{n}=0$ for all $n \geq 0$. We define $H_{b}^{n}(\mathfrak{g}_{r,s},M)$ to be the seminormed space $\text{Ker}(\delta_{n})/\text{Im}(\delta_{n-1})$, the \emph{$n$th bounded Lie algebra cohomology of $\mathfrak{g}_{r,s}$ with coefficients in $M$}. We will say that $C_{b}^{\bullet}(\mathfrak{g}_{r,s},M)$ is \emph{strictly exact} at $C_{b}^{n}(\mathfrak{g}_{r,s},M)$ if $H_{b}^{n}(\mathfrak{g}_{r,s},M)=0$ and the induced map
$$C_{b}^{n-1}(\mathfrak{g}_{r,s},M)/\text{Ker}(\delta_{n-1}) \rightarrow \text{Im}(\delta_{n-1})=\text{Ker}(\delta_{n})$$
is an isomorphism.
\end{defn}

\begin{defn}
Let $M$ be as above. Then we say that a strict epimorphism $p:\tilde{\mathfrak{g}} \rightarrow \mathfrak{g}_{r,s}$ of Banach Lie algebras is an extension with kernel $M$ if $\text{Ker}(p) \cong M$ as $\mathfrak{g}_{r,s}$ modules, where $\text{Ker}(p) \subset \tilde{\mathfrak{g}}$ has the adjoint action. This extension is split if there exists a morphism of Banach Lie algebras $s:\mathfrak{g}_{r,s} \rightarrow \tilde{\mathfrak{g}}$ with $p \circ s = \text{Id}$.
\end{defn}

\begin{lem}
\label{SplittingLemma}
If $H^{2}_{b}(\mathfrak{g}_{r,s},M)=0$ then any extension $p:\tilde{\mathfrak{g}} \rightarrow \mathfrak{g}_{r,s}$ with kernel $M$ and $\|p\| \leq 1$ which is already split in $\text{Ban}_{k}$ is split as an extension of Lie algebras. Moreover, if $C_{b}^{\bullet}(\mathfrak{g}_{r,s},M)$ is strictly exact at $C_{b}^{2}(\mathfrak{g}_{r,s},M)$, the splitting $s$ has norm $\|s\| \leq C$ where $C$ is the norm of the isomorphism $\text{Ker}(\delta_{2})=\text{Im}(\delta_{1}) \overset{\sim}{\rightarrow} C_{b}^{1}(\mathfrak{g}_{r,s},M)/\text{Ker}(\delta_{1})$.
\end{lem}

\begin{proof}
This proceeds as in Proposition XVIII.1.2. in \cite{QG}.  Fix a splitting $\tilde{\mathfrak{g}}\cong \mathfrak{g}_{r,s} \oplus M$ as Banach spaces, and let
$$f(x,y)=p([(x,0),(y,0)]) \text{ for } x,y \in \mathfrak{g}_{r,s},$$
so that $f \in C_{b}^{2}(\mathfrak{g}_{r,s},M)$ and $\|f\| \leq 1$. Then it is easily checked that $f \in \text{Ker}(\delta_{2})$ and hence $f=\delta_{1}(\alpha)$ for some $\alpha \in C_{b}^{1}(\mathfrak{g}_{r,s},M)$. If $C_{b}^{\bullet}(\mathfrak{g}_{r,s},M)$ is strictly exact at $C_{b}^{2}(\mathfrak{g}_{r,s},M)$ then $\|\alpha\| \leq C\|f\| \leq C$. Then, as in \emph{loc. cit.}, $s(x)=(x,-\alpha(x))$ gives a splitting, and $\|s\| \leq \text{max}\{C,1\}$. Note that, as the inverse $C_{b}^{1}(\mathfrak{g}_{r,s},M)/\text{Ker}(\delta_{1}) \overset{\sim}{\rightarrow} \text{Ker}(\delta_{2})$ is contracting, we automatically have $C \geq 1$.
\end{proof}

\begin{defn}
Let $M$ be a Banach $\mathfrak{g}_{r,s}$ bimodule. Then we denote by $\overline{M}$ the $\mathfrak{g}_{r,s}$ module whose underlying Banach space is $M$ with action $x \cdot m := xm-mx$.
\end{defn}

\begin{lem}
\label{VanishingofHb2Consequence}
Let $M$ be a Banach $\mathfrak{g}_{r,s}$ bimodule with $\|x m\| \leq \|x\| \|m \|$ and $\|m  x\| \leq \|x\|  \|m \|$ for all $x \in \mathfrak{g}_{r,s}$, $m \in M$. Let $f:U(\mathfrak{g})_{r,s}^{\text{an}} \hat{\otimes} U(\mathfrak{g})_{r,s}^{\text{an}} \rightarrow M$ be a bounded linear map such that
$$f(1,x)=f(x,1)=0, \quad xf(y,z)-f(xy,z)+f(x,yz)-f(x,y)z=0,$$
for all $x,y,z \in U(\mathfrak{g})_{r,s}^{\text{an}}$, and $\|f\| \leq 1$. Suppose that $H_{b}^{2}(\mathfrak{g}_{r,s},\overline{M})=0$. Then there is a bounded bilinear map $\alpha:U(\mathfrak{g})_{r,s}^{\text{an}} \rightarrow M$ such that
$$\alpha(1)=0 \text{ and } f(x,y)=x\alpha(y)-\alpha(xy)+\alpha(x)y \text{ for all } x,y \in U(\mathfrak{g})_{r,s}^{\text{an}}.$$
Furthermore, if we assume that $C_{b}^{\bullet}(\mathfrak{g}_{r,s},\overline{M})$ is strictly exact at $C_{b}^{2}(\mathfrak{g}_{r,s},\overline{M})$, and the isomorphism $\text{Ker}(\delta_{2})=\text{Im}(\delta_{2}) \overset{\sim}{\rightarrow} C_{b}^{1}(\mathfrak{g}_{r,s},\overline{M})/\text{Ker}(\delta_{1})$ is contracting, then we can take $\alpha$ such that $\|\alpha\| \leq 1$ .
\end{lem}

\begin{proof}
We proceed as in Proposition XVIII.1.3. in \cite{QG}. We may define a contracting multiplication on $U(\mathfrak{g})_{r,s}^{\text{an}} \oplus M$ by
$$(x,m) \cdot (y,n) = (xy, xn+my+f(x,y)) \text{ for } x,y \in U(\mathfrak{g})_{r,s}^{\text{an}}, m,n \in M.$$
Under the commutator bracket, $\tilde{\mathfrak{g}}=\mathfrak{g}_{r,s} \oplus M \subset U(\mathfrak{g})_{r,s}^{\text{an}} \oplus M$ is a Banach Lie algebra, and an extension of $\mathfrak{g}_{r,s}$ with kernel $\overline{M}$. Thus by Lemma \ref{SplittingLemma}, there exists $s:\mathfrak{g}_{r,s} \rightarrow \tilde{\mathfrak{g}}$ splitting the projection $p:\tilde{\mathfrak{g}} \rightarrow \mathfrak{g}_{r,s}$, with $\|s\| \leq 1$ under the stronger assumption. Then the map $\mathfrak{g}_{r,s} \overset{s}{\rightarrow} \tilde{\mathfrak{g}} \rightarrow U(\mathfrak{g})_{r,s}^{\text{an}} \oplus M$ induces a unique algebra homomorphism $s':U(\mathfrak{g})_{r,s}^{\text{an}} \rightarrow U(\mathfrak{g})_{r,s}^{\text{an}} \oplus M$ which splits the first projection. This map must be of the form $s'(x)=(x,-\alpha(x))$ for some $\alpha:U(\mathfrak{g})_{r,s}^{\text{an}} \rightarrow M$. But then, for all $x,y \in U(\mathfrak{g})_{r,s}^{\text{an}}$,
$$(xy,-\alpha(xy))=s'(xy)=s'(x)s'(y)=(xy,-x\alpha(y)-\alpha(x)y+f(x,y))$$
which completes the proof.
\end{proof}

\begin{theorem}
\label{Rigidity1}
Let $\mathfrak{g}_{r,s}$ and $\mathfrak{g}'_{r',s'}$ be Banach Lie algebras, each coming from some root datum. Let $1>\varepsilon>0$. Suppose we have two morphisms of $k\{\frac{\hslash}{\epsilon}\}$ algebras $\alpha, \alpha':U(\mathfrak{g})_{r,s}^{\text{an}}\{\frac{\hslash}{\varepsilon}\} \rightarrow U(\mathfrak{g}')_{r',s'}^{\text{an}}\{\frac{\hslash}{\varepsilon}\}$ such that $\alpha \equiv \alpha'$ modulo $\hslash$ and $\|\alpha\| \leq 1$, $\|\alpha'\| \leq 1$. Suppose that $C_{b}^{\bullet}(\mathfrak{g}_{r,s},U(\mathfrak{g}')_{r',s'}^{\text{an}})$ is strictly exact at $C_{b}^{1}(\mathfrak{g}_{r,s},U(\mathfrak{g}')_{r',s'}^{\text{an}})$ and the isomorphism
$$\text{Ker}(\delta_{1})=\text{Im}(\delta_{0}) \overset{\sim}{\rightarrow} U(\mathfrak{g}')_{r',s'}^{\text{an}}/\text{Ker}(\delta_{0})$$
is contracting. Then, for any $\varepsilon>\varepsilon'>0$, there exists an invertible $F \in U(\mathfrak{g}')_{r',s'}^{\text{an}}\{\frac{\hslash}{\varepsilon'}\}$ such that $\alpha'(x)=F\alpha(x)F^{-1}$ for all $x \in U(\mathfrak{g})_{r,s}^{\text{an}}\{\frac{\hslash}{\varepsilon'}\}$ where we now view $\alpha$ and $\alpha'$ as maps from $U(\mathfrak{g})_{r,s}^{\text{an}}\{\tfrac{\hslash}{\varepsilon'}\} = U(\mathfrak{g})_{r,s}^{\text{an}}\{\tfrac{\hslash}{\varepsilon}\} \hat{\otimes}_{k\{\frac{\hslash}{\varepsilon}\}}k\{\tfrac{\hslash}{\varepsilon'}\}$ to $U(\mathfrak{g}')_{r',s'}^{\text{an}}\{\tfrac{\hslash}{\varepsilon'}\} = U(\mathfrak{g}')_{r',s'}^{\text{an}}\{\tfrac{\hslash}{\varepsilon}\} \hat{\otimes}_{k\{\frac{\hslash}{\varepsilon}\}}k\{\tfrac{\hslash}{\varepsilon'}\}$.
\end{theorem}

\begin{proof}
We proceed as in Theorem XVIII.2.1 of \cite{QG}. Fix $1>\varepsilon>\varepsilon'>0$ and a sequence $(\varepsilon_{n})_{n \geq 0}$ with
$$\varepsilon>\varepsilon_{1}>\varepsilon_{2}>...>\varepsilon_{n}>\varepsilon_{n+1}>...>\varepsilon'.$$
Since $\alpha$ is $k\{\frac{\hslash}{\varepsilon}\}$-linear, it is uniquely determined by its restriction to $U(\mathfrak{g})_{r,s}^{\text{an}}$. We may write $\alpha$ in the form
$$\alpha(x)=\sum \alpha_{i}(x)\hslash^{i} \text{ for } x \in U(\mathfrak{g})_{r,s}^{\text{an}}$$
for $\alpha_{i}:U(\mathfrak{g})_{r,s}^{\text{an}} \rightarrow U(\mathfrak{g}')_{r',s'}^{\text{an}}$, $\|\alpha_{i}\|\varepsilon^{i} \leq 1$. Now, suppose we have $u_{0},u_{1},...,u_{n} \in U(\mathfrak{g}')_{r',s'}^{\text{an}}$ such that $U_{i}\alpha \equiv \alpha_{0} U_{i}$ modulo $\hslash^{i+1}$, where
$$U_{i}=(1+u_{i}\hslash^{i})(1+u_{i-1}\hslash^{i-1})...(1+u_{0}),$$
and $\|u_{i}\|\varepsilon_{i}^{i}< 1$. Note that $(1+u_{i}\hslash^{i})$ is then invertible in $U(\mathfrak{g})_{r,s}^{\text{an}}\{\tfrac{\hslash}{\varepsilon_{n}}\}$ for each $i=0,..,n$, with inverse $\sum_{j\geq 0} (-u_{i}\hslash^{i})^{j}$, and hence so is $U_{i}$, in which case
$$\alpha^{(i)}(x):=U_{i}\alpha(x) U_{i}^{-1} \equiv \alpha_{0}(x) \text{ modulo } \hslash^{i+1}$$
for all $x \in U(\mathfrak{g})_{r,s}^{\text{an}}$. Note also that
$$\|1+u_{i}\hslash^{i}\| = \|(1+u_{i}\hslash^{i})^{-1}\| = \|U_{i}\| = \|U_{i}^{-1}\|=1,$$
and so $\|\alpha^{(n)}\| \leq \|\alpha\|$ as maps $U(\mathfrak{g})_{r,s}^{\text{an}}\{\frac{\hslash}{\varepsilon_{n}}\} \rightarrow U(\mathfrak{g}')_{r',s'}^{\text{an}}\{\frac{\hslash}{\varepsilon_{n}}\}$. Again, $\alpha^{(n)}$ is $k\{\frac{\hslash}{\varepsilon_{n}}\}$-linear, and we may write the restricted map $\alpha^{(n)}:U(\mathfrak{g})_{r,s}^{\text{an}} \rightarrow U(\mathfrak{g}')_{r',s'}^{\text{an}}\{\tfrac{\hslash}{\varepsilon_{n}}\}$ as
$$\alpha^{(n)}(x)=\sum \alpha^{(n)}_{i}(x)\hslash^{i} \text{ for } x \in U(\mathfrak{g})_{r,s}^{\text{an}}$$
for $\alpha_{i}^{(n)}:U(\mathfrak{g})_{r,s}^{\text{an}} \rightarrow U(\mathfrak{g}')_{r',s'}^{\text{an}}$, $\|\alpha_{i}^{(n)}\|\varepsilon_{n}^{i} \leq \|\alpha^{(n)}\|\leq\|\alpha\|\leq 1$. By assumption, $\alpha_{0}^{(n)}=\alpha_{0}$ and $\alpha_{i}^{(n)}=0$ for $i=1,...,n$. Looking at the $\hslash^{n+1}$ coefficient of $\alpha^{(n)}(xy)=\alpha^{(n)}(x)\alpha^{(n)}(y)$ we see that
$$\begin{array}{rcl}
\alpha^{(n)}_{n+1}(xy)&=&\alpha^{(n)}_{0}(x)\alpha^{(n)}_{n+1}(y)+\alpha^{(n)}_{n+1}(x)\alpha^{(n)}_{0}(y)\\
&=&\alpha_{0}(x)\alpha^{(n)}_{n+1}(y)+\alpha^{(n)}_{n+1}(x)\alpha_{0}(y)
\end{array}$$
for all $x,y \in U(\mathfrak{g})_{r,s}^{\text{an}}$. Thus
$$\alpha^{(n)}_{n+1}([x,y])=[\alpha_{0}(x),\alpha^{(n)}_{n+1}(y)]-[\alpha_{0}(y),\alpha^{(n)}_{n+1}(x)]$$ for all $x,y \in \mathfrak{g}_{r,s}$. Given that $x \in \mathfrak{g}_{r,s}$ acts on $U(\mathfrak{g}')_{r',s'}^{\text{an}}$ via $[\alpha_{0}(x),-]$, this is precisely the fact that $\alpha_{n+1}^{(n)}$ restricted to $\mathfrak{g}_{r,s}$ is in $\text{Ker}(\delta_{1})$. Hence there is a $u_{n+1} \in U(\mathfrak{g}')_{r',s'}^{\text{an}}$ such that $\alpha^{(n)}_{n+1}(x)=[\alpha_{0}(x),u_{n+1}]$ for all $x \in \mathfrak{g}_{r,s}$ and $\|u_{n+1}\| \leq \|\alpha_{n+1}^{(n)}\|$, so that $\|u_{n+1}\|\varepsilon_{n}^{n+1} \leq 1$. Then, in $U(\mathfrak{g}')_{r',s'}^{\text{an}}\{\tfrac{\hslash}{\varepsilon_{n+1}}\}$, $\|u_{n+1}\|\varepsilon_{n+1}^{n+1} < 1$ and $(1+u_{n+1}\hslash^{n+1})$ is invertible with
$$\begin{array}{rcll}
\alpha^{(n+1)}(x)&:=& (1+u_{n+1}\hslash^{n+1})\alpha(x)(1+u_{n+1}\hslash^{n+1})^{-1}&\\
&\equiv& \alpha_{0}(x) + (u_{n+1}\alpha_{0}(x)-\alpha_{0}(x)u_{n+1}+\alpha_{n+1}^{(n)}(x))\hslash^{n+1} &\text{ mod } \hslash^{n+2}\\
&\equiv&\alpha_{0}(x) &\text{ mod } \hslash^{n+2}.
\end{array}$$
Taking $u_{0}:=0$ as our base case, we obtain inductively sequences $(u_{n})_{n \geq 0}$ and $(U_{n})_{n \geq 0}$ in $U(\mathfrak{g}')_{r',s'}^{\text{an}}\{\tfrac{\hslash}{\varepsilon'}\}$. The sequence $(U_{n})_{n \geq 0}$ converges to $U:=1+\sum_{n=1}^{\infty} v_{n}\hslash^{n}$ where $v_{n}=\sum u_{i_{1}}u_{i_{2}}...u_{i_{k}}$ whose sum is taken over all finite sequences $i_{1}>i_{2}>...>i_{k}$ with $i_{1}+i_{2}+...+i_{k}=n$. Since $\|v_{n}\|(\varepsilon')^{n}<1$ for all $n \geq 0$, $U$ is invertible in $U(\mathfrak{g}')_{r',s'}^{\text{an}}\{\tfrac{\hslash}{\varepsilon'}\}$ with inverse $U^{-1}=\sum_{i=0}^{\infty}(-\sum_{n=1}^{\infty} v_{n}\hslash^{n})^{i}$. It follows from the fact that
$$U_{n}\alpha(x) U_{n}^{-1} \equiv \alpha_{0}(x) \text{ modulo } \hslash^{n+1}$$
for each $n\geq 0$ that $U\alpha(x)U^{-1}=\alpha_{0}(x)$ for each $x \in U(\mathfrak{g})_{r,s}^{\text{an}}\{\tfrac{\hslash}{\varepsilon'}\}$. Similarly, there exist mutual inverses $U'$, $U'^{-1}$ in $U(\mathfrak{g}')_{r',s'}^{\text{an}}\{\tfrac{\hslash}{\varepsilon'}\}$ such that $U'\alpha'(x)U'^{-1}=\alpha'_{0}(x)=\alpha_{0}(x)$ for each $x \in U(\mathfrak{g})_{r,s}^{\text{an}}\{\tfrac{\hslash}{\varepsilon'}\}$. Thus taking $F=U'^{-1}U$ gives our result.
\end{proof}

\begin{theorem}
\label{Rigidity2}
Let $\varepsilon >0$. Suppose $A$ is a Banach $k\{\frac{\hslash}{\varepsilon}\}$ algebra with contracting multiplication such that there is a bounded $k\{\frac{\hslash}{\varepsilon}\}$-linear isomorphism of Banach spaces $A \cong U(\mathfrak{g})_{r,s}^{\text{an}}\{\frac{\hslash}{\varepsilon}\}$ that preserves the unit, and that $A/\hslash A \cong U(\mathfrak{g})_{r,s}^{\text{an}}$ is an isomorphism of Banach algebras. Suppose that $C_{b}^{\bullet}(\mathfrak{g}_{r,s},U(\mathfrak{g})_{r,s}^{\text{an}})$ is strictly exact at $C_{b}^{2}(\mathfrak{g}_{r,s},U(\mathfrak{g})_{r,s}^{\text{an}})$ and the isomorphism
$$\text{Ker}(\delta_{2})=\text{Im}(\delta_{1}) \overset{\sim}{\rightarrow} C_{b}^{1}(\mathfrak{g}_{r,s},U(\mathfrak{g})_{r,s}^{\text{an}})/\text{Ker}(\delta_{1})$$
is contracting. Then for any $\varepsilon > \varepsilon' > 0$, there is an isomorphism of $k\{\frac{\hslash}{\varepsilon'}\}$ algebras
$$\alpha:A_{\varepsilon'}:=A \hat{\otimes}_{k\{\frac{\hslash}{\varepsilon}\}}k\{\tfrac{\hslash}{\varepsilon'}\} \overset{\sim}{\rightarrow} U(\mathfrak{g})_{r,s}^{\text{an}}\{\tfrac{\hslash}{\varepsilon'}\}$$
inducing the given isomorphism $A_{\varepsilon'}/\hslash A_{\varepsilon'} \cong A/\hslash A \cong U(\mathfrak{g})_{r,s}^{\text{an}}$.
\end{theorem}

\begin{proof}
We proceed as in Theorem XVIII.2.2 of \cite{QG}. As before, fix $1>\varepsilon>\varepsilon'>0$ and a sequence $(\varepsilon_{n})_{n \geq 0}$ with
$$\varepsilon>\varepsilon_{1}>\varepsilon_{2}>...>\varepsilon_{n}>\varepsilon_{n+1}>...>\varepsilon'.$$
The given isomorphism $A \cong U(\mathfrak{g})_{r,s}^{\text{an}}\{\frac{\hslash}{\varepsilon}\}$ induces a multiplication
$$\mu:(U(\mathfrak{g})_{r,s}^{\text{an}} \hat{\otimes} U(\mathfrak{g})_{r,s}^{\text{an}})\{\tfrac{\hslash}{\varepsilon}\} \cong U(\mathfrak{g})_{r,s}^{\text{an}}\{\tfrac{\hslash}{\varepsilon}\} \hat{\otimes} U(\mathfrak{g})_{r,s}^{\text{an}}\{\tfrac{\hslash}{\varepsilon}\} \rightarrow U(\mathfrak{g})_{r,s}^{\text{an}}\{\tfrac{\hslash}{\varepsilon}\}$$
which is $k\{\frac{\hslash}{\varepsilon}\}$-linear and hence is determined by its restriction to $U(\mathfrak{g})_{r,s}^{\text{an}} \hat{\otimes} U(\mathfrak{g})_{r,s}^{\text{an}}$. We may write this restriction as
$$\mu = \sum \mu_{i} \hslash^{i}, \quad \mu_{i}:U(\mathfrak{g})_{r,s}^{\text{an}} \hat{\otimes} U(\mathfrak{g})_{r,s}^{\text{an}} \rightarrow U(\mathfrak{g})_{r,s}^{\text{an}}, \quad \|\mu_{i}\|\varepsilon^{i} \leq 1.$$
Suppose we have maps $\alpha_{0},\alpha_{1},...,\alpha_{n}:U(\mathfrak{g})_{r,s}^{\text{an}} \rightarrow U(\mathfrak{g})_{r,s}^{\text{an}}$ such that $\|\alpha_{i}\|\varepsilon_{i}^{i}< 1$ and $V_{i}(\mu(x,y)) \equiv \mu_{0}(V_{i}(x),V_{i}(y))$ modulo $\hslash^{i+1}$ for all $x,y \in U(\mathfrak{g})_{r,s}^{\text{an}}$, where $V_{i}$ is the endomorphism of $U(\mathfrak{g})_{r,s}^{\text{an}}\{\frac{\hslash}{\varepsilon}\}$ with
$$V_{i}|_{U(\mathfrak{g})_{r,s}^{\text{an}}}=(\text{Id}+\hslash^{i}\alpha_{i})(\text{Id}+\hslash^{i-1}\alpha_{i-1})...(\text{Id}+\alpha_{0}).$$
Note that, since $\|\alpha_{i}\|\varepsilon_{n}^{i} < \|\alpha_{i}\|\varepsilon_{i}^{i}\leq 1$, the endomorphism of $U(\mathfrak{g})_{r,s}^{\text{an}}\{\frac{\hslash}{\varepsilon_{n}}\}$ whose restriction to $U(\mathfrak{g})_{r,s}^{\text{an}}$ is $(\text{Id}+\hslash^{i}\alpha_{i})$ is then invertible for each $i=0,..,n$, with inverse $\sum_{j\geq 0} (-\hslash^{i}\alpha_{i})^{j}$, and hence so is $V_{i}$ as an endomorphism of $U(\mathfrak{g})_{r,s}^{\text{an}}\{\frac{\hslash}{\varepsilon_{n}}\}$, in which case
$$\mu^{(i)}(x,y):=V_{i}(\mu(V_{i}^{-1}(x),V_{i}^{-1}(y))) \equiv \mu_{0}(x) \text{ modulo } \hslash^{i+1}$$
for all $x,y \in U(\mathfrak{g})_{r,s}^{\text{an}}$. Note also that
$$\|1+\hslash^{i}\alpha_{i}\| = \|(1+\hslash^{i}\alpha_{i})^{-1}\| = \|V_{i}\| = \|V_{i}^{-1}\|=1,$$
and so $\|\mu^{(n)}\| \leq \|\mu\|$ as multiplication maps on $U(\mathfrak{g})_{r,s}^{\text{an}}\{\frac{\hslash}{\varepsilon_{n}}\}$.  Again, $\mu^{(n)}$ is $k\{\frac{\hslash}{\varepsilon_{n}}\}$-linear, and we may write
$$\mu^{(n)}(x,y)=\sum \mu^{(n)}_{i}(x,y)\hslash^{i} \text{ for } x,y \in U(\mathfrak{g})_{r,s}^{\text{an}}$$
where $\mu_{i}^{(n)}:U(\mathfrak{g})_{r,s}^{\text{an}}\hat{\otimes}U(\mathfrak{g})_{r,s}^{\text{an}} \rightarrow U(\mathfrak{g})_{r,s}^{\text{an}}$, $\|\mu_{i}^{(n)}\|\varepsilon_{n}^{i} \leq \|\mu^{(n)}\|\leq\|\mu\|\leq 1$. By assumption, $\mu_{0}^{(n)}=\mu_{0}$ and $\mu_{i}^{(n)}=0$ for $i=1,...,n$. Looking at the $\hslash^{n+1}$ coefficient of $\mu^{(n)}(\mu^{(n)}(x,y),z)=\mu^{(n)}(x,\mu^{(n)}(y,z))$ we see that
$$\mu^{(n)}_{n+1}(xy,z)+\mu_{n+1}^{(n)}(x,y)z=\mu^{(n)}_{n+1}(x,yz)+x\mu^{(n)}_{n+1}(y,z)$$
for all $x,y,z \in U(\mathfrak{g})_{r,s}^{\text{an}}$. Here, we are using the simplified notation $xy:=\mu_{0}(x,y)$. So, by Lemma \ref{VanishingofHb2Consequence}, there exists $\alpha_{n+1}:U(\mathfrak{g})_{r,s}^{\text{an}} \rightarrow U(\mathfrak{g})_{r,s}^{\text{an}}$ such that $\|\alpha_{n+1}\| \leq \|\mu^{(n)}_{n+1}\|$, $\alpha_{n+1}(1)=0$ and
$$\mu_{n+1}^{(n)}(x,y)=x\alpha_{n+1}(y)-\alpha_{n+1}(xy)+\alpha_{n+1}(x)y.$$
Setting $V_{n+1}$ as the endomorphism of $U(\mathfrak{g})_{r,s}^{\text{an}}\{\frac{\hslash}{\varepsilon_{n+1}}\}$ whose restriction to $U(\mathfrak{g})_{r,s}^{\text{an}}$ is $(\text{Id}+\hslash^{n+1}\alpha_{n+1})V_{n}$, we have that $V_{n+1}$ is invertible since $\|\alpha_{n+1}\|\varepsilon_{n+1}^{n+1} < 1$. Let
$$\mu^{(n+1)}(x,y):=V_{n+1}(\mu(V_{n+1}^{-1}(x),V_{n+1}^{-1}(y))).$$
Then, modulo $\hslash^{n+2}$,
$$\begin{array}{rcl}
\mu^{(n+1)}(x,y) &\equiv& (\text{Id}+\alpha_{n+1}\hslash^{n+1}) \circ (\mu_{0}+\mu^{(n)}_{n+1}\hslash^{n+1})\\
&& \qquad (x-\alpha_{n+1}(x)\hslash^{n+1},y-\alpha_{n+1}(y)\hslash^{n+1})\\
&\equiv& xy + (\alpha_{n+1}(xy) + \mu^{(n)}_{n+1}(x,y)-\alpha_{n+1}(x)y-x\alpha_{n+1}(y))\hslash^{n+1}\\
&\equiv& xy.
\end{array}$$
Taking $\alpha_{0}=0$ as a base case, we inductively obtain sequences $(\alpha_{n})_{n \geq 0}$ and $(V_{n})_{n \geq 0}$. The sequence $(V_{n})_{n \geq 0}$ converges on $U(\mathfrak{g})_{r,s}^{\text{an}}\{\frac{\hslash}{\varepsilon'}\}$ to $V:=\text{Id}+\sum_{n=1}^{\infty} \hslash^{n}\beta_{n}$ where $\beta_{n}=\sum \alpha_{i_{1}}\alpha_{i_{2}}...\alpha_{i_{k}}$ whose sum is taken over all finite sequences $i_{1}>i_{2}>...>i_{k}$ with $i_{1}+i_{2}+...+i_{k}=n$. Since $\|\alpha_{n}\|(\varepsilon')^{n}<1$, $V$ is invertible on $U(\mathfrak{g})_{r,s}^{\text{an}}\{\tfrac{\hslash}{\varepsilon'}\}$ with inverse
$$V^{-1}=\sum_{i=0}^{\infty}(-\sum_{n=1}^{\infty} \hslash^{n}\beta_{n})^{i}.$$
It follows from the fact that $V_{n}\mu(V_{n}^{-1}(x),V_{n}^{-1}(y)) \equiv \mu_{0}(x,y) \text{ modulo } \hslash^{n+1}$ for each $n\geq 0$ that $V\mu(V^{-1}(x),V^{-1}(y)) = \mu_{0}(x,y)$ for each $x,y \in U(\mathfrak{g})_{r,s}^{\text{an}}\{\tfrac{\hslash}{\varepsilon'}\}$. It then follows that
$$A_{\varepsilon'} \cong U(\mathfrak{g})_{r,s}^{\text{an}}\{\tfrac{\hslash}{\varepsilon'}\} \overset{V}{\longrightarrow} U(\mathfrak{g})_{r,s}^{\text{an}}\{\tfrac{\hslash}{\varepsilon'}\}$$
is our desired isomorphism of algebras.
\end{proof}

The following are slight variations of the above theorems.

\begin{corollary}
\label{Rigidity1a}
Let $\mathfrak{g}_{r,s}$ and $\mathfrak{g}'_{r',s'}$ be Banach Lie algebras, each coming from some root datum. Let $\varepsilon \geq 0$. Suppose we have two morphisms of $k\{\frac{\hslash}{\varepsilon}\}^{\dagger}$ algebras $\alpha, \alpha':U(\mathfrak{g})_{r,s}^{\text{an}}\{\frac{\hslash}{\varepsilon}\}^{\dagger} \rightrightarrows U(\mathfrak{g}')_{r',s'}^{\text{an}}\{\frac{\hslash}{\varepsilon}\}^{\dagger}$ such that $\alpha \equiv \alpha'$ modulo $\hslash$. Suppose that $C_{b}^{\bullet}(\mathfrak{g}_{r,s},U(\mathfrak{g}')_{r',s'}^{\text{an}})$ is strictly exact at $C_{b}^{1}(\mathfrak{g}_{r,s},U(\mathfrak{g}')_{r',s'}^{\text{an}})$ and the isomorphism
$$\text{Ker}(\delta_{1})=\text{Im}(\delta_{0}) \overset{\sim}{\rightarrow} U(\mathfrak{g}')_{r',s'}^{\text{an}}/\text{Ker}(\delta_{0})$$
is contracting. Then there exists a convolution invertible generalised element $F:k \rightarrow U(\mathfrak{g}')_{r',s'}^{\text{an}}\{\frac{\hslash}{\varepsilon}\}^{\dagger}$ such that $\alpha'=F \ast \alpha \ast F^{-1}$.
\end{corollary}

\begin{proof}
The morphisms $\alpha$ and $\alpha'$ are $k\{\frac{\hslash}{\varepsilon}\}^{\dagger}$-linear, so are determined by their restrictions $U(\mathfrak{g})_{r,s}^{\text{an}} \rightrightarrows U(\mathfrak{g}')_{r',s'}^{\text{an}}\{\frac{\hslash}{\varepsilon}\}^{\dagger}$. Since $U(\mathfrak{g})_{r,s}^{\text{an}}$ is Banach, there is $\varepsilon'>\varepsilon$ such that the restrictions of $\alpha$ and $\alpha'$ are determined by morphisms of Banach algebras $U(\mathfrak{g})_{r,s}^{\text{an}} \rightrightarrows U(\mathfrak{g}')_{r',s'}^{\text{an}}\{\frac{\hslash}{\varepsilon'}\}$. By the proof of Theorem \ref{Rigidity1} there is $\varepsilon' > \varepsilon'' > \varepsilon$ and an invertible element $F \in U(\mathfrak{g}')_{r',s'}^{\text{an}}\{\frac{\hslash}{\varepsilon''}\}$ such that $\alpha'(x)=F\alpha(x)F^{-1} \in U(\mathfrak{g}')_{r',s'}^{\text{an}}\{\frac{\hslash}{\varepsilon''}\}$ for all $x \in U(\mathfrak{g})_{r,s}^{\text{an}}$. It then follows that $\alpha = F \ast \alpha' \ast F^{-1}$ as maps $U(\mathfrak{g})_{r,s}^{\text{an}}\{\frac{\hslash}{\varepsilon}\}^{\dagger} \rightrightarrows U(\mathfrak{g}')_{r',s'}^{\text{an}}\{\frac{\hslash}{\varepsilon}\}^{\dagger}$, where we denote by $F$ the generalised element $k \xrightarrow{1 \mapsto F} U(\mathfrak{g}')_{r',s'}^{\text{an}}\{\frac{\hslash}{\varepsilon''}\} \rightarrow U(\mathfrak{g}')_{r',s'}^{\text{an}}\{\frac{\hslash}{\varepsilon}\}^{\dagger}$.
\end{proof}

\begin{corollary}
\label{Rigidity2a}
Let $\varepsilon \geq 0$. Suppose we have a Banach $k\{\frac{\hslash}{\varepsilon'}\}$-algebra $A_{\varepsilon'}$, for some $\varepsilon'>\varepsilon$, equipped with a $k\{\frac{\hslash}{\varepsilon'}\}$-linear isomorphism $A_{\varepsilon'} \cong U(\mathfrak{g})_{r,s}^{\text{an}}\{\frac{\hslash}{\varepsilon'}\}$ that preserves the unit such that $A_{\varepsilon'}/\hslash A_{\varepsilon'} \cong U(\mathfrak{g})_{r,s}^{\text{an}}$ is an isomorphism of Banach algebras. Suppose further that $C_{b}^{\bullet}(\mathfrak{g}_{r,s},U(\mathfrak{g})_{r,s}^{\text{an}})$ is strictly exact at $C_{b}^{2}(\mathfrak{g}_{r,s},U(\mathfrak{g})_{r,s}^{\text{an}})$ and the isomorphism
$$\text{Ker}(\delta_{2})=\text{Im}(\delta_{1}) \overset{\sim}{\rightarrow} C_{b}^{1}(\mathfrak{g}_{r,s},U(\mathfrak{g})_{r,s}^{\text{an}})/\text{Ker}(\delta_{1})$$
is contracting. Then there is an isomorphism of $k\{\frac{\hslash}{\varepsilon}\}^{\dagger}$ algebras
$$\alpha:A \overset{\sim}{\rightarrow} U(\mathfrak{g})_{r,s}^{\text{an}}\{\tfrac{\hslash}{\varepsilon}\}^{\dagger},$$
where $A=k\{\frac{\hslash}{\varepsilon}\}^{\dagger} \hat{\otimes}_{k\{\frac{\hslash}{\varepsilon'}\}} A_{\varepsilon'}$, inducing the given isomorphism $A/\hslash A \cong A_{\varepsilon'}/\hslash A_{\varepsilon'} \cong U(\mathfrak{g})_{r,s}^{\text{an}}$.
\end{corollary}

\begin{proof}
By Theorem \ref{Rigidity2} there is $\varepsilon' > \varepsilon'' > \varepsilon$ and an isomorphism of $k\{\frac{\hslash}{\varepsilon''}\}$-algebras
$$\alpha_{\varepsilon''}:A_{\varepsilon''} \overset{\sim}{\rightarrow} U(\mathfrak{g})_{r,s}^{\text{an}}\{\tfrac{\hslash}{\varepsilon''}\},$$
where $A_{\varepsilon''}=k\{\frac{\hslash}{\varepsilon''}\} \hat{\otimes}_{k\{\frac{\hslash}{\varepsilon'}\}} A_{\varepsilon'}$, inducing the given isomorphism $A_{\varepsilon''}/\hslash A_{\varepsilon''} \cong A_{\varepsilon'}/\hslash A_{\varepsilon'} \cong U(\mathfrak{g})_{r,s}^{\text{an}}$. Then $\alpha$ is the composition
$$A \cong k\{ \tfrac{\hslash}{\varepsilon}\}^{\dagger} \hat{\otimes}_{k\{\frac{\hslash}{\varepsilon''}\}} A_{\varepsilon''} \xrightarrow{\text{Id} \otimes \alpha_{\varepsilon''}} k\{\tfrac{\hslash}{\varepsilon}\}^{\dagger} \hat{\otimes}_{k\{\frac{\hslash}{\varepsilon''}\}} U(\mathfrak{g})_{r,s}^{\text{an}}\{\tfrac{\hslash}{\varepsilon''}\} \cong U(\mathfrak{g})_{r,s}^{\text{an}}\{\tfrac{\hslash}{\varepsilon}\}^{\dagger}.$$
\end{proof}

\begin{remark}
If we let $k=\mathbb{C}$ with the trivial valuation then $k\{\frac{\hslash}{\varepsilon}\}=k\{\frac{\hslash}{\varepsilon}\}^{\dagger}=\mathbb{C}\llbracket\hslash\rrbracket$ and we recover the classical rigidity results as stated in Section XVIII.2 of \cite{QG}.
\end{remark}
\

We may construct analytic quantum groups over $k\{\frac{\hslash}{\varepsilon}\}$ and $k\{\frac{\hslash}{\varepsilon}\}^{\dagger}$ for a formal parameter $\hslash$, where $q=e^{\hslash}$, to which the above deformation theory applies. For the remainder of this section, assume that $\varepsilon>0$ is sufficiently small such that $\text{exp}(\hslash)$ converges in $k\{\frac{\hslash}{\varepsilon}\}$.

\begin{defn}
Let $\mathscr{H}_{\frac{\hslash}{\varepsilon}}:=\coprod^{\leq 1}_{\alpha \in \mathbb{Z}I}k\{\frac{\hslash}{\varepsilon}\} \cdot H_{\alpha}$ be the Banach Hopf algebra over $k\{\frac{\hslash}{\varepsilon}\}$ generated by $H_{i}$ for $i \in I$, where
$$\Delta_{\mathscr{H}}(H_{i})=1 \otimes H_{i} + H_{i} \otimes 1.$$
Let $V_{\frac{\hslash}{\varepsilon}}:=\coprod^{\leq 1}_{i \in I}k\{\frac{\hslash}{\varepsilon}\} \cdot v_{i}$ and define
$$c:V_{\frac{\hslash}{\varepsilon}} \hat{\otimes}_{k\{\frac{\hslash}{\varepsilon}\}} V_{\frac{\hslash}{\varepsilon}} \rightarrow V_{\frac{\hslash}{\varepsilon}} \hat{\otimes}_{k\{\frac{\hslash}{\varepsilon}\}} V_{\frac{\hslash}{\varepsilon}}, \quad v_{i} \otimes v_{j} \mapsto q^{\lambda_{i}(\alpha_{j})} v_{j} \otimes v_{i},$$
where $q=e^{\hslash}$. Let $T_{r}(V_{\frac{\hslash}{\varepsilon}})$ be the resulting braided analytically graded Hopf algebra on $\coprod_{n \geq 0}^{\leq 1} (V_{\frac{\hslash}{\varepsilon}}^{\hat{\otimes} n})_{r^{n}}$ defined using the tensor product $\hat{\otimes}_{k\{\frac{\hslash}{\varepsilon}\}}$ over $k\{\tfrac{\hslash}{\varepsilon}\}$. We define a bilinear form $\langle-,- \rangle: V_{\frac{\hslash}{\varepsilon}} \hat{\otimes}_{k\{\frac{\hslash}{\varepsilon}\}} V_{\frac{\hslash}{\varepsilon}} \rightarrow k\{\frac{\hslash}{\varepsilon}\}$ by
$$\langle v_{i},v_{j} \rangle = \delta_{i,j} \frac{\hslash}{(q_{i}-q_{i}^{-1})}$$
where $q_{i}=q^{\frac{(\alpha_{i},\alpha_{i})}{2}}$. By Lemma \ref{ExtendBilinearForm} this extends to a bilinear form on $T_{r}(V_{\frac{\hslash}{\varepsilon}})$. Let $\mathfrak{B}_{r}(V_{\frac{\hslash}{\varepsilon}})$ be the quotient of $T_{r}(V_{\frac{\hslash}{\varepsilon}})$ by the radical of this bilinear form, which again is a braided analytically graded Banach Hopf algebra.
\end{defn}

\begin{remark}
Note that
$$q_{i}-q_{i}^{-1}=(\alpha_{i},\alpha_{i})\hslash \left(\sum_{k=0}^{\infty}\frac{(\frac{(\alpha_{i},\alpha_{i})}{2}\hslash)^{2k}}{k!}\right)$$
is not invertible in $k\{\frac{\hslash}{\varepsilon}\}$, but $\frac{1}{\hslash}(q_{i}-q_{i}^{-1})$ is. Thus we have had to rescale the inner product from Definition \ref{AnalyticQuantumGroupPositivePart} in order to define it over $k\{\frac{\hslash}{\varepsilon}\}$.
\end{remark}

\begin{theorem}
Let $r,s>0$ such that $1 \leq |q_{i}-q_{i}^{-1}|rs$. Then there is a Banach algebra structure on
$$U_{\frac{\hslash}{\varepsilon}}(\mathfrak{g})_{r,s}^{\text{an}}:=\mathfrak{B}_{r}(V_{\frac{\hslash}{\varepsilon}}) \hat{\otimes}_{k\{\frac{\hslash}{\varepsilon}\}} \mathscr{H}_{\frac{\hslash}{\varepsilon}} \hat{\otimes}_{k\{\frac{\hslash}{\varepsilon}\}} \mathfrak{B}_{s}(V_{\frac{\hslash}{\varepsilon}})$$
such that $\mathfrak{B}_{r}(V_{\frac{\hslash}{\varepsilon}})$, $\mathscr{H}_{\frac{\hslash}{\varepsilon}}$ and $\mathfrak{B}_{s}(V_{\frac{\hslash}{\varepsilon}})$ are all subalgebras, and
$$[H_{i},E_{j}]=(\alpha_{i},\alpha_{j})E_{j}, \quad [H_{i},F_{j}]=-(\alpha_{i},\alpha_{j})F_{j}, \quad [E_{i},F_{j}]=\delta_{i,j} \frac{t_{i}-t_{i}^{-1}}{q_{i}-q_{i}^{-1}},$$
where $F_{i}=v_{i} \otimes 1 \otimes 1$, $E_{i}=1 \otimes 1 \otimes v_{i}$ and $t_{i}=\text{exp}(\frac{(\alpha_{i},\alpha_{i})}{2} \hslash H_{i})$. This becomes a Banah Hopf algebra with $\mathscr{H}_{\frac{\hslash}{\varepsilon}}$ as a sub-Hopf algebra and
$$\begin{array}{rclrcl}
\Delta(E_{i})&=&E_{i} \otimes t_{i} + 1 \otimes E_{i},& S(E_{i})&=&-E_{i}t_{i}^{-1},\\
\Delta(F_{i}) &=& F_{i} \otimes 1 + t_{i}^{-1} \otimes F_{i},& S(F_{i})&=&-t_{i}F_{i}.
\end{array}$$
\end{theorem}

\begin{proof}
The fact that this is a Hopf algebra is checked on a dense subspace in Proposition 7 of Section 6.1.3 of \cite{QGaTR}. We must check that it  extends continuously to $U_{\frac{\hslash}{\varepsilon}}(\mathfrak{g})_{r,s}^{\text{an}}$. By construction
$$U_{\frac{\hslash}{\varepsilon}}^{\leq 1}(\mathfrak{g})_{r,s}^{\text{an}}:=\mathfrak{B}_{r}(V_{\frac{\hslash}{\varepsilon}}) \rtimes \mathscr{H}_{\frac{\hslash}{\varepsilon}} \quad \text{and} \quad U_{\frac{\hslash}{\varepsilon}}^{\geq 1}(\mathfrak{g})_{r,s}^{\text{an}}:=\mathscr{H}_{\frac{\hslash}{\varepsilon}} \ltimes \mathfrak{B}_{s}(V_{\frac{\hslash}{\varepsilon}})$$
sit as sub-Hopf algebras in $U_{\frac{\hslash}{\varepsilon}}(\mathfrak{g})_{r,s}^{\text{an}}$, it is enough to check that the restriction of the multiplication map
$$\mathfrak{B}_{s}(V_{\frac{\hslash}{\varepsilon}}) \hat{\otimes} \mathfrak{B}_{r}(V_{\frac{\hslash}{\varepsilon}}) \rightarrow U_{\frac{\hslash}{\varepsilon}}(\mathfrak{g})_{r,s}^{\text{an}}, \quad E_{i} \otimes F_{j} \mapsto F_{i}E_{j} + \delta_{i,j} \frac{t_{i}-t_{i}^{-1}}{q_{i}-q_{i}^{-1}},$$
is continuous. This follows from the assumption that
$$\left\|\frac{t_{i}-t_{i}^{-1}}{q_{i}-q_{i}^{-1}}\right\|=|q_{i}-q_{i}^{-1}|^{-1} \leq rs.$$
\end{proof}

\begin{theorem}
\label{QuantumGroupModuloh}
Let $r,s>0$ such that $1 \leq |q_{i}-q_{i}^{-1}|rs$. Then there is an isomorphism of $k\{\frac{\hslash}{\varepsilon}\}$-modules $U_{\frac{\hslash}{\varepsilon}}(\mathfrak{g})_{r,s}^{\text{an}} \cong U(\mathfrak{g})_{r,s}^{\text{an}}\{\frac{\hslash}{\varepsilon}\}$ that descends to an isomorphism of Banach Hopf algebras $U_{\frac{\hslash}{\varepsilon}}(\mathfrak{g})_{r,s}^{\text{an}} / \hslash U_{\frac{\hslash}{\varepsilon}}(\mathfrak{g})_{r,s}^{\text{an}} \cong U(\mathfrak{g})_{r,s}^{\text{an}}$.
\end{theorem}

\begin{proof}
By Theorem 33.1.3 of \cite{ItQG} and Proposition \ref{WeakClassicalNicholsAlgebrasDense}, $\mathfrak{B}_{r}(V_{\frac{\hslash}{\varepsilon}})$ is the quotient of $T_{r}(V_{\frac{\hslash}{\varepsilon}})$ by the closed homogeneous ideal generated by the quantum Serre relations
$$\sum_{k=0}^{1-(\alpha_{i},\alpha_{j})}(-1)^{k} \frac{[1-(\alpha_{i},\alpha_{j})]_{q}}{[k]_{q}[1-(\alpha_{i},\alpha_{j})-k]_{q}} v_{i}^{1-(\alpha_{i},\alpha_{j})-k}v_{j}v_{i}^{k}=0$$
for $i \neq j$. Since $V_{\frac{\hslash}{\varepsilon}} \rightarrow V_{0}\{\frac{\hslash}{\varepsilon}\}$, $v_{i} \mapsto v_{i}$, is an isomorphism there is an isomorphism of $k\{\frac{\hslash}{\varepsilon}\}$-modules $T_{r}(V_{\frac{\hslash}{\varepsilon}}) \cong T_{r}(V_{0})\{\frac{\hslash}{\varepsilon}\}$. By the Poincar\'e-Birkhoff-Witt Theorem this descends to isomorphisms between the graded pieces $\mathfrak{B}_{r}(V_{\frac{\hslash}{\varepsilon}})(n) \cong U^{-}(\mathfrak{g})_{r}^{\text{an}}(n)\{\frac{\hslash}{\varepsilon}\}$, hence $\mathfrak{B}_{r}(V_{\frac{\hslash}{\varepsilon}}) \cong U^{-}(\mathfrak{g})_{r}^{\text{an}}\{\frac{\hslash}{\varepsilon}\}$. Likewise $\mathscr{H}_{\frac{\hslash}{\varepsilon}} \cong \mathscr{H}_{0}\{\frac{\hslash}{\varepsilon}\}$. So, as Banach spaces,
$$\mathfrak{B}_{r}(V_{\frac{\hslash}{\varepsilon}})/\hslash \mathfrak{B}_{r}(V_{\frac{\hslash}{\varepsilon}}) \cong U^{-}(\mathfrak{g})_{r}^{\text{an}} \quad \text{and} \quad \mathscr{H}_{\frac{\hslash}{\varepsilon}}/\hslash \mathscr{H}_{\frac{\hslash}{\varepsilon}} \cong \mathscr{H}_{0},$$
and so $U_{\frac{\hslash}{\varepsilon}}(\mathfrak{g})_{r,s}^{\text{an}} / \hslash U_{\frac{\hslash}{\varepsilon}}(\mathfrak{g})_{r,s}^{\text{an}} \cong U(\mathfrak{g})_{r,s}^{\text{an}}$. By Remark 4 of Section 6.1.3 this restricts to a Hopf algebra isomorphism on a dense subspace, hence is a Hopf algebra isomorphism by continuity.
\end{proof}

\begin{defn}
Let $r,s>0$ such that $1 \leq |q_{i}-q_{i}^{-1}|rs$ and let $\varepsilon >0$. Then for any $\varepsilon'>\varepsilon$ we define
$$U_{\frac{\hslash}{\varepsilon}}^{\dagger}(\mathfrak{g})_{r,s}^{\text{an}}:=U_{\frac{\hslash}{\varepsilon'}}(\mathfrak{g})_{r,s}^{\text{an}} \hat{\otimes}_{k\{\frac{\hslash}{\varepsilon'}\}} k\{\hslash/\varepsilon\}^{\dagger}.$$
\end{defn}

\begin{corollary}
\label{QuantumGroupModuloDaggerh}
Let $r,s>0$ such that $1 \leq |q_{i}-q_{i}^{-1}|rs$. Then there is an isomorphism of $k\{\frac{\hslash}{\varepsilon}\}^{\dagger}$-modules $U_{\frac{\hslash}{\varepsilon}}^{\dagger}(\mathfrak{g})_{r,s}^{\text{an}} \cong U(\mathfrak{g})_{r,s}^{\text{an}}\{\frac{\hslash}{\varepsilon}\}^{\dagger}$ that descends to an isomorphism of Banach Hopf algebras $U^{\dagger}_{\frac{\hslash}{\varepsilon}}(\mathfrak{g})_{r,s}^{\text{an}} / \hslash U^{\dagger}_{\frac{\hslash}{\varepsilon}}(\mathfrak{g})_{r,s}^{\text{an}} \cong U(\mathfrak{g})_{r,s}^{\text{an}}$.
\end{corollary}

\begin{proof}
This follows from Theorem \ref{QuantumGroupModuloh}.
\end{proof}

\begin{corollary}
\label{RigidityApplied}
Suppose that the complex $C_{b}^{\bullet}(\mathfrak{g}_{r,s},U(\mathfrak{g})_{r,s}^{\text{an}})$ is strictly exact at both $C_{b}^{1}(\mathfrak{g}_{r,s},U(\mathfrak{g})_{r,s}^{\text{an}})$ and $C_{b}^{2}(\mathfrak{g}_{r,s},U(\mathfrak{g})_{r,s}^{\text{an}})$, and that the isomorphisms
$$\text{Ker}(\delta_{1})=\text{Im}(\delta_{0}) \overset{\sim}{\longrightarrow} U(\mathfrak{g})_{r,s}^{\text{an}}/\text{Ker}(\delta_{0})$$
and
$$\text{Ker}(\delta_{2})=\text{Im}(\delta_{1}) \overset{\sim}{\longrightarrow} C_{b}^{1}(\mathfrak{g}_{r,s},U(\mathfrak{g})_{r,s}^{\text{an}})/\text{Ker}(\delta_{1})$$
are contracting. Then there is an isomorphism of algebras
$$U_{\frac{\hslash}{\varepsilon}}^{\dagger}(\mathfrak{g})_{r,s}^{\text{an}} \overset{\sim}{\longrightarrow} U(\mathfrak{g})_{r,s}^{\text{an}}\{\tfrac{\hslash}{\varepsilon}\}^{\dagger}$$
that induces the isomorphism $U_{\frac{\hslash}{\varepsilon}}^{\dagger}(\mathfrak{g})_{r,s}^{\text{an}} / \hslash U_{\frac{\hslash}{\varepsilon}}^{\dagger}(\mathfrak{g})_{r,s}^{\text{an}} \cong U(\mathfrak{g})_{r,s}^{\text{an}}$ of Corollary \ref{QuantumGroupModuloDaggerh}. Furthermore, this isomorphism is unique up to conjugation by a convolution invertible generalised element of $U(\mathfrak{g})_{r,s}^{\text{an}}\{\tfrac{\hslash}{\varepsilon}\}^{\dagger}$.
\end{corollary}

\begin{proof}
This follows from Corollary \ref{Rigidity1a}, Corollary \ref{Rigidity2a} and Corollary \ref{QuantumGroupModuloDaggerh}.
\end{proof}

If $\mathfrak{g}$ is finite dimensional then, as vector spaces, bounded Lie algebra cohomology $H_{b}^{n}(\mathfrak{g}_{r,s},M)$ agrees with the unbounded Lie algebra cohomology $H^{n}(\mathfrak{g},M)$ as defined in Chapter XVIII of \cite{QG}. Furthermore, if we also assume that $k$ is algebraically closed and $\mathfrak{g}$ is semisimple then $H^{n}(\mathfrak{g},U(\mathfrak{g}))=0$ by Corollary XVIII.3.3 of \cite{QG} for $n=1,2$. Unfortunately, it is unclear whether the same argument can be extended to show that $H_{b}^{n}(\mathfrak{g}_{r,s},U(\mathfrak{g})^{\text{an}}_{r,s})$ vanishes when $n=1,2$, or whether we have the strict exactness required for Corollary \ref{RigidityApplied} to hold. This is a goal of future work by the author. At the end of the following section we show that if we allow ourselves to work over formal powerseries in $\hslash$ then we may weaken these assumptions on bounded cohomology to remove the requirement of strict exactness.\\

In \cite{QG}, Kassel uses an algebraic analogue of the rigidity theorems of this Section, alongside another based on results of Drinfel'd from 1989 regarding quasi-Hopf algebra structures on universal enveloping algebra, to present a proof of the Drinfel'd-Khono theorem. This theorem states that the category of representations of the quantum enveloping algebra is equivalent, as a braided monoidal category, to the category of $U(\mathfrak{g})$-modules with associativity constraint given by the Drinfel'd associator and braiding given by the associated R-matrix. As a result of this, the associated braid group representations are equivalent. This can be interpreted as a statement about the monodromy of the Knizhnik-Zamolodchikov (KZ) equations that govern the Drinfel'd associator. In \cite{pMPatpKZE}, Furusho uses $p$-adic multiple polylogarithms to construct solutions to the $p$-adic KZ equations and a $p$-adic Drinfel'd associator. In the future the author hopes to prove a $p$-adic analogue of the Drinfel'd-Khono theorem and to investigate this link to Furusho's work.

\subsection{Analytic quantum groups over $k\llbracket \hslash \rrbracket$}
\label{WorkingOverk[[h]]}

\begin{defn}
Let $k \llbracket \hslash \rrbracket$ denote the IndBanach algebra of powerseries in $\hslash$,
$$k \llbracket \hslash \rrbracket := \lim_{n \geq 0} k[\hslash]/(\hslash^{n}) = \prod_{n \geq 0} k \cdot \hslash^{n}.$$
For an IndBanach space $V$ let $V\llbracket \hslash \rrbracket := \prod\nolimits _{n\geq 0} V\cdot \hslash^{n}$, which forms an IndBanach algebra if $V$ is an algebra.
\end{defn}

\begin{lem}
\label{CountableProductsCommuteWithTensor}
Let $(V(n))_{n \in \mathbb{N}}$ be a countable collection of Banach spaces, and $W$ be a Banach space. Then the natural map
$$\left(\prod_{n \geq 0}V(n)\right) \hat{\otimes} W \rightarrow \prod_{n \geq 0} \left(V(n) \hat{\otimes} W\right)$$
is an isomorphism.
\end{lem}

\begin{proof}
Fix a summable sequence of positive real numbers $a_{n} \in \mathbb{R}_{>0}$. By the explicit description of products in Section 1.4.1 in \cite{LaACH}, we have that
$$\prod\nolimits_{n \geq 0}V(n) \cong \text{colim}_{r_{n}>0} \prod\nolimits^{\leq 1}_{n \geq 0} V(n)_{r_{n}}.$$
The maps
$$\coprod\nolimits^{\leq 1}_{n \geq 0} V(n)_{r_{n}} \rightarrow \prod\nolimits^{\leq 1}_{n \geq 0} V(n)_{r_{n}}, \quad (v_{n})_{n \geq 0} \mapsto (v_{n})_{n \geq 0},$$
and
$$\prod\nolimits^{\leq 1}_{n \geq 0} V(n)_{r_{n}} \rightarrow \coprod\nolimits^{\leq 1}_{n \geq 0} V(n)_{a_{n}r_{n}}, \quad (v_{n})_{n \geq 0} \mapsto (v_{n})_{n \geq 0},$$
induce an isomorphism
$$\text{colim}_{r_{n}>0}\prod\nolimits^{\leq 1}_{n \geq 0} V(n)_{r_{n}} \overset{\sim}{\longrightarrow} \text{colim}_{r_{n}>0}\coprod\nolimits^{\leq 1}_{n \geq 0} V(n)_{r_{n}}.$$
Hence
$$\begin{array}{rcl}
\left(\prod_{n \geq 0}V(n)\right) \hat{\otimes} W &\cong& \left( \text{colim}_{r_{n}>0}\prod\nolimits^{\leq 1}_{n \geq 0} V(n)_{r_{n}} \right) \hat{\otimes} W\\
&\cong& \left( \text{colim}_{r_{n}>0}\coprod\nolimits^{\leq 1}_{n \geq 0} V(n)_{r_{n}} \right) \hat{\otimes} W\\
&\cong& \text{colim}_{r_{n}>0}\coprod\nolimits^{\leq 1}_{n \geq 0} (V(n) \hat{\otimes} W)_{r_{n}}\\
&\cong& \text{colim}_{r_{n}>0}\prod\nolimits^{\leq 1}_{n \geq 0} (V(n) \hat{\otimes} W)_{r_{n}}\\
&\cong& \prod_{n \geq 0} (V(n) \hat{\otimes} W).
\end{array}$$
\end{proof}

\begin{corollary}
\label{V[[h]]=Votimesk[[h]]}
For a Banach space $V$, $V\llbracket \hslash \rrbracket \cong V \hat{\otimes}_{k} k \llbracket \hslash \rrbracket$. Hence
$$V\llbracket \hslash \rrbracket \hat{\otimes}_{k\llbracket \hslash \rrbracket} W\llbracket \hslash \rrbracket \cong (V \hat{\otimes} W)\llbracket \hslash \rrbracket$$
for all Banach spaces $V$ and $W$.
\end{corollary}

\begin{proof}
This follows from Lemma \ref{CountableProductsCommuteWithTensor}.
\end{proof}

\begin{defn}
Let $r,s>0$ such that $1 \leq |q_{i}-q_{i}^{-1}|rs$. Then for any sufficiently small $\varepsilon>0$ we define
$$U_{\llbracket \hslash \rrbracket}(\mathfrak{g})_{r,s}^{\text{an}}:=U_{\frac{\hslash}{\varepsilon}}(\mathfrak{g})_{r,s}^{\text{an}} \hat{\otimes}_{k\{\frac{\hslash}{\varepsilon}\}} k\llbracket \hslash \rrbracket$$
as an IndBanach Hopf algebra over $k \llbracket \hslash \rrbracket$.
\end{defn}

\subsubsection{Quasi-triangularity}
\

The following is essentially a restatement of a well known result of Drinfel'd.

\begin{prop}
Suppose that $\mathfrak{g}$ is a simple Lie algebra. Then the IndBanach Hopf algebra $U_{\llbracket \hslash \rrbracket}(\mathfrak{g})_{r,s}^{\text{an}}$ is quasi-triangular.
\end{prop}

\begin{proof}
Fix $\varepsilon>0$. By the proof of Lemma \ref{CountableProductsCommuteWithTensor} we may write $U_{\llbracket \hslash \rrbracket}(\mathfrak{g})_{r,s}^{\text{an}}$ as the colimit
$$U_{\llbracket \hslash \rrbracket}(\mathfrak{g})_{r,s}^{\text{an}} = \text{"colim"}_{(\varepsilon_{n})} \ U_{\frac{\hslash}{\varepsilon}}(\mathfrak{g})_{r,s}^{\text{an}} \hat{\otimes}_{k\{\frac{\hslash}{\varepsilon}\}} \left( \coprod\nolimits_{n \geq 0}^{\leq 1} k_{\varepsilon_{n}} \cdot \hslash^{n} \right)$$
where the colimit is taken over all sequences of positive real numbers $(\varepsilon_{n})_{n \geq 0}$ such that $\varepsilon^{n} \varepsilon_{n'} \geq \varepsilon_{n+n'}$ for all $n,n' \geq 0$. Note that the requirement on $(\varepsilon_{n})_{n \geq 0}$ ensures that $\coprod\nolimits_{n \geq 0}^{\leq 1} k_{\varepsilon_{n}} \cdot \hslash^{n}$ is naturally a $k\{\frac{\hslash}{\varepsilon}\}$-module. Theorem 17 of Section 8.3.2 of \cite{QGaTR} exhibits a formula for an R-matrix in the $\hslash$-adic qunatum enveloping algebra over $\mathbb{C}$. As in \emph{loc. cit.} this formula naturally converges in $U_{\frac{\hslash}{\varepsilon}}(\mathfrak{g})_{r,s}^{\text{an}} \hat{\otimes}_{k\{\frac{\hslash}{\varepsilon}\}} \left( \coprod\nolimits_{n \geq 0}^{\leq 1} k_{\varepsilon_{n}} \cdot \hslash^{n} \right)$ for some sufficiently small sequence $(\varepsilon_{n})_{n \geq 0}$, which gives a generalised element of $U_{\llbracket \hslash \rrbracket}(\mathfrak{g})_{r,s}^{\text{an}}$ that makes $U_{\llbracket \hslash \rrbracket}(\mathfrak{g})_{r,s}^{\text{an}}$ quasi-triangular.
\end{proof}

\subsubsection{Rigidity}

\begin{theorem}
\label{Rigidity1b}
Let $\mathfrak{g}$ and $\mathfrak{g}'$ be Banach Lie algebras. Suppose we have two morphisms of $k\llbracket \hslash \rrbracket$ algebras $\alpha, \alpha':U(\mathfrak{g})_{r,s}^{\text{an}}\llbracket \hslash \rrbracket \rightarrow U(\mathfrak{g}')_{r',s'}^{\text{an}}\llbracket \hslash \rrbracket$ such that $\alpha \equiv \alpha'$ modulo $\hslash$. Suppose that $H_{b}^{1}(\mathfrak{g},U(\mathfrak{g}')_{r',s'}^{\text{an}})=0$. Then there exists a convolution invertible generalised element $F: k \rightarrow U(\mathfrak{g}')_{r',s'}^{\text{an}}\llbracket \hslash \rrbracket$ such that $\alpha'=F \ast \alpha \ast F^{-1}$.
\end{theorem}

\begin{proof}
This follows as in the proof of Theorem \ref{Rigidity1}. By Corollary \ref{V[[h]]=Votimesk[[h]]}, $\alpha$ is uniquely determined by its restriction to $U(\mathfrak{g})_{r,s}^{\text{an}}$, which may be written as a formal sum $\sum \hslash^{i}\alpha_{i}$ for $\alpha_{i}:U(\mathfrak{g})_{r,s}^{\text{an}} \rightarrow U(\mathfrak{g}')_{r',s'}^{\text{an}}$ bounded. We suppose we have $u_{0},u_{1},...,u_{n} \in U(\mathfrak{g}')_{r',s'}^{\text{an}}$ such that $U_{i}\alpha \equiv \alpha_{0} U_{i}$ modulo $\hslash^{i+1}$ as before, where $U_{i}=(1+u_{i}\hslash^{i})(1+u_{i-1}\hslash^{i-1})...(1+u_{0})$ is now a generalised element of $U(\mathfrak{g}')_{r',s'}^{\text{an}}\llbracket \hslash \rrbracket$. Since we are working with formal powerseries, both $(1+u_{i}\hslash^{i})$ and $U_{i}$ are automatically convolution invertible. Again, if $\alpha^{(n)}:=U_{n} \ast \alpha \ast U_{n}^{-1}$ whose restriction to $U(\mathfrak{g})_{r,s}^{\text{an}}$ is given by the formal sum $\sum \hslash^{i}\alpha^{(n)}_{i}$, then $\alpha_{n+1}^{(n)}$ restricts to a bounded 1-cocycle on $\mathfrak{g}$, hence is a 1-coboundary. We therefore obtain $u_{n+1} \in U(\mathfrak{g}')_{r',s'}^{\text{an}}$ such that $\alpha^{(n)}_{n}(x)=[\alpha_{0}(x),u_{n+1}]$ for all $x \in \mathfrak{g}_{r,s}$. Then $(1+u_{n+1}\hslash^{n+1})$ is an invertible generalised element of $U(\mathfrak{g}')_{r',s'}^{\text{an}}\llbracket \hslash \rrbracket$ and as in the proof of Theorem \ref{Rigidity1} we have
$$(1+u_{n+1}\hslash^{n+1})\alpha(x)(1+u_{n+1}\hslash^{n+1})^{-1}\equiv\alpha_{0}(x) \text{ mod } \hslash^{n+2}.$$
Taking $u_{0}:=0$ as our base case, as before, inductively gives a sequence of convolution invertible generalised elements $(U_{n})_{n \geq 0}$ that converge in each $U(\mathfrak{g})_{r,s}^{\text{an}} \cdot  \hslash^{n}$ to give $U:k \rightarrow U(\mathfrak{g}')_{r',s'}^{\text{an}}\llbracket \hslash \rrbracket$ such that $U \ast \alpha \ast U^{-1}=\alpha_{0}$. Similarly there is a convolution invertible generalised element $U'$ of $U(\mathfrak{g}')_{r',s'}^{\text{an}}\llbracket \hslash \rrbracket$ such that $U' \ast \alpha' \ast U'^{-1}=\alpha'_{0}=\alpha_{0}$. Taking $F=U'^{-1} \ast U$ gives our result.
\end{proof}

\begin{theorem}
\label{Rigidity2b}
Suppose $A$ is an IndBanach $k\llbracket \hslash \rrbracket$ algebra equipped with a $k\llbracket \hslash \rrbracket$-linear isomorphism $A \cong U(\mathfrak{g})_{r,s}^{\text{an}}\llbracket \hslash \rrbracket$ that preserves the unit, and that the induced isomorphism $A/\hslash A \cong U(\mathfrak{g})_{r,s}^{\text{an}}$ is an isomorphism of algebras. Suppose that $H_{b}^{2}(\mathfrak{g},U(\mathfrak{g})_{r,s}^{\text{an}})=0$. Then there is an isomorphism of $k\llbracket \hslash \rrbracket$-algebras $A \cong U(\mathfrak{g})_{r,s}^{\text{an}}\llbracket \hslash \rrbracket$ inducing the given algebra isomorphism modulo $\hslash$.
\end{theorem}

\begin{proof}
This proof follows as for Theorem \ref{Rigidity2}. The alternate multiplication $\mu$ on $U(\mathfrak{g})_{r,s}^{\text{an}}$ induced by $A$ is again determined by its restriction to $U(\mathfrak{g})_{r,s}^{\text{an}} \hat{\otimes} U(\mathfrak{g})_{r,s}^{\text{an}}$, which may written as a formal sum $\mu = \sum \hslash^{i}\mu_{i}$ for
$$\mu_{i}:U(\mathfrak{g})_{r,s}^{\text{an}} \hat{\otimes} U(\mathfrak{g})_{r,s}^{\text{an}} \rightarrow U(\mathfrak{g})_{r,s}^{\text{an}}.$$
We again assume we have maps $\alpha_{0},\alpha_{1},...,\alpha_{n}:U(\mathfrak{g})_{r,s}^{\text{an}} \rightarrow U(\mathfrak{g})_{r,s}^{\text{an}}$ such that $V_{i}(\mu(x,y)) \equiv \mu_{0}(V_{i}(x),V_{i}(y))$ modulo $\hslash^{i+1}$ for all $x,y \in U(\mathfrak{g})_{r,s}^{\text{an}}$, where
$$V_{i}=(\text{Id}+\hslash^{i}\alpha_{i})(\text{Id}+\hslash^{i-1}\alpha_{i-1})...(\text{Id}+\alpha_{0}):U(\mathfrak{g})_{r,s}^{\text{an}}\llbracket \hslash \rrbracket \rightarrow U(\mathfrak{g})_{r,s}^{\text{an}}\llbracket \hslash \rrbracket.$$
Again, since we are working with formal powerseries, each $V_{i}$ is automatically invertible and give new multiplication maps $\mu^{(i)}:=V_{i} \circ \mu \circ (V_{i}^{-1} \otimes V_{i}^{-1})$ on $U(\mathfrak{g})_{r,s}^{\text{an}}\llbracket \hslash \rrbracket$. We write the restriction of $\mu^{(i)}$ to $U(\mathfrak{g})_{r,s}^{\text{an}}$ as a formal sum $\sum \hslash^{j} \mu^{(i)}_{j}$. As in the proof of Theorem \ref{Rigidity2}, $\mu_{n+1}^{(n)}$ satisfies the conditions of Lemma \ref{VanishingofHb2Consequence}, so gives $\alpha_{n+1}:U(\mathfrak{g})_{r,s}^{\text{an}} \rightarrow U(\mathfrak{g})_{r,s}^{\text{an}}$ such that
$$(\text{Id}+\alpha_{n+1}\hslash^{n+1}) \circ \mu^{(n)} \circ \left( (\text{Id}+\alpha_{n+1}\hslash^{n+1})^{-1} \otimes (\text{Id}+\alpha_{n+1}\hslash^{n+1})^{-1} \right)$$
is equivalent to $\mu_{0}$ modulo $\hslash^{n+1}$. Taking $\alpha_{0}=0$ as a base case, we inductively obtain a sequence of morphisms $(V_{n})_{n \geq 0}$. Their restrictions to $U(\mathfrak{g})_{r,s}^{\text{an}}$ converges in each $U(\mathfrak{g})_{r,s}^{\text{an}} \cdot \hslash^{n}$ to give a morphism $V: U(\mathfrak{g})_{r,s}^{\text{an}} \llbracket \hslash \rrbracket \rightarrow U(\mathfrak{g})_{r,s}^{\text{an}} \llbracket \hslash \rrbracket$. It then follows that
$$A \cong U(\mathfrak{g})_{r,s}^{\text{an}} \llbracket \hslash \rrbracket \xrightarrow{V} U(\mathfrak{g})_{r,s}^{\text{an}}\llbracket \hslash \rrbracket$$
is our desired isomorphism of algebras.
\end{proof}

\begin{lem}
\label{QuantumGroupModuloPowerseriesh}
Let $r,s>0$ such that $1 \leq |q_{i}-q_{i}^{-1}|rs$. Then there is an isomorphism of $k\llbracket \hslash \rrbracket$-modules $U_{\llbracket \hslash \rrbracket}(\mathfrak{g})_{r,s}^{\text{an}} \cong U(\mathfrak{g})_{r,s}^{\text{an}}\llbracket \hslash \rrbracket$ that descends to an isomorphism of Banach Hopf algebras $U_{\llbracket \hslash \rrbracket}(\mathfrak{g})_{r,s}^{\text{an}} / \hslash U_{\llbracket \hslash \rrbracket}(\mathfrak{g})_{r,s}^{\text{an}} \cong U(\mathfrak{g})_{r,s}^{\text{an}}$.
\end{lem}

\begin{proof}
This follows from Theorem \ref{QuantumGroupModuloh}.
\end{proof}

\begin{corollary}
Suppose that
$$H_{b}^{1}(\mathfrak{g},U(\mathfrak{g})_{r,s}^{\text{an}})=0=H_{b}^{2}(\mathfrak{g},U(\mathfrak{g})_{r,s}^{\text{an}}).$$
Then there is an isomorphism of algebras
$$U_{\llbracket \hslash \rrbracket}(\mathfrak{g})_{r,s}^{\text{an}} \overset{\sim}{\longrightarrow} U(\mathfrak{g})_{r,s}^{\text{an}} \llbracket \hslash \rrbracket$$
that induces the isomorphism $U_{\llbracket \hslash \rrbracket}(\mathfrak{g})_{r,s}^{\text{an}} / \hslash U_{\llbracket \hslash \rrbracket}(\mathfrak{g})_{r,s}^{\text{an}} \cong U(\mathfrak{g})_{r,s}^{\text{an}}$. Furthermore, this isomorphism is unique up to conjugation by a convolution invertible generalised element of $U(\mathfrak{g})_{r,s}^{\text{an}}\llbracket \hslash \rrbracket$.
\end{corollary}

\begin{proof}
This follows from Theorem \ref{Rigidity1b}, Theorem \ref{Rigidity2b} and Corollary \ref{QuantumGroupModuloPowerseriesh}.
\end{proof}

\section{Archimedean analytic quantum groups}
\label{AQGSection}

Throughout this section we shall assume {\bf (A)}.

\subsection{Constructing Archimedean Analytic Nichols algebras}
\label{ANicholsAlegbraSection}

\begin{lem}
\label{ACoidealsOfTensorAlgebras}
Let $V$ be a Banach space with pre-braiding $c$ of norm at most 1. Suppose we have closed homogeneous subspaces $I_{r} \subset J_{r} \subset T_{r}(V)$ for each $r>0$ such that, whenever $r \geq r'$, $T_{r}(V) \rightarrow T_{r'}(V)$ maps $I_{r}$ and $J_{r}$ to $I_{r'}$ and $J_{r'}$ respectively. Suppose further that $\Delta(I_{r}) \subset \overline{I_{\frac{r}{2}} \hat{\otimes} T_{\frac{r}{2}}(V) + T_{\frac{r}{2}}(V) \hat{\otimes} I_{\frac{r}{2}}}$ and $\Delta(J_{r}) \subset \overline{J_{\frac{r}{2}} \hat{\otimes} T_{\frac{r}{2}}(V) + T_{\frac{r}{2}}(V) \hat{\otimes} J_{\frac{r}{2}}}$. If the induced map
$$P(\text{colim}_{r>0} T_{r}(V)/I_{r}) \rightarrow \text{colim}_{r>0} T_{r}(V)/I_{r} \rightarrow \text{colim}_{r>0} T_{r}(V)/J_{r}$$
is injective then $I_{r}=J_{r}$ for all $r>0$.
\end{lem}

\begin{proof}
This is similar to Lemma \ref{CoidealsOfTensorAlgebras}. We will denote by $R_{r}=T_{r}(V)/I_{r}$, $R_{r}'=T_{r}(V)/J_{r}$, by $R_{r}(n)$, $R'_{r}(n)$ their respective $n$th graded pieces, and by $R_{r}(\leq n)=\bigoplus_{i \leq n}R_{r}(i)$ and $R'_{r}(\leq n)=\bigoplus_{i \leq n}R'_{r}(i)$. Since $I_{r} \subset J_{r}$ we have strict analytically graded epimorphisms $f_{r}:R_{r} \rightarrow R'_{r}$ which restricts to isometries $R_{r}(1) \rightarrow R'_{r}(1)$. Suppose that $f_{r}$ restricts also to an isometry $R_{r}(\leq n) \rightarrow R'_{r}(\leq n)$ for all $r>0$. Let $r>0$ and $x \in {R_{r}(\leq n+1)}$. By the proof of Lemma \ref{GradedPiecesOfTensorAlgebraCoideals}, we see that
$$
\begin{array}{rcl}
\Delta(R_{r}(n+1)) &\subset& \sum_{i=0}^{n+1} R_{\frac{r}{2}}(i) \hat{\otimes} R_{\frac{r}{2}}(n-i)\\
&\subset& R_{\frac{r}{2}}(n+1) \hat{\otimes} k + k \hat{\otimes} R_{\frac{r}{2}}(n+1) + R_{\frac{r}{2}}(\leq n) \hat{\otimes} R_{\frac{r}{2}}(\leq n),
\end{array}
$$
so $\Delta(x)=y \otimes 1 + 1 \otimes y' + z$ for some $y,y' \in R_{\frac{r}{2}}(n+1)$, $z \in R_{\frac{r}{2}}(\leq n) \hat{\otimes} R(\leq n)$. But then $x-y = (\text{Id} \otimes \varepsilon)\Delta(x)-y= \varepsilon(y)\cdot 1 + (\text{Id} \otimes \varepsilon)(z) \in R_{\frac{r}{2}}(\leq n)$ and likewise $x-y' \in R_{\frac{r}{2}}(\leq n)$. So $\Delta(x)=x \otimes 1 + 1 \otimes x + z'$ for some $z' \in {R_{\frac{r}{2}}(\leq n)} \hat{\otimes} {R_{\frac{r}{2}}(\leq n)}$. If $f_{r}(x)=0$ then $(f_{\frac{r}{2}} \otimes f_{\frac{r}{2}})(z')=0$, but, since each $R_{\frac{r}{2}}(\leq n)$ is finite dimensional and hence flat, $f_{\frac{r}{2}} \otimes f_{\frac{r}{2}}$ is injective on $R_{\frac{r}{2}}(\leq n) \hat{\otimes} R_{\frac{r}{2}}(\leq n)$. So $z'=0$ and $x$ is primitive, hence $x=0$. Thus $f_{r}$ is an isometry ${R_{r}(\leq n)} \rightarrow {R'_{r}(\leq n)}$ since the norms on $R_{r}(\leq n)$ and $R'_{r}(\leq n)$ are the quotient norms from $\coprod_{i \leq n}^{\leq 1} V_{r}^{\hat{\otimes} i}$. Taking contracting colimits over $n$ we see that $f_{r}$ is isometric. Hence $I_{r}=J_{r}$.
\end{proof}

\begin{lem}
\label{ANicholsAlgebrasAreFlat}
Let $W=\coprod_{n \geq 0}^{\leq 1} W(n)$ be a Banach space where each $W(n)$ is finite dimensional. Then $W$ is flat.
\end{lem}

\begin{proof}
Let $f:A \rightarrow B$ be a morphism of Banach spaces. Then
$$\begin{array}{rcl}
\text{Ker}\left(W \hat{\otimes} A \xrightarrow{\text{Id} \otimes f} W \hat{\otimes} B \right) &\cong& \text{Ker}\left(\coprod^{\leq 1} (W(n) \hat{\otimes} A) \rightarrow \coprod^{\leq 1} (W(n) \hat{\otimes} B)\right)\\
&\cong& \coprod^{\leq 1}\text{Ker}\left( W(n) \hat{\otimes} A \xrightarrow{\text{Id} \otimes f} W(n) \hat{\otimes} B\right)\\
&\cong& \coprod^{\leq 1} \left( W(n) \hat{\otimes} \text{Ker}\left( A \xrightarrow{f} B\right) \right)\\
&\cong& W \hat{\otimes} \text{Ker}\left( A \xrightarrow{f} B\right),
\end{array}$$
where the second isomorphism is an easy computation and the third is because each $W(n)$ is finite dimensional. Hence $W$ is flat.
\end{proof}

\begin{prop}
\label{Radius1NicholsAlgebrasExistsA}
Let $V$ be a finite dimensional Banach space with pre-braiding $c$ of norm at most $1$. Then a dagger-0 Nichols algebra of $V$ exists.
\end{prop}

\begin{proof}
As a variation on the proof of Proposition \ref{Radius1NicholsAlgebrasExists}, let $\mathcal{I}(V)$ be the set consisting of collections $(I_{r})_{r \in \mathbb{R}_{>0}}$ where each $I_{r}$ is a homogeneous ideals of $T_{r}(V)$ contained in $\coprod_{n \geq 2}^{\leq 1} V_{r}^{\hat{\otimes}n}$ such that the maps $T_{r}(V) \rightarrow T_{r'}(V)$ for $r>r'$ send $I_{r}$ to a subspace of $I_{r'}$ and the maps $\Delta:T_{r}(V) \rightarrow T_{\frac{r}{2}}(V) \hat{\otimes} T_{\frac{r}{2}}(V)$ giving the comultiplication send $I_{r}$ to a subspace of $I_{\frac{r}{2}} \hat{\otimes} T_{\frac{r}{2}}(V) + T_{\frac{r}{2}}(V) \hat{\otimes} I_{\frac{r}{2}}$. Let $\mathcal{I}'(V)$ be the subset of $\mathcal{I}(V)$ consisting of $(I_{r})_{r \in \mathbb{R}_{>0}}$ for which $\overline{I_{r}} \hat{\otimes} T(V) + T(V) \hat{\otimes} \overline{I_{r}}$ is preserved by $\tilde{c}$. Let $I_{r}(V)$ be the sum of all ideals $I_{r}$ for $(I_{r})_{r \in \mathbb{R}_{>0}} \in \mathcal{I}(V)$, which is closed by a similar argument to the proof of Proposition \ref{Radius1NicholsAlgebrasExists}. Likewise let $I'_{r}(V)$ be the sum of all ideals $I'_{r}$ for $(I'_{r})_{r \in \mathbb{R}_{>0}} \in \mathcal{I}'(V)$, which is again closed. We will denote by $I(V)$ and $I'(V)$ the collection $(I_{r}(V))_{r \in \mathbb{R}_{>0}}$ and $(I'_{r}(V))_{r \in \mathbb{R}_{>0}}$, and let $T_{0}(V)^{\dagger}/I(V):=\text{"colim"}_{r>0}T_{r}(V)/I_{r}(V)$ and $T_{0}(V)^{\dagger}/I'(V):=\text{"colim"}_{r>0}T_{r}(V)/I'_{r}(V)$. We must check that
$$P(T_{0}(V)^{\dagger}/I(V))=(T_{0}(V)^{\dagger}/I(V))(1).$$
As before, the closed ideals generated by $I_{r}(V)$ and
$$\left\lbrace x \in \coprod\nolimits_{n \geq 2}^{\leq 1} \middle| \substack{\Delta(x) \in x \otimes 1 + 1 \otimes x + I_{r'}(V) \otimes T_{r'}(V) + T_{r'}(V) \otimes I_{r'}(V) \\ \text{for some sufficiently small } \tfrac{r}{2} \geq r'>0} \right\rbrace$$
form an element of $\mathcal{I}(V)$. So $P(T_{0}(V)^{\dagger}/I(V))=(T_{0}(V)^{\dagger}/I(V))(1)$, and likewise $P(T_{0}(V)^{\dagger}/I'(V))=(T_{0}(V)^{\dagger}/I'(V))(1)$. So, by Lemma \ref{ACoidealsOfTensorAlgebras}, we have $I_{r}(V)=I'_{r}(V)$ for all $r>0$, and so the braiding $\tilde{c}$ descends to braidings of $T_{r}(V)/I_{r}(V)$. Hence $T_{0}(V)^{\dagger}/I(V):=\text{"colim"}_{r>0}T_{r}(V)/I_{r}(V)$ forms a $k\{t\}^{\dagger}$-graded braided IndBanach Hopf algebra with $(T_{0}(V)^{\dagger}/I(V))(0) \cong k$ and generated by $(T_{0}(V)^{\dagger}/I(V))(1) \cong V$.  Finally, by Lemma \ref{ANicholsAlgebrasAreFlat}, $T_{r}(V)/I_{r}(V)$ is flat for each $r>0$. Hence the filtered colimit $\text{"colim"}_{r>0}T_{r}(V)/I_{r}(V)$ is flat.
\end{proof}

\begin{defn}
Let $\mathfrak{B}_{r}(V):=T_{r}(V)/I_{r}(V)$. We shall use the notation $\mathfrak{B}_{0}(V)^{\dagger}$ for the dagger Nichols algebra $\text{colim}_{r>0} \mathfrak{B}_{r}(V)$ defined in the proof of Proposition \ref{Radius1NicholsAlgebrasExistsA}.
\end{defn}

\begin{remark}
Note that, in the proof of Proposition \ref{Radius1NicholsAlgebrasExistsA}, each $I_{r}(V)$ is a closed ideal of $T_{r}(V)$, and hence each $\mathfrak{B}_{r}(V)$ is a Banach algebra. So $\mathfrak{B}_{0}(V)^{\dagger}$ is locally Banach as an algebra.
\end{remark}

\begin{lem}
\label{ExtendBilinearFormA}
Let $V$ be a Banach space with pre-braiding $c$ of norm at most 1, and suppose we have a symmetric non-degenerate bilinear form $\langle - , - \rangle: V \hat{\otimes} V \rightarrow k$. Then this extends to a dual pairing $(\coprod_{n \geq 0} V^{\hat{\otimes} n}) \hat{\otimes} T_{0}^{c}(V)^{\dagger} \rightarrow k$ such that the composition
$$V^{\hat{\otimes} n} \hat{\otimes} V^{\hat{\otimes} m} \longrightarrow \left( \coprod_{n \geq 0} V^{\hat{\otimes} n} \right) \hat{\otimes} T_{0}^{c}(V)^{\dagger} \longrightarrow k$$
is symmetric when $n=m$ and is $0$ if $n \neq m$.
\end{lem}

\begin{proof}
Again, we proceed as in the proof of Proposition 1.2.3 of \cite{ItQG}, as we did for Lemma \ref{ExtendBilinearForm}. The given bilinear form induces a morphism $V \rightarrow V^{\ast}$ whilst the projections $T_{r}(V) \twoheadrightarrow V$ induce a morphism $V^{\ast} \rightarrow (T_{0}(V)^{\dagger})^{\ast}$. The coalgebra structure on $T_{0}(V)^{\dagger}$ gives $(T_{0}(V)^{\dagger})^{\ast}$ an algebra structure, and so the composition $V \rightarrow V^{\ast} \rightarrow (T_{0}(V)^{\dagger})^{\ast}$ induces a unique algebra homomorphism $\coprod_{n \geq 0} V^{\hat{\otimes} n} \rightarrow (T_{0}(V)^{\dagger})^{\ast}$ that gives our bilinear form. The fact that this is is a dual pairing that is symmetric on $V^{\hat{\otimes} n} \hat{\otimes} V^{\hat{\otimes} n}$ with $V^{\hat{\otimes} n}$ perpendicular to $V^{\hat{\otimes} m}$ for $m \neq n$ follows as in the proof of Lemma \ref{ExtendBilinearForm}.
\end{proof}

\begin{remark}
Suppose we have a Banach space $V$ with pre-braiding $c$ of norm at most 1 and a non-degenerate symmetric bilinear form $\langle - , - \rangle: V \hat{\otimes} V \rightarrow k$. Let $r,s>0$ with $\|\langle - , - \rangle\| \leq rs$. The $n$-fold comultiplication $T_{r}(V) \rightarrow T_{r/2^{n}}(V)^{\hat{\otimes} n}$ given by Proposition \ref{ArchimedeanTensorBialgebra} induces an $n$-fold multiplication $(T_{r/2^{n}}(V)^{\ast})^{\hat{\otimes} n} \rightarrow T_{r}(V)^{\ast}$. Also, for each $n >0$ the bilinear form $\langle - , - \rangle$ and the projection $T_{r/2^{n}}(V) \twoheadrightarrow V_{r/2^{n}}$ induce contracting homomorphisms $V_{2^{n}s} \rightarrow V_{r/2^{n}}^{\ast}$ and $V_{r/2^{n}}^{\ast} \twoheadrightarrow T_{r/2^{n}}(V)^{\ast}$. Then the compositions
$$V_{2^{n}s}^{\hat{\otimes} n} \rightarrow (V_{r/2^{n}}^{\ast})^{\hat{\otimes} n} \rightarrow (T_{r/2^{n}}(V)^{\ast})^{\hat{\otimes} n} \rightarrow T_{r}(V)^{\ast}$$
induce a bilinear pairing $\mathscr{T}_{s}(V) \hat{\otimes} T_{r}(V) \rightarrow k$, where
$$\mathscr{T}_{s}(V):=\coprod\nolimits^{\leq 1}_{n \geq 0} (V_{2^{n}s})^{\hat{\otimes} n} = \left\lbrace \sum x_{n} \middle| x_{n} \in V^{\hat{\otimes} n} \text{ and } \sum_{n}\|x\|2^{n^{2}}s^{n} < \infty \right\rbrace,$$
and hence a bilinear pairing $\mathscr{T}_{\infty}(V) \hat{\otimes} T_{0}(V)^{\dagger} \rightarrow k$, where $\mathscr{T}_{\infty}(V):= \text{lim}_{s>0} \mathscr{T}_{s}(V)$. Furthermore, the restriction of these bilinear forms to $\bigoplus_{n \geq 0} V^{\otimes n}$ is a dual pairing by Proposition 1.2.3 of \cite{ItQG}. Unfortunately, the algebra structure of $\bigoplus_{n \geq 0} V^{\otimes n}$ neither extends to $\mathscr{T}_{s}(V)$ nor $\mathscr{T}_{\infty}(V)$. It does, however, extend to
$$\text{lim}_{s_{n}>0} \coprod\nolimits^{\leq 1}_{n \geq 0} (V_{s_{n}})^{\hat{\otimes} n} = \text{lim}_{s_{n}>0} \left\lbrace \sum x_{n} \middle| x_{n} \in V^{\hat{\otimes} n} \text{ and } \sum_{n}\|x\|s_{n}^{n} < \infty \right\rbrace$$
which we may pair with $T_{0}(V)^{\dagger}$ via the composition
$$\left( \text{lim}_{s_{n}>0} \coprod\nolimits^{\leq 1}_{n \geq 0} (V_{s_{n}})^{\hat{\otimes} n}\right) \hat{\otimes} T_{0}(V)^{\dagger} \rightarrow \mathscr{T}_{\infty}(V) \hat{\otimes} T_{0}(V)^{\dagger} \rightarrow k.$$
The following lemma shows that $\text{lim}_{s_{n}>0} \coprod\nolimits^{\leq 1}_{n \geq 0} (V_{s_{n}})^{\hat{\otimes} n}$ is just $\coprod_{n \geq 0} V^{\hat{\otimes} n}$ as in the previous lemma.
\end{remark}

\begin{lem}
The natural morphism
$$\coprod_{n \geq 0} V^{\hat{\otimes} n} \rightarrow \text{lim}_{s_{n}>0} \coprod\nolimits^{\leq 1}_{n \geq 0} (V_{s_{n}})^{\hat{\otimes} n}$$
is an isomorphism.
\end{lem}

\begin{proof}
As in the proof of Lemma \ref{CountableProductsCommuteWithTensor}, for a sequence of positive real numbers $a_{n}>0$ such that $(a_{n}^{n})_{n \geq 0}$ is summable, the maps
$$\coprod\nolimits^{\leq 1}_{n \geq 0} (V_{s_{n}})^{\hat{\otimes} n} \rightarrow \prod\nolimits^{\leq 1}_{n \geq 0} (V_{s_{n}})^{\hat{\otimes} n}, \quad (x_{n})_{n \geq 0} \mapsto (x_{n})_{n \geq 0},$$
and
$$\prod\nolimits^{\leq 1}_{n \geq 0} (V_{s_{n}})^{\hat{\otimes} n} \rightarrow \coprod\nolimits^{\leq 1}_{n \geq 0} (V_{a_{n}s_{n}})^{\hat{\otimes} n}, \quad (x_{n})_{n \geq 0} \mapsto (x_{n})_{n \geq 0},$$
induce an isomorphism
$$\text{lim}_{s_{n}>0}\coprod\nolimits^{\leq 1}_{n \geq 0} (V_{s_{n}})^{\hat{\otimes} n} \cong \text{lim}_{s_{n}>0}\prod\nolimits^{\leq 1}_{n \geq 0} (V_{s_{n}})^{\hat{\otimes} n}.$$
For a Banach space $W$,
$$\begin{array}{rcl}
\text{Hom}(W,\text{lim}_{s_{n}>0}\coprod\nolimits^{\leq 1}_{n \geq 0} (V_{s_{n}})^{\hat{\otimes} n}) &\cong& \text{Hom}(W,\text{lim}_{s_{n}>0}\prod\nolimits^{\leq 1}_{n \geq 0} (V_{s_{n}})^{\hat{\otimes} n})\\
&\cong& \text{lim}_{s_{n}>0}\prod\nolimits^{\leq 1}_{n \geq 0} \text{Hom}(W,(V_{s_{n}})^{\hat{\otimes} n})\\
&=& \left\lbrace (f_{n})_{n \geq 0} \middle| \substack{f_{n}:W \rightarrow V^{\hat{\otimes} n} \text{ and } (\|f_{n}\|s_{n})_{n \geq 0}\\ \text{ is bounded for all } s_{n}>0} \right\rbrace\\
&=& \left\lbrace (f_{n})_{n \geq 0} \middle| \substack{f_{n}:W \rightarrow V^{\hat{\otimes} n} \text{ and}\\ f_{n} = 0 \text{ for } n\gg 0} \right\rbrace
\end{array}$$
where the last equality is because if $s_{n}:=\frac{n}{\|f_{n}\|}$ when $f_{n} \neq 0$ and $s_{n}:=1$ if $f_{n}=0$ then $(\|f_{n}\|s_{n})_{n \geq 0}$ is only bounded if $f_{n}=0$ for $n \gg 0$. Hence the map of sets
$$\text{Hom}(W,\coprod_{n \geq 0}V^{\hat{\otimes} n}) \cong \bigoplus_{n \geq 0}\text{Hom}(W,V^{\hat{\otimes} n}) \rightarrow \text{Hom}(W,\text{lim}_{s_{n}>0}\coprod\nolimits^{\leq 1}_{n \geq 0} (V_{s_{n}})^{\hat{\otimes} n})$$
is bijective. Since this holds for all Banach spaces $W$ we must have
$$\coprod_{n \geq 0} V^{\hat{\otimes} n} \cong \text{lim}_{s_{n}>0} \coprod\nolimits^{\leq 1}_{n \geq 0} (V_{s_{n}})^{\hat{\otimes} n}.$$
\end{proof}

\begin{prop}
\label{BilinearFormGivesNicholsAlgebraA}
Let $V$ be a Banach space with pre-braiding $c$ of norm at most 1, and suppose we have a non-degenerate bilinear form $\langle - , - \rangle: V \hat{\otimes} V \rightarrow k$. Then for each $0< r$, let $I_{r}$ be the radical in $T_{r}^{c}(V)$ of
$$\left( \coprod_{n \geq 0} V^{\hat{\otimes} n} \right) \hat{\otimes} T_{r}^{c}(V) \longrightarrow \left( \coprod_{n \geq 0} V^{\hat{\otimes} n} \right) \hat{\otimes} T_{0}^{c}(V)^{\dagger} \longrightarrow k.$$
Then the $I_{r}$ are closed homogeneous ideals in $T_{r}^{c}(V)$ such that $T_{r}(V) \rightarrow T_{r'}(V)$ maps $I_{r}$ to $I_{r'}$ and $\Delta:T_{r}(V) \rightarrow T_{\frac{r}{2}}(V) \hat{\otimes} T_{\frac{r}{2}}(V)$ maps $I_{r}$ to $I_{\frac{r}{2}} \hat{\otimes} T_{\frac{r}{2}}(V) + T_{\frac{r}{2}}(V) \hat{\otimes} I_{\frac{r}{2}}$. Furthermore, $P(\text{colim}_{r>0} T_{r}^{c}(V)/I_{r}) = V$, hence $\text{colim}_{r>0} T_{r}^{c}(V)/I_{r}$ is a dagger Nichols algebra of $V$.
\end{prop}

\begin{proof}
This is analogous to the proof of Proposition \ref{BilinearFormGivesNicholsAlgebra}. The fact that $I_{r}$ are closed homogeneous ideals compatible under restrictions to smaller radii with $\Delta(I_{r}) \subset I_{\frac{r}{2}} \hat{\otimes} T_{\frac{r}{2}}(V) + T_{\frac{r}{2}}(V) \hat{\otimes} I_{\frac{r}{2}}$ follows from Lemma \ref{ExtendBilinearFormA}. As the bilinear form  on $V$ is non-degenerate, $I_{r} \subset \coprod_{n \geq 2}^{\leq 1} V_{r}^{\hat{\otimes}n}$, and the quotient $T_{r}(V)/I_{r}$ is generated by $V$. It remains to check that the subspace of primitive elements is just $V$, which follows from the same argument as for Proposition \ref{BilinearFormGivesNicholsAlgebra}.  Again, $\tilde{c}$ preserves $I_{r} \hat{\otimes} T_{r}(V) + T_{r}(V) \hat{\otimes} I_{r}$ for all $r>0$, so the braidings on $T_{r}^{c}(V)$ descend to braidings of $T_{r}^{c}(V)/I_{r}$ and $\text{colim}_{r>0} T_{r}^{c}(V)/I_{r}$ is a dagger Nichols algebra of $V$.
\end{proof}

\begin{prop}
\label{NicholsAlgebraUniversalPropertyA}
Let $V$ be a finite dimensional Banach space with pre-braiding $c$ of norm at most $1$. Let $R=\text{colim}_{r>0} \coprod^{\leq 1}R(n)_{r^{n}}$ be a dagger graded pre-braided IndBanach Hopf algebra. Suppose further that the algebra structure on $R$ is determined by algebra structures on $\coprod^{\leq 1}R(n)_{r^{n}}$ for each $r>0$ which are generated by $R(1)_{r}$. Then, if $R(0)\cong k$, $P(R)=R(1) \cong V$ as pre-braided Banach spaces, there is an epimorphism of dagger graded braided Hopf algebras $\mathfrak{B}_{0}^{c}(V) \rightarrow R$ extending $V \overset{\sim}{\longrightarrow} R(1)$.
\end{prop}

\begin{proof}
Without loss of generality, we may assume that the isomorphism $V \xrightarrow{\sim} R(1)$ is of norm at most 1. Hence we obtain morphisms of Banach algebras $T_{r}(V) \rightarrow \coprod^{\leq 1}_{n \geq 0} R(n)_{r^{n}}$ that are epic since $R(1)_{r}$ generate $\coprod^{\leq 1}_{n \geq 0} R(n)_{r^{n}}$. Since $P(R)=R(1)$, this induces a morphism of Hopf algebras $T_{0}(V)^{\dagger} \rightarrow R$. Hence the kernels of the maps $T_{r}(V) \rightarrow \coprod^{\leq 1}_{n \geq 0} R(n)_{r^{n}}$ form an element of $\mathcal{I}(V)$, in the notation of the proof of Proposition \ref{Radius1NicholsAlgebrasExistsA}. Thus we obtain a well-defined epimorphism $\mathfrak{B}_{0}^{c}(V) \rightarrow R$.
\end{proof}

\begin{prop}
\label{NicholsAlgebrasEquivalentDefinitionA}
Let $V$ be a finite dimensional Banach space with pre-braiding $c$ of norm at most $1$, and suppose we have a non-degenerate symmetric bilinear form $\langle - , - \rangle: V \hat{\otimes} V \rightarrow k$. Retaining the notation of Proposition \ref{BilinearFormGivesNicholsAlgebraA}, there is an isomorphism $\mathfrak{B}_{0}^{c}(V)^{\dagger} \rightarrow \text{colim}_{r>0}T_{r}^{c}(V)/I_{r}$.
\end{prop}

\begin{proof}
By Proposition \ref{NicholsAlgebraUniversalPropertyA} there is an epimorphism $\mathfrak{B}_{0}^{c}(V) \rightarrow \text{colim}_{r>0}T_{r}^{c}(V)/I_{r}$. By Lemma \ref{ACoidealsOfTensorAlgebras} the maps $\mathfrak{B}_{r}^{c}(V) \rightarrow T_{r}^{c}(V)/I_{r}$ are all isomorphisms. Hence $\mathfrak{B}_{0}^{c}(V) \rightarrow \text{colim}_{r>0}T_{r}^{c}(V)/I_{r}$ is an isomorphism.
\end{proof}

\begin{defn}
Let $V$ be a finite dimensional Banach space with pre-braiding $c$ of norm at most $1$, and suppose we have a non-degenerate symmetric bilinear form $\langle - , - \rangle: V \hat{\otimes} V \rightarrow k$. For each $n \geq 0$ let $I(n)$ be the radical in $V^{\hat{\otimes} n}$ of the composition
$$V^{\hat{\otimes} n} \hat{\otimes} T_{0}^{c}(V)^{\dagger} \longrightarrow \left( \coprod_{n \geq 0} V^{\hat{\otimes} n} \right) \hat{\otimes} T_{0}^{c}(V)^{\dagger} \longrightarrow k$$
and let $\mathfrak{B}_{\text{alg}}^{c}(V)$ be the braided graded Hopf algebra $\coprod_{n \geq 0} V^{\hat{\otimes} n}/I(n)$. This is the algebraic Nichols algebra of $V$, as defined in \cite{PHA}, by Proposition 2.10 of \emph{loc. cit.} By Lemma \ref{ExtendBilinearFormA} there is a dual pairing
$$\mathfrak{B}_{\text{alg}}^{c}(V) \hat{\otimes} \mathfrak{B}_{0}^{c}(V)^{\dagger} \rightarrow k,$$
which we also denote by $\langle - , - \rangle$.
\end{defn}

\subsection{Constructing Archimedean analytic quantum groups}
\label{ConstructingArchimedeanAnalyticQuantumGroups}
\

Again, we will use $q$ to denote an element of $k \setminus \{0\}$ of norm 1 and we fix root datum as in Definition \ref{KacMoodyRootDatum} for a Lie algebra $\mathfrak{g}$.

\begin{defn}
\label{CartanPartA}
Let $H=\coprod_{\lambda \in \Phi^{\ast}}^{\leq 1} k \cdot K_{\lambda}$ denote the Banach group Hopf algebra of $\Phi^{\ast}$ with
$$K_{\lambda} \cdot K_{\lambda'}=K_{\lambda + \lambda'}, \quad \Delta_{H}(K_{\lambda})=K_{\lambda} \otimes K_{\lambda}\quad\text{and}\quad S(K_{\lambda})=K_{-\lambda}.$$
We continue to use the notation
$$t_{i}:=K_{\frac{(\alpha_{i},\alpha_{i})}{2}\lambda_{i}}$$
and let $H'$ be the closed sub-Hopf algebra generated by $\{t_{i} \mid i \in I\}$. As before, there is a duality pairing $H \hat{\otimes} H' \rightarrow k$, which we continue to denote by $\langle-,- \rangle$, defined by
$$ K_{\lambda} \otimes \underline{t}^{\underline{n}} \mapsto q^{\lambda(\sum n_{i} \alpha_{i})}$$
and simultaneous algebra homomorphisms and coalgebra anti-homomorphisms
$$\mathscr{R}:H' \rightarrow H, \ t_{i} \mapsto t_{i}, \quad \overline{\mathscr{R}}:H' \rightarrow H, \ t_{i} \mapsto t_{i}^{-1},$$
for $i \in I$, making $H$ and $H'$ a weakly quasi-triangular dual pair.
\end{defn}

\begin{remark}
\label{HAPairingWeakQuasiTriangularA}
Proposition 4.4 of \cite{TRTfIBS} says that the category of IndBanach $H$-modules that decompose locally into Banach weight spaces with weights in the root lattice $\Psi \subset \Phi$, which we will continue to denote by $H\text{-Mod}_{\Psi}$ as before, is braided.
\end{remark}

\begin{defn}
\label{QuantumBraidingA}
\label{AnalyticQuantumGroupPositivePart}
Let $V=\coprod_{i \in I}^{\leq 1} k \cdot v_{i}$ with basis $\{v_{i} \mid i \in I \}$ have the $H'$-coaction $v_{i} \mapsto t_{i} \otimes v_{i}$. Then $V$ is a $H$-module with braiding $c(x_{i} \otimes x_{j})=q^{(\alpha_{i},\alpha_{j})}x_{j} \otimes x_{i}$. Let $\langle -,- \rangle$ be the non-degenerate bilinear form on $V$ where
$$\langle v_{i},v_{j} \rangle = \delta_{i,j}\frac{1}{(q_{i}-q_{i}^{-1})} \quad \text{for} \quad q_{i}=q^{\frac{(\alpha_{i},\alpha_{i})}{2}}.$$
Given $0 < r$ we denote by $\mathbf{f}_{r}^{\text{an}}$ the algebras $\mathfrak{B}_{r}^{c}(V)$. We then use the notation $\mathbf{f}_{0}^{\dagger}$ for the dagger Nichols algebra $\mathfrak{B}_{0}^{c}(V)^{\dagger}$ and $\mathbf{f}$ for the algebraic Nichols algebra $\mathfrak{B}_{\text{alg}}^{c}(V)$.
\end{defn}

\begin{lem}
\label{PositivePartDenseA}
For each $0<r$, the positive part of the quantum group is dense in the Banach space $\mathbf{f}_{r}^{\text{an}}$.
\end{lem}

\begin{proof}
The proof of this is the same as for Lemma \ref{PositivePartDense}.
\end{proof}

\begin{lem}
Suppose $0< r$. Then there is a $H'$-equivariant dual pairing $\langle-,-\rangle:\mathbf{f}_{0}^{\dagger} \hat{\otimes} \mathbf{f} \rightarrow k$ extending $\langle-,-\rangle$ in Definition \ref{AnalyticQuantumGroupPositivePart} such that $\langle \mathbf{f}_{0}^{\dagger}(n) , \mathbf{f}(m) \rangle = \{0\}$ for $n \neq m$.
\end{lem}

\begin{proof}
This follows from Lemma \ref{ExtendBilinearFormA}.
\end{proof}

\begin{defn}
\label{AnalyticQuantumGroupsA}
We denote by $U_{q}(\mathfrak{g})_{0}^{\dagger}$ the double bosonisation $U(\mathbf{f}_{0}^{\dagger},H,\mathbf{f})$. Let us denote by $F_{i}$ the generalised element in $\mathbf{f}_{0}^{\dagger} \rtimes H$ represented by ${v_{i} \otimes 1} \in \mathbf{f}_{r}^{\text{an}} \hat{\otimes} H$, and by $E_{i}$ the generalised element in $\overline{H} \ltimes \overline{\mathbf{f}}$ represented by $1 \otimes v_{i}\in H \hat{\otimes} V$ for $i \in I$. We retain this notation when viewing $\mathbf{f}_{0}^{\dagger} \rtimes H$ and $\overline{H} \ltimes \overline{\mathbf{f}}$ as sub-Hopf algebras of $U_{q}(\mathfrak{g})_{0}^{\dagger}$.
\end{defn}

\begin{prop}
$U_{q}(\mathfrak{g})_{0}^{\dagger}$ is analytically graded by $\mathbb{Z}I \cong \Psi$.
\end{prop}

\begin{proof}
As in the proof of Proposition \ref{NAQGGraded} we have that $\mathbf{f}_{r}^{\text{an}}$, $\mathbf{f}_{0}^{\dagger}$ and $\mathbf{f}$ are $H'$-comodules, and if we give $H$ the trivial $H'$-coaction $H \cong k \hat{\otimes} H \xrightarrow{\eta_{H'} \otimes \text{Id}_{H}} H' \hat{\otimes} H$ then all of the morphisms involved in defining $U(\mathbf{f}_{0}^{\dagger},H,\mathbf{f})$ are $H'$-comodule homomorphisms.
\end{proof}

\subsection{Quantum groups as Drinfel'd doubles and braided monoidal representations}

\begin{lem}
There is a duality pairing
$$\langle -,- \rangle:(\overline{H} \ltimes \overline{\mathbf{f}}) \hat{\otimes} (\mathbf{f}^{\dagger}_{0} \rtimes H')^{\text{op}} \rightarrow k$$
given by the composition
$$
\begin{array}{rcccl}
H \hat{\otimes} \mathbf{f}^{\text{an}}_{\frac{s}{4}} \hat{\otimes} \mathbf{f}^{\text{an}}_{\frac{r}{4}} \hat{\otimes} H' &\overset{\text{Id} \otimes \text{Id} \otimes S \otimes\text{Id}}{\xrightarrow{\hspace*{1.5cm}}}& H \hat{\otimes} \mathbf{f}^{\text{an}}_{\frac{s}{4}} \hat{\otimes} \mathbf{f}^{\text{an}}_{\frac{r}{4}} \hat{\otimes} H' 
&\overset{\text{Id} \otimes \langle -,- \rangle \otimes\text{Id}}{\xrightarrow{\hspace*{1.5cm}}}& H \hat{\otimes} k \hat{\otimes} H' \\
&\overset{\text{Id} \otimes S}{\xrightarrow{\hspace*{1.5cm}}}& H \hat{\otimes} H' 
&\overset{\langle -,- \rangle}{\xrightarrow{\hspace*{1.5cm}}}&  k.
\end{array}
$$
\end{lem}

\begin{proof}
As with Lemma \ref{DualityPairingforQuantumDouble}, this follows from Proposition 34 of Section 6.3.1 of \cite{QGaTR}.
\end{proof}

\begin{defn}
We will denote by $D(\overline{H} \ltimes \overline{\mathbf{f}},\mathbf{f}^{\dagger}_{0} \rtimes H')$ the relative Drinfel'd double of $\overline{H} \ltimes \overline{\mathbf{f}}$ and $\mathbf{f}^{\dagger}_{0} \rtimes H'$ as constructed in Lemma \ref{Drinfel'dDoubleConstruction}.
\end{defn}

Recall the definition of crossed bimodules from Definition \ref{CrossedBimodules}.

\begin{lem}
\label{ModulesofQuantumDoubleA}
There is a fully faithful functor
$$_{\mathbf{f}^{\dagger}_{0} \rtimes H'}\text{Cross}^{\mathbf{f}^{\dagger}_{0} \rtimes H'} \rightarrow D(\overline{H} \ltimes \overline{\mathbf{f}},\mathbf{f}^{\dagger}_{0} \rtimes H')\text{-Mod}.$$
\end{lem}

\begin{proof}
This functor is constructed in Lemma \ref{ModulesofQuantumDouble}. Given a morphism ${M \xrightarrow{f} N}$ in $D(\overline{H} \ltimes \overline{\mathbf{f}},\mathbf{f}^{\dagger}_{0} \rtimes H')\text{-Mod}$ between objects in the image of $_{\mathbf{f}^{\dagger}_{0} \rtimes H'}\text{Cross}^{\mathbf{f}^{\dagger}_{0} \rtimes H'}$ then $f$ commutes with both the action of $H'$ and $\mathbf{f}$. Hence $f$ must preserve the locally Banach weight space decomposition, hence commutes with the coaction of $H'$. Also, since the bilinear pairing between the corresponding graded pieces of $\mathbf{f}$ and $\mathbf{f}^{\dagger}_{0}$ is nondegenrate, $f$ must also commute with the coaction of $\mathbf{f}^{\dagger}_{0}$. Hence $f$ is a morphism in $_{\mathbf{f}^{\dagger}_{0} \rtimes H'}\text{Cross}^{\mathbf{f}^{\dagger}_{0} \rtimes H'}$.
\end{proof}

\begin{prop}
\label{AnalyticQuantumGroupAsDrinfel'dDoubleA}
There is a strict epimorphism
$$D(\overline{H} \ltimes \overline{\mathbf{f}},\mathbf{f}^{\dagger}_{0} \rtimes H') \rightarrow U_{q}(\mathfrak{g})^{\dagger}_{0}$$
whose kernel is
$$\mathbf{f}^{\dagger}_{0} \hat{\otimes} \overline{\langle t_{i} \otimes 1 - 1 \otimes t_{i} \mid i \in I \rangle} \hat{\otimes} \mathbf{f} \hookrightarrow \mathbf{f}^{\dagger}_{0} \hat{\otimes} H' \hat{\otimes} H \hat{\otimes} \mathbf{f}.$$
\end{prop}

\begin{proof}
This follows as in the proof of Proposition \ref{NAAnalyticQuantumGroupAsDrinfel'dDouble}. Again, this morphism can be written as
$$\mathbf{f}^{\dagger}_{0} \hat{\otimes} H' \hat{\otimes} H \hat{\otimes} \mathbf{f} \xrightarrow{\text{Id} \otimes \mu_{H} \otimes \text{Id}} \mathbf{f}^{\dagger}_{0} \hat{\otimes} H \hat{\otimes} \mathbf{f}.$$
By Proposition \ref{ANicholsAlgebrasAreFlat}, $\mathbf{f}^{\dagger}_{0}$ is flat. Also, since $\mathbf{f}$ is a colimit of finite dimensional spaces it is also flat. The result then follows from the fact that $\overline{\langle t_{i} \otimes 1 - 1 \otimes t_{i} \mid i \in I \rangle}$ is the kernel of $\mu_{H}: H' \hat{\otimes}H \rightarrow H$.
\end{proof}

\begin{defn}
\label{SubcatOfCrossedBimodulesA}
Let us denote by $\mathcal{C}$ the full subcategory of $_{(\mathbf{f}^{\dagger}_{0} \rtimes H')}\text{Cross}^{(\mathbf{f}^{\dagger}_{0} \rtimes H')}$ consisting of IndBanach spaces $V$ equipped with both a left action and right coaction of $\mathbf{f}^{\dagger}_{0} \rtimes H'$, $\mu_{V}: (\mathbf{f}^{\dagger}_{0} \rtimes H') \hat{\otimes} V \rightarrow V$ and $\Delta_{V}: V \rightarrow V \hat{\otimes} (\mathbf{f}^{\dagger}_{0} \rtimes H')$, such that the diagram

\begin{center}
\begin{tikzpicture}[node distance=6cm, auto]
  \node (A) {$H' \hat{\otimes} V$};
  \node (B) [below=1cm of A] {$(\mathbf{f}^{\dagger}_{0} \rtimes H') \hat{\otimes} V$};
  \node (C) [right=1cm of A] {$(\overline{H} \ltimes \overline{\mathbf{f}}) \hat{\otimes} V$};
  \node (D) [below=1.05cm of C] {$V$};
  \draw[->] (A) to node {} (B);
  \draw[->] (C) to node {$\mu_{V}'$} (D);
  \draw[->] (A) to node {} (C);
  \draw[->] (B) to node {$\mu_{V}$} (D);
\end{tikzpicture}
\end{center}
commutes, where $\mu_{V}'$ is the action of $\overline{H} \ltimes \overline{\mathbf{f}}$ on $V$ induced by $\Delta_{V}$.
\end{defn}

\begin{lem}
There is a fully faithful functor $\mathcal{C} \rightarrow U_{q}(\mathfrak{g})_{0}^{\dagger}\text{-Mod}$.
\end{lem}

\begin{proof}
This follows from Lemma \ref{ModulesofQuantumDouble} and Proposition \ref{AnalyticQuantumGroupAsDrinfel'dDoubleA}.
\end{proof}

\begin{defn}
\label{IntegrableRepresentations}
Let us denote by $\mathcal{O}_{\Psi}$ the essential image of $\mathcal{C}$ in $U_{q}(\mathfrak{g})_{0}^{\dagger}\text{-Mod}$. By the previous Lemma this is precisely the full subcategory of $U_{q}(\mathfrak{g})_{0}^{\dagger}\text{-Mod}$ consisting of modules, $M$, such that the action of $H$ gives a module in $H\text{-Mod}_{\Psi}$ and the action of $\mathbf{f}$  is induced by a coaction of $\mathbf{f}_{0}^{\dagger}$ via their pairing.
\end{defn}

\begin{corollary}
The category $\mathcal{O}_{\Psi}$ is braided.
\end{corollary}

\begin{proof}
This follows from Lemma \ref{CrossedBimodulesBraided}.
\end{proof}

In time, the author hopes to study the representations in $\mathcal{O}_{\Psi}$ further and compute examples of the braiding. The hope is that this will produce interesting new braid group representations in which we might see some special analytic functions arising. For example, in \cite{PRftQDaQM}, Goncharov exhibits an automorphism of a Schwarz space using the quantum dilogarithm that satisfies a pentagon relation. This Schwarz space is equipped with an action of the algebra of regular functions on a quantised cluster variety, and this automorphism of the Schwarz space intertwines an automorphism of this algebra of regular functions. The action of these regular functions gives a natural action of the positive part of $U_{q}(\mathfrak{sl}_{2})$. It would be interesting to see whether we can use this to obtain a representation in $\mathcal{O}_{\Psi}$ for $\mathfrak{g}=\mathfrak{sl}_{2}$ whose braiding is related to the quantum dilogarithm and this work of Goncharov.

\end{document}